\documentclass[reqno,11pt]{amsart}
\usepackage{amsmath,amssymb,mathrsfs,amsthm,amsfonts}
\usepackage[inline]{enumitem} 
\usepackage[usenames,dvipsnames]{xcolor}
\usepackage{hyperref}
\usepackage{comment}
\hypersetup{%
	colorlinks=true, linkcolor=blue,
	citecolor=ForestGreen
}
\usepackage[paper=letterpaper,margin=1in]{geometry}

\newtheorem{theorem}{Theorem}[section]
\newtheorem{lemma}[theorem]{Lemma}
\newtheorem{proposition}[theorem]{Proposition}
\newtheorem{corollary}[theorem]{Corollary}
\theoremstyle{definition}
\newtheorem{definition}[theorem]{Definition}
\newtheorem{remark}[theorem]{Remark}
\numberwithin{equation}{section}
\allowdisplaybreaks
\usepackage{acronym}

\acrodef{KPZ}{Kardar--Parisi--Zhang}
\acrodef{SHE}{Stochastic Heat Equation}
\acrodef{LDP}{Large Deviation Principle}

\renewcommand{\Pr}{\mathbb{P}}	
\newcommand{\Ex}{\mathbb{E}}	
\renewcommand{\d}{\mathrm{d}}	
\newcommand{\ind}{\mathbf{1}}	
\newcommand{\sd}{\textcolor{black}}
\newcommand{\e}{\varepsilon}

\newcommand{\z}{\epsilon}

\newcommand{\mur}{\mathfrak{G}}

\newcommand{\N}{\mathbb{N}}
\newcommand{\R}{\mathbb{R}} 
\newcommand{\Z}{\mathbb{Z}} 

\newcommand{\g}{\mathfrak{g}}
\newcommand{\m}{\mathsf}

\newcommand{\h}{\mathfrak{h}}

\newcommand{\sh}{\mathfrak{g}}
\newcommand{\calA}{\mathcal{A}}

\newcommand{\calD}{\mathcal{D}}

\newcommand{\calH}{\mathcal{H}}

\newcommand{\calZ}{\mathcal{Z}}

\renewcommand{\bar}{\overline}

\usepackage{graphicx}

\title[LIL and Fractal Properties of the KPZ]{Law of Iterated Logarithms and Fractal Properties of the KPZ equation}

\author[S.\ Das]{Sayan Das}
\address{S.\ Das,
	Department of Mathematics, Columbia University,
	\newline\hphantom{\quad \ \ S. Das}
	2990 Broadway, New York, NY 10027, USA
	}
\email{sayan.das@columbia.edu}
\author[P.\ Ghosal]{Promit Ghosal}
\address{P.\ Ghosal,
	Department of Mathematics, Massachusetts Institute of Technology,
	\newline\hphantom{\quad \ \ P. Ghosal}
	77 Massachusetts Avenue, Cambridge, MA 02139, USA
}
\email{promit@mit.edu}

\begin{document}
\begin{abstract}
	  We consider the Cole-Hopf solution of the $(1+1)$-dimensional KPZ equation started from the narrow wedge initial condition. In this article, we ask how the peaks and valleys of the KPZ height function (centered by time/$24$) at any spatial point grow as time increases. Our first main result is about the law of iterated logarithms for the KPZ equation. As time variable $t$ goes to $\infty$, we show that the limsup of the KPZ height function with the scaling by $t^{1/3}(\log \log t)^{2/3}$ is almost surely equal to $(3/4\sqrt{2})^{2/3}$ whereas the liminf of the height function with the scaling by $t^{1/3}(\log \log t)^{1/3}$ is almost surely equal to $-6^{1/3}$. Our second main result concerns with the \emph{macroscopic} fractal properties of the KPZ equation. Under exponential transformation of the time variable, we show that the peaks of KPZ height function mutate from being monofractal  to multifractal, a property reminiscent of a similar phenomenon in Brownian motion \cite[Theorem~1.4]{KKX17}. 
	  
 The proofs of our main results hinge on the following three key tools: $(1)$ a \emph{multi-point composition law} of the KPZ equation which can be regarded as a generalization of the two point composition law from \cite[Proposition~2.9]{corwin2019kpz}, $(2)$ the Gibbsian line ensemble techniques from \cite{CH14,CH16,corwin2019kpz} and, $(3)$ the tail probabilities of the KPZ height function in short time and its spatio-temporal modulus of continuity. We advocate this last tool as one of our new and important contributions which might garner independent interest.      	     
\end{abstract}

\maketitle

\section{Introduction}

We study the Kardar-Parisi-Zhang (KPZ) equation, a stochastic PDE which is formally written
 \begin{align}\label{kpz}
			\partial_t\mathcal{H}=\tfrac12\partial_{xx}\mathcal{H} + \tfrac12(\partial_x \mathcal{H})^2 + \xi, \qquad \mathcal{H} := \mathcal{H}(t,x) \qquad (t,x)\in [0,\infty)\times\R.
			\end{align} 
 Here, $\xi=\xi(t,x)$ is the space time white noise. The KPZ equation was originally introduced in \cite{kardar1986dynamic} for studying the fluctuation of growing interfaces and since then, it has found links to many systems including directed polymers, last passage percolation, interacting particle systems, and random matrices via its connections to the \emph{KPZ universality class} (see \cite{ferrari2010random,quastel2011introduction,corwin2012kardar,quastel2015one}). 
  
  The KPZ equation, as given in \eqref{kpz}, is ill-posed as a stochastic PDE due to the presence of the nonlinear term $(\partial_x\calH)^2$. The physically relevant notion of solution for the KPZ equation is given by the \textit{Cole-Hopf solution} which is defined as
  \begin{align*}
  \mathcal{H}(t,x):= \log \mathcal{Z}(t,x)
  \end{align*}
  where $\mathcal{Z}(t,x)$ is the solution of the stochastic heat equation (SHE)
    \begin{align}\label{she}
  	\partial_t\mathcal{Z} = \tfrac12\partial_{xx}\mathcal{Z} + \xi\mathcal{Z}, 
  	\qquad \calZ:=\calZ(t,x).
  \end{align} 
Throughout this paper, we work with the fundamental solution $\mathcal{Z}^{\mathbf{nw}}(t,x)$ of \eqref{she} and the associated Cole-Hopf solution $\mathcal{H}^{\mathbf{nw}}(t,x):=\log\mathcal{Z}^{\mathbf{nw}}(t,x)$ which corresponds to the SHE being started from the delta initial measure, i.e., $\mathcal{Z}^{\mathbf{nw}}(0,x)=\delta_{x=0}$. For any positive $t>0$, $\mathcal{Z}^{\mathbf{nw}}(t,x)$ is strictly positive (see \cite{flores2014strict}) which makes the Cole-Hopf solution $\mathcal{H}^{\mathbf{nw}}(t,x)$ well-defined. The corresponding initial data of the KPZ equation is termed as the \emph{narrow wedge} initial data.


%
 
 
 The ubiquity of the SHE is discernible in many applications stretching from modeling the density of the particles diffusing through random environments \cite{molchanov1996reaction, khoshnevisan2014analysis, BC17,CG17} to the partition function of the continuum directed random polymer model \cite{alberts2014,CDR10,borodin2014macdonald}. The solution theory for the SHE is standard \cite{Walsh86, quastel2011introduction, Cor18}; based on It\^{o} integral theory or martingale problems. The mathematical theory of the KPZ equation however has unleashed new challenges in recent years. Most notably, the study of the KPZ equation can now be classified into three broad directions, namely, to understand how the KPZ equation approximates the interface fluctuation of the random growth models, to build a robust solution theory of the KPZ equation and to unveil fine properties and asymptotics of the solution of the KPZ equation. The Cole-Hopf solution of the KPZ equation coincides with the limits of certain growth processes \cite{BG97,CT17,CST,G17,CGST,Lin20}. The KPZ equation being a testing ground for the nonlinear stochastic PDEs, stirs up intense recent innovations in the theory of singular PDEs including regularity structures \cite{Hai13}, paracontrolled distributions \cite{GIP15, GP17}, energy solution \cite{GJ14} and renormalisation group \cite{Kup16} methods. In this paper, we seek to pursue the third direction, i.e., to unravel finer properties of the Cole-Hopf solution of the KPZ equation.

In this paper, we consider the following  $1:2:3$ scaled version of the KPZ height function:
\begin{align}\label{def:kpz-long-scale}
	\h_t(\alpha,x):=\frac{\calH^{\mathbf{nw}}(\alpha t,t^{2/3}x)+\frac{\alpha t}{24}}{t^{1/3}}.
\end{align}
where $t$ specifies the time scale and $\alpha$ measures the time judged on that scale, $x$ measures the space judged on $t^{2/3}$ scale. Although the presence of $t$ and $\alpha$ bears a stain of redundancy, the notation introduced in \eqref{def:kpz-long-scale} will be useful in stating and proving many of our results. For $\alpha=1$, we will often use the shorthand $\h_t(x):=\h_t(1,x)$ and $\h_t:=\h_t(0)$. We will call the stochastic process $\h_t$ indexed by the time parameter $t$ as the \emph{KPZ temporal process}. In a seminal work, \cite{ACQ11} showed that 
\begin{align*}%
\h_t \stackrel{d}{\to} 2^{-1/3}TW_{\mathrm{GUE}}, \quad \text{as }t\to \infty. 
\end{align*}
Here, $TW_{\mathrm{GUE}}$ is the Tracy-Widom GUE distribution. 
\sd{The \emph{KPZ scaling} of the fluctuation, space and time, i.e., the ratio of the corresponding scaling exponents being $1:2:3$ and the Tracy-Widom distribution as the limit of the fluctuations are the characteristics of the models in the KPZ universality class. Recently, \cite{QS20,Vir20} have announced proofs of the convergence of the spatial process $\h_t(x)$ (upto a parabola) to the universal limiting process of the KPZ universality class, namely the \emph{KPZ fixed point} as $t$ goes to $\infty$}. 

\smallskip


Our objects of study are the large \textit{peaks and valleys} of the KPZ temporal process as the KPZ equation approaches the KPZ fixed point. Such study for any generic one-dimensional stochastic process with a macroscopic limiting profile usually starts up with two questions: \emph{What are the scalings of the large peaks and valleys}? \emph{Do they converge to any limit under such scaling}? For a Brownian motion $\mathfrak{B}_t$, these questions are answered via the (Brownian) \emph{law of iterated logarithms} (LIL). Under the $\sqrt{t}$ scaling, the fluctuation of the Brownian motion $\mathfrak{B}_t$ has the Gaussian limit. At the onset of this macroscopic Gaussianity, the peaks and valleys of $\mathfrak{B}_t/\sqrt{t}$ under further scaling by $\sqrt{2\log\log t}$ stays in between $-1$ and $1$. The extra scaling by an iterated logarithmic factor $\sqrt{2\log\log t}$ inflicts the name `law of iterated logarithms'.  

 Our first main result which is stated as follows concerns with the law of iterated logarithms of the KPZ equation started from the narrow wedge initial data. 


\begin{theorem}\label{lil} With probability $1$, we have
	\begin{align*}
	\limsup_{t\to \infty}\frac{\h_t}{(\log\log t)^{2/3}}=\Big(\frac{3}{4\sqrt{2}}\Big)^{\frac{2}{3}}, \quad \mbox{and } \ \  \liminf_{t\to \infty}\frac{\h_t}{(\log\log t)^{1/3}}  = -6^{\frac{1}{3}}.
	\end{align*}
\end{theorem}

The above law of iterated logarithms reveals the scaling of the large peaks and valleys of $\mathfrak{h}_t$. As we may see, the scalings for limsup and liminf differ from each other. This naturally gives rise to the following two questions:

\medskip

\noindent $(1)$ \emph{Where are the scaling $(\log \log t)^{2/3}$ and $(\log \log t)^{1/3}$ coming from?}

\medskip

 The scaling of the large peaks and valleys for the KPZ height fluctuation are in fact orchestrated by the Tracy-Widom GUE distribution. This is in line with the LIL for the Brownian motion where the exponent $1/2$ of $(\log \log t)$ factor stems from the Gaussian tail decay of the limiting law. For the KPZ equation, the peaks and valleys have different scaling thanks to the distinct decay exponents of the upper and lower tail probabilities of the Tracy-Widom GUE. If $X$ is a Tracy-Widom GUE random variable, then, the probability of $X$ being higher than $s$ (i.e., upper tail probability) decays as $e^{-4s^{3/2}/3}$ and the probability of lower (i.e., lower tail probability) than $-s$ decays as $e^{-s^3/12}$. So, the upper tail decay exponent is $3/2$ which induce the scaling $(\log \log t)^{2/3}$ for the peaks of the KPZ fluctuation whereas the lower tail exponent being $3$ is the source for the scaling $(\log \log t)^{1/3}$ of the valleys.   
Interestingly, as one may observe, the values of the limsup and liminf in Theorem~\ref{lil} are seemingly connected to the constants $4/3$ and $1/12$ of the respective tail decays of the Tracy-Widom GUE distribution. This association is commensurate with the Brownian LIL and predicted in other works (discussed in Section~\ref{sec:PrevWorks}).     
 
\medskip 
 
\noindent $(2)$ \emph{How the LILs will vary with the initial data?}

\medskip

Based on the LIL for the narrow wedge solution, one may insinuate that the scaling of the peaks and valleys of the KPZ solution under other initial condition will be governed by the tail exponents of the limiting random variables. It follows from Theorem~1.1 and~1.4 of \cite{CG18b} that for a wide class of initial data, the upper tail exponents of the limiting r.v. of the KPZ equation under KPZ scaling is $3/2$ and the lower tail exponent is at least $3$. By drawing the analogy with the narrow wedge case, we conjecture that correct scaling of the peaks and valleys of the KPZ height fluctuation will be $(\log \log t)^{2/3}$ and $(\log \log t)^{1/3}$ respectively. Proving these claims is beyond the scope of the present paper since some of the major tools that we use are not available for the KPZ solution under other initial data. However, we hope to explore this direction in future works.

\medskip

Our next objective is to quantify how often the peaks and valleys of the KPZ fluctuation exceed a given level. This entails to studying the \emph{upper level sets} $\{t>t_0:\h_t\geq \gamma (\log \log t)^{2/3}\}$ and \emph{lower level sets} $\{t>t_0:\h_t\leq -\gamma (\log \log t)^{1/3}\}$ for different values of $\gamma$ where $\gamma>0$ is a tuning parameter and $t_0$ is an arbitrary constant. In particular, we study the \emph{macroscopic} fractal nature of the level sets. For brevity, we mainly focus on the study of the upper level sets \sd{in this paper}.

Fractal nature of the level sets of the KPZ equation is intimately connected to the moment growth of the SHE which is captured through the \emph{Lyapunov exponents}, i.e., the limit of $t^{-1}\mathbb{E}[(\mathcal{Z}^{\mathbf{nw}}(t,0))^{k}]$ as $t\to \infty$ for any integer $k$. The nonlinear nature of the Lyapunov exponents of the SHE (predicted by \emph{Kardar's formula} \cite{Kardar87}) suggests an abundance of the large peaks of the SHE. This is manifested through the existence of infinitely many scales for the peaks, a property often called as \emph{multifractality}. In contrast, the peaks of a scaled Brownian motion $\mathfrak{B}_t/\sqrt{t}$ only show a single scale as time $t$ increases to infinity. This latter property is named as \emph{monofractality}. In the following, we give a mathematical definition of these two different natures of the (macroscopic) fractality.


\begin{definition}[Mono- and Multifractality]\label{bd:Fractality}
Let $X$ be a stochastic process. Suppose there exists a non-random gauge function $g$ such that $g(r)$ increases to $\infty$ as $r\to \infty$ and 
$$ \limsup_{r\to \infty}\frac{X(r)}{g(r)} = 1\qquad \text{a.s.}$$
Fix a scalar $\gamma, t_0>0$. Define  
\begin{align*}
\Xi_{X,g}(\gamma) := \Big\{t>t_0: \frac{X(t)}{g(t)}> \gamma\Big\}.
\end{align*}
We denote the (Barlow-Taylor) \emph{macroscopic Hausdorff dimension} (see Definition~\ref{bfHausdorffCont}) of any Borel set $\mathfrak{F}$ by $\mathrm{Dim}_{\mathbb{H}}(\mathfrak{F})$. The tall peaks of $X$ is \emph{multifractal} in gauge $g$ when there exist infinitely many length scales $\gamma_1 >\gamma_2>\ldots > 0$ such that, with probability one, 
\begin{align*}
\mathrm{Dim}_{\mathbb{H}}\big(\Xi_{X,g}(\gamma_{i+1})\big)> \mathrm{Dim}_{\mathbb{H}}\big(\Xi_{X,g}(\gamma_{i})\big), \quad \forall i\geq 1.
\end{align*}
On the other hand, the peaks of $X$ with gauge function $g$ is \emph{monofractal} when 
\begin{align*}
\mathrm{Dim}_{\mathbb{H}}(\Xi_{X,g}(\gamma)) = \begin{cases}
\mathrm{Constant} & \gamma\leq \gamma_0\\
0 & \gamma> \gamma_0
\end{cases}
\end{align*}
for some $\gamma_0>0$.  
\end{definition}	

  By the law of iterated logarithms, the gauge function of a scaled Brownian motion $\mathfrak{B}_t/\sqrt{t}$ is dictated as $(2\log \log t)^{1/2}$. It follows from the works of \cite{KKX17,Str64,Mo58} that the Brownian motion with such choice of the gauge function is monofractal. However, the macroscopic nature of the peaks undergoes a transition under the exponential transformation of the time variable underpinning the Brownian motion. For instance, the Ornstein-Uhlenbeck process which is defined as $U(t):= \exp(-t/2)\mathfrak{B}_{e^{t}}$ for $t\in \mathbb{R}$ is multifractal in the gauge function $(2\log t)^{1/2}$.  
  
  \sd{Our second main result which is stated below shows that the KPZ tempral process is monofractal in the gauge function $(\tfrac{3}{4\sqrt{2}}\log \log t)^{2/3}$. Whereas under the exponential transformation of the time variable, the peaks of the KPZ temporal process exhibits multifractality.}

\begin{theorem}\label{thm:FracDim}
Consider the rescaled height function $\mathfrak{h}_t$ of the KPZ equation and the exponential time-changed process $\mur(t):= \mathfrak{h}_{e^{t}}$. Then, we have the following: $\mathfrak{h}_t$ is monofractal with positive probability in gauge function $(\log \log t)^{2/3}$, i.e., for every $t_0,\gamma>0$,
\begin{align}\label{eq:Monofractal}
\mathrm{Dim}_{\mathbb{H}}\Big\{t\geq e^e: \frac{\mathfrak{h}_t}{(\log \log t)^{2/3}}\geq \gamma\Big\} \stackrel{a.s.}{=} \begin{cases} 
1 & \text{when }\gamma\leq  \big(\frac{3}{4\sqrt{2}}\big)^{\frac{2}{3}}, \\
0 &  \text{when }\gamma> \big(\frac{3}{4\sqrt{2}}\big)^{\frac{2}{3}}.
\end{cases}
\end{align} 
 In contrast, $\mur(t)$ is multifractal in gauge function $(3/4\sqrt{2})^{2/3}(\log t)^{2/3}$. In fact,
\begin{align}\label{eq:Mult}
\mathrm{Dim}_{\mathbb{H}}\Big\{t\geq e^e: \frac{\mur(t)}{(3/4\sqrt{2})^{2/3}(\log t)^{2/3}}\geq \gamma\Big\} \stackrel{a.s.}{=} 1- \gamma^{3/2}, \quad \text{for }\gamma\in [0,1].
\end{align}
\end{theorem}

Note that \eqref{eq:Monofractal} show that the peaks of $\h_t$ are monofractal in the gauge function $(\log \log t)^{2/3}$. On the other hand, the multifractality of the peaks of $\mur(t)$ is clear from \eqref{eq:Mult} since 
$$\Xi_{\mur(t), (3/4\sqrt{2})^{2/3}(\log t)^{2/3}}(\gamma_2)\stackrel{a.s}{=}1- \gamma^{3/2}_2 <1- \gamma^{3/2}_1\stackrel{a.s}{=}\Xi_{\mur(t), (3/4\sqrt{2})^{2/3}(\log t)^{2/3}}(\gamma_1)$$ 
for $0\leq \gamma_1<\gamma_2\leq 1$. This raises the following three interesting questions. 

\medskip

\noindent $(1)$ \emph{What is the minimal speed up needed for the time variable to see transition from monofractality to multifractality of the peaks of the KPZ equation?}

\medskip

 We are indebted to Davar Khoshnevisan for asking this question. By carefully studying the outreach of our tools, we expect to see the appearance of  multifractality of the peaks under the transformation $t\mapsto \h_{e^{(\log t)^a}}$ for any $a>1$. Due to lack of detailed information on the correlation decay of the KPZ temporal process, we are unable to make precise prediction of the fractality under the transformation $t\mapsto \h_{t^a}$ for any $a>1$. Based on our intuition, we expect that the monofractality will still be survived under such transformations.  
 
 \medskip

\noindent $(2)$ \emph{Is there a similar notion of macroscopic fractality for the valleys? What are the macroscopic fractal properties of the valleys of the KPZ height function?}

\medskip

The fractal properties of the valleys can be studied using the lower level sets. For instance, if $X$ is a stochastic process such that $\liminf_{r\to -\infty} X(r)/f(r) =-1$ almost surely for some gauge function $f$, then, the multifractality and/or monofractality of the valleys of $X$ can be defined in the same way as in Definition~\ref{bd:Fractality} using the macroscopic Hausdorff dimension of the following lower level sets
$$\widehat{\Xi}_{X,f}(\gamma):= \Big\{t>t_0: \frac{X(t)}{f(t)}<- \gamma\Big\}.$$
For studying the valleys of $\h_t$, the natural choice of the gauge function is $(6\log \log t)^{1/3}$ as shown by Theorem~\ref{lil}. Using the tools of this paper, we expect that one can show monofractality of the valleys of $\h_t$ in the gauge function $(6\log \log t)^{1/3}$. Furthermore, drawing the analogy with \eqref{eq:Mult}, we also expect the following equality holds 
$$\mathrm{Dim}_{\mathbb{H}}\Big(\widehat{\Xi}_{\mur(t), (6\log t)^{1/3}}(\gamma)\Big) \stackrel{a.s}{=} 1- \gamma^3. $$ 
 While the fractal properties of the valleys seem extremely exciting, for brevity, we restrict ourselves only to exploring the peaks of the KPZ temporal process \sd{in this paper}. 
 
\medskip
 
\noindent $(3)$ \emph{What is expected about the peaks and valleys of the KPZ fixed point in the temporal direction?}

\medskip

It is believed that $\h_t(\alpha,x)$ weakly converges as a time-space process to the KPZ fixed point (started from the narrow wedge data) which has recently been constructed in \cite{MQR16} via its transition probability and simultaneously in \cite{DOV18} via the \emph{Airy sheet}. Very recently, \cite{QS20,Vir20} announced proofs of a special case of this conjecture, namely the weak convergence of the spatial process $x\mapsto (2\alpha^{-1})^{1/3}(\h_t(\alpha,x)+\frac{x^2}{2})$ to the Airy$_2$ process (introduced in \cite{PS02}) for any fixed $\alpha>0$. In light of this conjecture, we expect that the law of iterated logarithms of the KPZ fixed point in the temporal direction bear the same scaling as in Theorem~\ref{lil}. Moreover, the macroscopic nature of the peaks and valleys of the KPZ equation as revealed in the above discussion is expected to be reflective of the case for the KPZ fixed point. Although, our proof techniques which will be touched on in Section~\ref{sec:ProofIdea} are very much likely to be applicable for the KPZ fixed point, we defer from proving results analogous to Theorem~\ref{lil} and~\ref{thm:FracDim} for the KPZ fixed point.

\medskip

Proving the law of iterated logarithms and the fractal properties of the KPZ equation requires information on the growth of $\mathfrak{h}_{t_1}-\mathfrak{h}_{t_2}$ for $t_1>t_2>0$. When $t_1-t_2$ is large, \cite[Theorem~1.5]{corwin2019kpz} obtained upper and lower bounds on the tail probabilities of $\mathfrak{h}_{t_1}-\mathfrak{h}_{t_2}$. However, controlling the variations of the peaks in a smaller interval necessitates the study of the tail probabilities of the increments $\mathfrak{h}_{t_1}-\mathfrak{h}_{t_2}$ for $t_1-t_2$ small. One of the main obstructions for studying the increments of $\mathfrak{h}_t$ in a small interval is the lack of uniform tail bounds of $\mathfrak{h}_t$ for all small $t>0$. In the following two results, we seek to fill this gap. To state those results, we introduce the following notations: 
\begin{align*}
\sh_t := \frac{\mathcal{H}^{\mathbf{nw}}(t,0)+ \log\sqrt{2\pi t}}{(\pi t/4)^{1/4}}.
\end{align*}

The first result proves a uniform bound on the upper tail probabilities of $\mathfrak{h}_t$ for all small $t>0$. 

\begin{theorem}\label{uptail-upb} Fix $\e>0$. There exist $t_0=t_0(\e)>0$, $c=c(\e)>0$, and $s_0=s_0(\e)>0$ such that for all $t\le t_0$ and $s\ge s_0$, 
	\begin{align}\label{e:short-up}
	\Pr(\sh_t\ge s) \le \exp\Big(-\frac{cs^2}{1+\sqrt{1+st^{1/4-4\e}}}\Big).
	\end{align}
\end{theorem}

\begin{remark}
 Note that the right hand side of \eqref{e:short-up} decays like Gaussian tails, i.e., $\exp(-cs^2)$ for some constant $c>0$ as $t\downarrow 0$. This is embraced by the fact that $\sh_t$ weakly converges to a standard Gaussian distribution as $t$ approaches $0$ (shown in \cite[Proposition~1.8]{ACQ11}). On the other hand, for large $t$, the decay turns to $\exp(-cs^{3/2}t^{-1/8+2\e})$. The decay exponent $3/2$ accords with  the finite time upper tail exponent (see \cite[Theorem~1.10]{CG18b}) of the KPZ equation. 
\end{remark}

For the purpose of latter use, we will only require the following loose bound which is  free of the time variable and follows immediately from Theorem~\ref{uptail-upb}. 
\begin{corollary}\label{short-uptail} There exists $t_0>0$, $c>0$, and $s_0>0$ such that for all $t\le t_0$ and $s\ge s_0$, we have  
	\begin{align*}
	\Pr(\sh_t\ge s) \le \exp(-cs^{3/2}).
	\end{align*}

\end{corollary}

The second result shows an uniform bound on the lower tail probability of $\sh_t$ for all small $t>0$.

\begin{theorem}\label{short:lowertail} There exist constants $t_0 \in (0,2]$, $s_0>0$ and $c>0$ such that for all $t\le t_0$, $s\ge s_0$,  
	\begin{align}\label{eq:LowTailUP}
	\Pr(\sh_t\le -s)\le e^{-cs^2}.
	\end{align}
\end{theorem}

\begin{remark}
The decay exponent of the upper bound in \eqref{eq:LowTailUP} is consistent with the Gaussian limit of $\sh_t$ as $t$ goes down to zero. It is worthwhile to note that Theorem~\ref{short:lowertail} provides an upper bound to the lower tail probability which holds uniformly for all small $t>0$. This should be contrasted with the work of \cite[Theorem~1.1]{CG18a} which showed that the lower tail probability at finite time $t>0$ decays as $\exp(-ct^{1/3}s^{5/2})$ for some constant $c>0$. The interpolation between the exponents $2$ and $5/2$ as one gradually increases time $t$ from $0$ to a finite value is not covered in Theorem~\ref{short:lowertail}.
\end{remark}

Short time uniform tail bounds of Theorem~\ref{uptail-upb} and~\ref{short:lowertail} opens directions to a plethora of new results. One of such directions is the study of modulus of continuity of the time-space process $\mathfrak{h}_{t}(\alpha,x)$. Our next and final main result proves a super-exponential tail bound of the modulus of continuity of $\mathfrak{h}_{t}(\alpha,x)$. 

\begin{theorem}\label{thm:ModCont}
Fix $t_0>0$, $\varepsilon\in (0,\frac{1}{4})$ and any interval $[a,b]\subset \mathbb{R}$ and $[c,d]\subset \mathbb{R}_{> t_0}$. Define $\mathrm{Norm}:([a,b]\times [c,d])^2\to \mathbb{R}_{\geq 0}$ 
\begin{align}\label{eq:NormCgg}
\mathrm{Norm}(\alpha_1,x_1;\alpha_2,x_2) = |x_1-x_2|^{\frac{1}{2}}\Big(\log\frac{|b-a|}{|x_1-x_2|}\Big)^{2/3} + |\alpha_1-\alpha_2|^{\frac{1}{4}-\varepsilon}\Big(\log\frac{|d-c|}{|\alpha_1-\alpha_2|}\Big)^{2/3} 
\end{align}
and 
\begin{align}\label{eq:Cgg}
\mathcal{C}:= \sup_{\alpha_1\neq \alpha_2,x_1\neq x_2}\frac{1}{\mathrm{Norm}(\alpha_1,x_1;\alpha_2,x_2)} \big|\h_t(\alpha_1,x_1)+\frac{x^2_1}{2\alpha_1} - \h_t(\alpha_2,x_2)- \frac{x^2_2}{2\alpha_2}\big|.
\end{align}
Then there exist $s_0= s_0(t_0,|b-a|,|c-d|,\e)>0$ and $c=c(t_0,|b-a|,|c-d|,\e)>0$ such that for all $s\geq s_0$ and $t\geq t_0$, 
\begin{align}\label{eq:CIneq}
\mathbb{P}(\mathcal{C}>s)\leq e^{-cs^{3/2}}.
\end{align}
\end{theorem} 

\begin{remark}
It was known (due to \cite[Theorem~2.2]{BC95}) that the fundamental solution of the SHE (i.e., $\mathcal{Z}^{\mathbf{nw}}(t,x)$) as a time-space process is almost surely H\"older continuous with the spatial and temporal H\"older exponents being less than $1/2$ and $1/4$ respectively. This indicates H\"older continuity of $\mathcal{H}^{\mathbf{nw}}(t,x)$ with same spatial and temporal H\"older exponents as that of $\mathcal{Z}^{\mathbf{nw}}(t,x)$. Theorem~\ref{thm:ModCont} corroborates to this fact by giving tail bounds to the modulus of continuity.  
\end{remark}

\subsection{Proof ideas}\label{sec:ProofIdea}

   We start with discussing what makes our work hard to accomplish using other approaches. As a testing ground for non-linear SPDE's, the KPZ equation embraces a stack of new tools including regularity structures, paracontrolled distributions, energy solution method. Through its connection with the KPZ universality class, the KPZ equation is a paramount testament of a playing field for the techniques from integrable systems and random matrix theory. While these tools unveiled salient features of the KPZ equation in the past, many finer properties are still out of reach. One of the basic requirement for showing the law of iterated logarithms and the fractal nature of the KPZ level sets is to attain a delicate understanding of the modulus of continuity of the KPZ temporal process. This entails to knowing multi-point joint distribution  of the KPZ equation. While the seminal paper \cite{ACQ11} derived one point distribution of the narrow wedge solution of the KPZ equation, the exact formulas of more than one point does not seem to be on the horizon (see \cite{evgeni} for some recent progress in other positive temperature models). In \cite[Theorem~1.5]{corwin2019kpz}, the authors derived near-exponentially decaying bounds on the tail probabilities of the difference of the KPZ equation at two time points. Although these tail bounds were useful for setting forth the two time correlations of the KPZ equation, they fell short of achieving the modulus of continuity of the KPZ temporal process since those bounds only valid when the two time points are far apart.

 Our approach is mainly probabilistic while some of the key inputs bear an integrable origin. Two of such examples are the short time (upper) tail bounds of the KPZ equation (see Theorem~\ref{uptail-upb}) and the Gibbsian line ensemble. The short time upper tail will be derived using the integer moments of the SHE which has the recourse to some amenable contour integral formulas. On the other hand, while the Gibbsian line ensemble owes it inception to some integrable system, it has so far been fostered by the probabilistic ideas. One of the other key tools which we will procure in the due course of this paper is the short time lower tail bound (see Theorem~\ref{short:lowertail}) which in contrast to the upper tail has its chassis made of core probabilistic ideas like Talagrand's concentration inequality.   

Our first main tool is a multi-point composition law (see Proposition~\ref{ppn:MulpointComposition}) which generalizes the two-point composition law of \cite[Proposition~2.9]{corwin2019kpz}. In words, for any given set of time points $0<t_1<t_2<\ldots<t_k$, this law constructs $k$ independent random spatial profiles equivalent in law to the narrow wedge solution such that the KPZ temporal process at at $t_i$ is obtained by exponential convolution of one of such independent profiles and $\h_{t_{i-1}}(\cdot)$ for $i=2,\ldots ,k$.   

Our second main tool is the Gibbsian line ensemble. More precisely, we use a special Gibbsian line ensemble called the KPZ line ensemble introduced by \cite{CH14}. In short, KPZ line ensemble is a set of random curves whose lowest indexed curve has the same law as the narrow wedge solution of the KPZ equation. Furthermore, this set of random curves satisfies the \emph{Brownian Gibbs property} which ensures that the law of any fixed index curve in an interval only depends on the boundary value and can be described using the law of a Brownian bridge conditioned to have same boundary values, a connection elicited through a very explicit Radon-Nikodym derivative expression. As it was revealed in \cite{CH14}, the Brownian Gibbs property of the KPZ line ensemble imparts  \emph{stochastic monotonicity} on its lowest indexed curve, a property amenable to finding exquisite tail bounds of the spatial profile of the KPZ equation. Furthermore, we also enrich the arsenal of the Gibbsian line ensemble by introducing and exploring a \emph{short time KPZ line ensemble} (see Proposition~\ref{line-ensemble}) whose lowest indexed curve is the narrow wedge solution with \emph{short-time KPZ scaling}, i.e., the scaling exponent of the fluctuation, space and time follows the ratio $1:2:4$. In order to distinguish, we would refer the KPZ line ensemble whose lowest indexed curve is narrow wedge solution with the KPZ scaling as the \emph{long-time KPZ line ensemble}.   

Our third main tool is the short time upper and lower tail bounds (Theorem~\ref{uptail-upb} and~\ref{short:lowertail}) and the long time tail bounds of the KPZ equation from \cite{CG18a, CG18b} (summarized in Proposition~\ref{onepointlowtail}-\ref{ppn:UpperTail}). The short time upper tail is derived using the contour integral formulas of the moments of the SHE whereas the short time lower tail (uniform in time) is obtained via controlling the tail estimates of the partition function of random polymer model whose continuum limit solves the SHE. We also improve the bounds available for the long time upper tail of the KPZ equation (see Proposition~\ref{lem:RefUpTail}), a key input for showing the fractal nature of the upper level sets in Theorem~\ref{thm:FracDim}.  
  
Now we proceed to discuss how we use those tools to prove our results. The one point tail estimates of the KPZ equation (from Theorem~\ref{uptail-upb},~\ref{short:lowertail} and Proposition~\ref{onepointlowtail}-~\ref{ppn:UpperTail}) in conjugation with the tail bounds of the Brownian bridge fluctuations would allow us to derive delicate tail bounds of the spatial profile of the narrow wedge solution in finite intervals at the behest of the Brownian Gibbs property of the long and short time KPZ line ensembles. All these new tail estimates are detailed in Section~\ref{sec:spatialtail}.         
For any given $t_1>t_2$, the two point composition law relates $\h_{t_1}$ with the narrow wedge profile $\h_{t_2}(1,\cdot)$ via an exponential convolution with another independent random spatial process which will be denoted as $\h_{t_2\downarrow t_1}(\cdot)$ and has the same distribution as $\h_{t_1}((t_2-t_1)/t_1,\cdot)$. Mating of this convolution principle with the tail bounds of the KPZ spatial process from Section~\ref{sec:spatialtail} propagates the one point tail estimates to the tail bounds of the difference of the KPZ height functions at two time points. These ideas, inculcated in Proposition~\ref{shortdiff}-\ref{lngdiff} of Section~\ref{sec:diffprob}, will unfold to be a mainstay on which the proof of Theorem~\ref{thm:ModCont} rests with.     

By the Borel-Cantelli lemmas,  the law of iterated logarithm of Theorem~\ref{lil} can be recasted as showing that the infimum and supremum of the LIL adjusted temporal processes $\h_t/(\log \log t)^{1/3}$ and $\h_t/(\log \log t)^{2/3}$ respectively over the intervals $[\exp(e^n), \exp(e^{n+1})]$ cannot stay further away from $-6^{1/3}$ and $(3/4\sqrt{2})^{2/3}$ infinitely often. For proving these claims, one needs delicate tail bounds of the supremum and infimum of the KPZ temporal process which will be obtained in the following two ways. The first way uses the multi-point composition law of the KPZ equation (from Proposition~\ref{ppn:MulpointComposition}) to find upper bounds to the upper tail probability of the infimum and lower tail probability of the supremum. For any given set of time points, the multi-point composition law returns a set of independent random spatial profiles which are same in law with the narrow wedge solution. By Proposition~\ref{shortdiff}-\ref{lngdiff} of Section~\ref{sec:diffprob}, we give upper bound of the multi-point tail probabilities of the temporal process by the one-point tail probabilities of those independent spatial processes upto to some sharply decaying additive terms (Proposition~\ref{prop:IndProx}). The tail estimates of the multi-point distribution of the KPZ temporal process which are later turned into the tail probabilities of the infimum and supremum bring forth a new set of tools, unknown previously and hefty to obtain otherwise. The second way would find upper bound to the lower tail probability of the infimum and upper tail probability of the supremum of the KPZ temporal process using the tail bounds of the modulus of continuity claimed and proved in Proposition~\ref{temp-modulus}.

Much akin to the law of iterated logarithms, the proofs of mono- and multi-fractality of the KPZ equation heavily rely on the tail probabilities of the supremum and infimum of the KPZ temporal process in compact intervals. In addition, the mono-fractality result (\eqref{eq:Monofractal} of Theorem~\ref{thm:FracDim}) requires fast decoupling of the two-point upper tail probabilities of the KPZ equation. While such decoupling results are obtained for the Brownian motion in \cite[Lemma~3.5-3.6]{KKX17} without much ado, the situation for the KPZ equation is complicated and hinges on getting fine estimates of the one-point upper tail probability. Based on similar techniques as in \cite[Proposition~4.1]{CG18b}, Proposition~\ref{thm:MainTheorem} of Section~\ref{sec:App} provides such tail bounds which will be finally used in Proposition~\ref{asymp_ind} for showcasing the decoupling in the KPZ upper tail probabilities.         

Our approach of studying the peaks and valleys of the KPZ equation has the potential to generalize for other models in the KPZ universality class. As it was mentioned earlier, our approach stands on the shoulders of three main components: multi-point composition law, Gibbsian line ensemble and one-point tail probabilities. For the zero temperature models like the last passage percolation model, Airy process and many more, the analogues of the multi-point composition law are easy to obtain and stated in terms of the \emph{maximum convolution} instead of the exponential convolution. Gibbsian line ensemble approach was first introduced by \cite{CH14} for studying the Airy line ensemble and then, latter been applied in numerous zero temperature models. Furthermore, precise one-point tail estimates are available for many zero temperature model including the KPZ fixed point. Some of these technical appliances are also available for few positive temperature models such as the asymmetric simple exclusion process (ASEP), stochastic six vertex model, strictly weak lattice polymer model etc.  With the aid of the above three proof components, the revelation of the landscape of the aforementioned models bears immense possibilities which we hope to explore in future works.

\subsection{Previous works}\label{sec:PrevWorks}

Studying \emph{macroscopic landscapes} of stochastic processes is one of the most compelling research directions in probability theory. Starting from the middle of the previous century to the present time, Brownian motion serves as a fertile ground for doing alluring predictions on the landscape of the models in the Gaussian universality class and demonstrating those with lots of success. One of the main goals of this work is to showcase the KPZ equation as a representative of the models in the KPZ universality class when it comes to explaining the macroscopic landscape of its members under the KPZ scaling. Below, we review some of the previous works on the LIL and fractal properties of the models in the KPZ universality with the aim of comparing and contrasting those with our results.

 Random matrix theory is intimately connected with the models of the KPZ universality class. In fact, the Tracy-Widom GUE distribution which became one of the characteristics of the fixed points of the universality class was born out \cite{Tracy94} as a by-product of a random matrix model. To be more precise, the limiting distribution of the largest eigenvalue $\lambda^{\mathrm{GUE}}_n$ of an $n\times n$ Gaussian unitary ensemble under centering by $\sqrt{2n}$ and scaling by $n^{1/6}$ is essentially known as the Tracy-Widom GUE distribution. One may also regard $\lambda^{\mathrm{GUE}}_n$ as the $n$-th element of the GUE minor process. From this point of view, it was an interesting open question to study the law of fractional logarithm of $\lambda^{\mathrm{GUE}}_n$ which was finally solved by \cite{paquette2017extremal}. The authors found the value of the limsup of $(\lambda^{\mathrm{GUE}}_n-\sqrt{2n})/\sqrt{2}n^{1/6}$ under a  normalization by $(\log n)^{2/3}$ when $n$ goes to $\infty$. The authors had shown that the value of the limsup is almost surely equal to $(1/4)^{2/3}$. On the other hand, \cite{paquette2017extremal} had also shown that the liminf of $(\lambda^{\mathrm{GUE}}_n-\sqrt{2n})/\sqrt{2}n^{1/6}$ under a normalization by $(\log n)^{1/3}$ is almost surely finite. They had conjectured that the liminf is  almost surely equal to $-4^{1/3}$. To the best of our knowledge, the macroscopic Hausdorff dimensions of the level sets of  $\lambda^{\mathrm{GUE}}_n$ are not known yet. Drawing the analogy with the KPZ equation, we conjecture that the peaks and valleys of $\lambda^{\mathrm{GUE}}_n$ are multifractal in the gauge functions $(\log n)^{2/3}$ and $(\log n)^{1/3}$ respectively.
 
  Last passage percolation (LPP) is one of the widely studied models in the KPZ universality class. Due to the presence of endearing geometric properties, the study of the LPP model fueled lots of interests in the recent times. \cite{ledoux2018law} had initiated the study on the laws of iterated logarithms in the case of integrable LPP models. In \cite{ledoux2018law}, the author had considered the LPP model in $\mathbb{Z}^2_{\geq 0}$ lattice where the weights of the lattice sites are independent exponential or, geometric random variables. It was shown in \cite{ledoux2018law} that the limsup of point to point last passage percolation time from $(0,0)$ to $(n,n)$ (centered by $4n$ and scaled by $(2^4n)^{1/3}(\log \log n)^{2/3}$) is almost surely bounded between $\alpha_{\mathrm{sup}}$ and $(3/4)^{2/3}$ for some $0<\alpha_{\mathrm{sup}}\leq (3/4)^{2/3}$. In fact, \cite{ledoux2018law} had conjectured that $\alpha_{\mathrm{sup}}$ is equal to $(3/4)^{2/3}$. \cite{ledoux2018law} had also investigated the liminf of the LPP model. It was shown that the LPP time between $(0,0)$ and $(n,n)$ (centered by $4n$ and scaled by $(2^4n)^{1/3}(\log \log n)^{1/3}$) is almost surely lower bounded by some constant. 
  Recently, \cite{basu2019lower} have shown that the value of liminf is almost surely equal to a constant. However, not much is known about the exact value.

Fractal properties of the putative distributional limit of the models in the KPZ universality class, namely the KPZ fixed point has been investigated in few of the latest works. Recently, \cite{DOV18} gave a probabilistic construction of the KPZ fixed point as a distributional limit of the point-to-point Brownian last passage percolation model. The limiting space-time process which they named as the \emph{directed landscape} led to a flurry of new discoveries.  The study of the fractal geometry of the directed landscape has lately been initiated by \cite{BGH19,BGH20} who considered the problem of fractal dimension of some exceptional points along the spatial direction. In spite of the recent developments, the fractal nature of the space-time process of the directed landscape is still not fully understood. We hope that our results on fractality of the KPZ equation would shed some light for such study in future.

In the last decade, fractal properties of stochastic partial differential equations (SPDE) became an active area of research. The main focus of a vast majority of those works resided on the study of the large peaks of the SPDEs with multiplicative noise \cite{GM90,CM94,BC95,HHNT,FK09,CJKS,CD15,BC16, Ch17,CHN19}. The growth of the large peaks of the SPDEs is attested by the intermittency property which is the center of attention in the field of the research of complex multiscale system for last five-six decades. See introduction of \cite{BC95} and \cite{CM94, khoshnevisan2014analysis} for a detailed discussion. Recently, \cite{KKX17} investigated the fractal properties of the stochastic heat equation started from the constant initial data at the onset of intermittency and established the multifractal nature of the spatial process. Denote the solution of the SHE started from the constant initial data (i.e., $\mathcal{Z}^{\mathrm{flat}}(0,x)=1$ for all $x\in \mathbb{R}$) by $\mathcal{Z}^{\mathrm{flat}}(t,x)$. Drawing on an earlier result of \cite{Chen15} which showed a fractional law of logarithm -
\begin{align}\label{spat}
\limsup_{x\to \infty} \frac{\log\mathcal{Z}^{\mathrm{flat}}(t,x)}{t^{1/3}(\frac{3}{4\sqrt2}\log_{+}x)^{2/3}} = 1 \quad \text{a.s.},
\end{align}
Theorem 1.2 of \cite{KKX17} established the multifractal nature of the  spatial process $\log\mathcal{Z}^{\mathrm{flat}}(t,\cdot)$ for any fixed $t>0$. The results of \cite{KKX17} is complemented by the study of the spatio-temporal fractal properties by \cite{KKX18} which showed that there are infinitely many different stretch scale (in the spatial direction) and time scale such that for any given stretch and time scale, the peaks of the spatio-temporal process of the stochastic heat equation attain non-trivial macroscopic Hausdorff dimensions. The idea of peaks of the stochastic heat equation forming complex multiscale system were also echoed in \cite{Zim00,GD05a,GD05b}. \sd{However, the macroscopic behavior of the KPZ temporal process as considered in this paper shows a different nature due to its slow decay of correlations in comparison to the KPZ equation along the spatial direction. For instance, our first result, Theorem \ref{lil} exhibits LIL for the KPZ temporal process as opposed to the fractional law of logarithm satisfied by the KPZ spatial process demonstrated in \eqref{spat}.} In the same spirit, our second result (Theorem~\ref{thm:FracDim}) which is reminiscent of a similar result in \cite[Theorem~1.4]{KKX17} for $1$-dimensional Brownian motion shows that the peaks of the KPZ temporal process exhibit a monofractal (see Definition~\ref{bd:Fractality}) nature as time $t$ goes to $\infty$. This is in contrast to the multifractal nature of the spatial process as shown in \cite{KKX17}. Nevertheless, Theorem \ref{thm:FracDim} shows that the crossover to the multifractality of the KPZ temporal process happens under exponential transformation of the time variable. While the complete understanding of the spatio-temporal landscape of the KPZ equation is far-off
to our present reach, we hope that our results will ignite further interests along this direction.


 We end this section with a review on the tail probabilities of the KPZ equation, one of the key tools of this paper. Study of the KPZ tail probabilities had been taken up in many works \cite{MN08,corwin2013crossover,flores2014strict} in the past. One of the recent major advances has been achieved in \cite{CG18a} which proved tight bounds to the lower tail probability of the KPZ equation started from the narrow wedge initial data. This sowed the seeds of a series of works \cite{corwin2018coulomb,tsai2018exact,KL19,CC19,Z19, CCR20} which studied in details the lower tail large deviation of the KPZ equation as time goes to $\infty$. The upper tail probabilities of the KPZ equation has been recently investigated by \cite{CG18b}. The same paper also initiated the study of the tail probabilities under general initial data. The upper tail large deviation was later found in \cite{DT19} for narrow wedge initial data and in \cite{GL20} for general initial data. In spite of these recent advances, not much were known about the evolution of the tail probabilities of the KPZ as time $t$ goes to $0$. In a very recent work, \cite{LT20} showed the large deviation of the KPZ equation as $t$ tends to $0$. However, this does not shed much light on the  uniform tail estimates of the KPZ height function starting from time equal to $0$ to a finite value. Such uniform estimates which were reported in Theorem~\ref{uptail-upb} and~\ref{short:lowertail} will be instrumental in obtaining our other main results Theorem~\ref{lil},~\ref{thm:FracDim} and~\ref{thm:ModCont}.      

\medskip

\noindent\textbf{Outline.} Section~\ref{sec:tools} will introduce the basic frameworks of the KPZ line ensemble and the Barlow-Taylor macroscopic fractal theory. It will also introduce other useful tools including multipoint composition law, one-point tail probabilities of the KPZ equation, tail probabilities of the supremum and infimum of the KPZ spatial process. Section~\ref{sec:onepoint} will prove Theorem~\ref{uptail-upb} and~\ref{short:lowertail}. This will be followed by Section~\ref{sec:spatialtail} where we derive delicate tail bounds of the KPZ spatial process for finite and short time. Section~\ref{sec:diffprob} will study the temporal modulus of continuity of the KPZ equation and use it to prove Theorem~\ref{thm:ModCont}. Based on the tools from Section~\ref{sec:tools}-\ref{sec:diffprob}, the law of iterated logarithms of Theorem~\ref{lil} will be proved in Section~\ref{sec:lil}. The proof of the mono- and multifractality results of the KPZ equation from Theorem~\ref{thm:FracDim} will be given in Section~\ref{sec:MuonoMult}. This last section will use an improved KPZ upper tail probability estimate which is proved in Proposition~\ref{lem:RefUpTail} of Appendix~\ref{sec:App}. 

\medskip

\noindent\textbf{Acknowledgements.} We are grateful to Ivan Corwin for numerous stimulating discussions, encouragement, and giving us valuable inputs in an earlier draft of the paper. We express our gratitude towards Benjamin Landon whose suggestions helped us to improve Theorem~\ref{thm:FracDim}. We thank Davar Khoshnevisan for asking interesting questions on our work. We thank Shalin Parekh and Li-Cheng Tsai for helpful conversations and discussions. We thank Yier Lin for reading the manuscript carefully and making useful comments which improve the presentation.

SD's research was partially supported by Ivan Corwin's NSF grant DMS-1811143 as well as the Fernholz Foundation's ``Summer Minerva Fellows'' program.

\section{Basic Framework and Tools} \label{sec:tools}

In this section, we will review mainly three topics which are required for our subsequent analysis. One of the main topics of this section is the KPZ line ensemble and its Brownian Gibbs property. The KPZ line ensemble is a set of random curves whose lowest indexed curve is same in distribution with the narrow wedge solution of the KPZ equation. The Brownian Gibbs property of the KPZ line ensemble induces stochastic monotonicity of the spatial profile the KPZ equation, one of the major tools in our analysis. Lemma~\ref{Coupling1} of  Section~\ref{sec:KPZLE} will precisely state such monotonicity result. In a similar way as in \cite{CH16}, we will introduce a short-time version of the KPZ line ensemble which would play a key role in later sections to find the temporal modulus of continuity of the KPZ equation. 

Our second main topic of this section is the Barlow-Taylor theory of macroscopic fractal properties of a stochastic process. In light of the expositions in \cite{KKX17,BT1,BT2}, the notions of Barlow-Taylor \emph{Hausdorff content} and \emph{dimension} of any Borel set will be recalled. Some of the basic properties of the Barlow-Taylor Hausdorff dimension are presented in Proposition~\ref{ppn:EssHaus},~\ref{ppn:Frostman} and~\ref{thick} of Section~\ref{sec:BT}.

Lastly, we recall some of the known facts about the KPZ equation including its \emph{multipoint composition law} and the tail estimates of its one point distribution in Section~\ref{sec:kpzeq}. 

\subsection{KPZ line ensemble}\label{sec:KPZLE}
Describing the KPZ line ensemble and its Brownian Gibbs property requires some notations which are introduced as follows. 
\begin{definition}[Brownian Gibbs line ensembles; Definitions 2.1 and 2.2 of \cite{CH16}]\label{LineEns}
 Fix intervals $\Sigma\subset \Z$ and $\Lambda\subset\R$. Let $X$ be the set of continuous functions $f:\Sigma \times \Lambda\to \R$, endowed with the topology of uniform convergence on compact subsets;. Denote the $\sigma$-field generated by $X$ by $\mathcal{C}$. 

A $(\Sigma\times \Lambda)$-\emph{indexed line ensemble} $\mathcal{L}$ is a random variable $\mathcal{L}$ in a probability space $(\Omega,\mathcal{B},\mathbb{P})$ taking values in $X$ such that $\mathcal{L}$ is measurable with respect to $(\mathcal{B}, \mathcal{C})$. In simple words, $\mathcal{L}$ is a set of random continuous curves indexed by $\Sigma$ where each of those curves maps $\Lambda$ to $\R$. An element of $\Sigma$ is a curve index, and we will write $\mathcal{L}_k(x)$ instead of $\mathcal{L}(k,x)$ for $k\in \Sigma$ and $x\in \Lambda$; we will write $\mathcal{L}_k$ for the entire index $k$ curve.


For any two integers  $ k_1 <k_2$, two vectors of reals $\vec{x}, \vec{y} \in \R^{k_1-k_2+1}$, and an interval $(a,b)$, we say that a $\{k_1, \ldots , k_2\}\times (a,b)$-indexed line ensemble is a \emph{Brownian bridge line ensemble with entrance data $\vec{x}$ and  exit data $\vec{y}$} if its law, which we denote by $\mathbb{P}^{k_1, k_2, (a,b), \vec{x}, \vec{y}}_{\mathrm{free}}$, is equal to that of $k_2- k_1+1$ independent Brownian bridges starting at values~$\vec{x}$ at $a$  and ending at values $\vec{y}$ at $b$. We use the notation $\mathbb{E}^{k_1, k_2, (a,b), \vec{x}, \vec{y}}_{\mathrm{free}}$ to denote the expectation with respect to the probability measure $\mathbb{P}^{k_1, k_2, (a,b), \vec{x}, \vec{y}}_{\mathrm{free}}$. When $k_1=k_2=1$, we write $\mathbb{P}^{(a,b), \vec{x}, \vec{y}}_{\mathrm{free}}$. One may think $a$ and $b$ as times and $\vec{x}$ and $\vec{y}$ as starting and ending locations for the Brownian bridges. 

Suppose we are given a continuous function  $\mathbf{H}:\R\to [0,\infty)$ which we will call a \emph{Hamiltonian}. We will consider the following two kinds of Hamiltonians:
\begin{align}
\mathbf{H}^{\mathrm{long}}_t(x)= e^{t^{1/3}x} \, \, \, \,  \quad \text{and, }\quad \,\, \, \mathbf{H}^{\mathrm{short}}_t(x)= e^{(\pi t/4)^{1/4}x} \, \, \, \, \textrm{for given $t > 0$} \,\label{eq:Ht}
\end{align}

For any given Hamiltonian $\mathbf{H}$ and two measurable functions $f,g:(a,b)\to \R$,  we define a $\{k_1, \ldots , k_2\}\times (a,b)$-{\em indexed $\mathbf{H}$-Brownian bridge line ensemble with entrance data $\vec{x}$, exit data $\vec{y}$ and boundary data $(f,g)$} to be a collection of random curves  $\mathcal{L}_{k_1}, \ldots , \mathcal{L}_{k_2}: (a,b)\to \R$ whose law will be denoted by $\mathbb{P}^{k_1, k_2, (a,b), \vec{x}, \vec{y}, f,g}_{\mathbf{H}}$ and is specified  by the Radon-Nikodym derivative
\begin{align*}
\frac{d\mathbb{P}^{k_1, k_2,(a,b), \vec{x}, \vec{y}, f, g}_{\mathbf{H}}}{d\mathbb{P}^{k_1, k_2, (a,b), \vec{x}, \vec{y}}_{\mathrm{free}}}(\mathcal{L}_{k_1}, \ldots , \mathcal{L}_{k_2}) &= \frac{W^{k_1, k_2, (a,b), \vec{x}, \vec{y}, f,g}_{\mathbf{H}}(\mathcal{L}_{k_1}, \ldots , \mathcal{L}_{k_2})}{Z^{k_1, k_2, (a,b), \vec{x}, \vec{y}, f,g}_{\mathbf{H}}} \, ,\\
W^{k_1, k_2, (a, b), \vec{x}, \vec{y}, f,g}_{\mathbf{H}}(\mathcal{L}_{k_1}, \ldots , \mathcal{L}_{k_2})  &= \exp\bigg\{- \sum_{i=k_1}^{k_2+1} \int \mathbf{H}\big(\mathcal{L}_{i}(x)- \mathcal{L}_{i-1}(x)\big) dx \bigg\} \, .
\end{align*}
where in the right-hand side of the preceding display, we use the convention that $\mathcal{L}_{k_1-1}$ is equal to $f$, or to   $+\infty$ if $k_1-1\notin \Sigma$; and of $\mathcal{L}_{k_2+1}$ is equal to $g$, or to $-\infty$ if $k_2+1\notin \Sigma$. Here, $Z^{k_1, k_2, (a,b), \vec{x}, \vec{y}, f, g}_{\mathbf{H}}$ is the normalizing constant which produces a probability measure.

A $(\Sigma\times\Lambda)$-indexed line ensemble $\mathcal{L}$ enjoys the \emph{$\mathbf{H}$-Brownian Gibbs property} if, for all $K = \{k_1,\ldots, k_2\}\subset \Sigma$ and $(a,b)\subset \Lambda$, the following distributional equality holds:
\begin{align*}
\mathrm{Law}\Big(\mathcal{L}_{K\times (a,b)} \text{ conditioned on }\mathcal{L}_{\Sigma\times \Lambda \backslash K \times (a,b)}\Big) = \mathbb{P}^{k_1,k_2, \vec{x}, \vec{y}, f, g}_{\mathbf{H}} \, ,
\end{align*}
where $\vec{x}= (\mathcal{L}_{k_1}(a), \ldots , \mathcal{L}_{k_2}(a))$, $\vec{y} =(\mathcal{L}_{k_1}(b),\ldots , \mathcal{L}_{k_2}(b))$, and where again $f = \mathcal{L}_{k_1-1}$ (or $+\infty$ if $k_1-1\notin \Sigma$) and $g = \mathcal{L}_{k_2+1}$ (or $-\infty$ if $k_2+1\notin \Sigma$). That is, the influence of the complementary part on the ensemble's restriction to $K\times (a,b)$ comes only through the boundary data, i.e., the starting and ending points and the neighbouring curves.

 Like as for Markov processes, there is a \emph{strong} version of the $\mathbf{H}$-Brownian Gibbs property which is valid with respect to \emph{stopping domains} which we now describe. For a given line ensemble $\mathcal{L}$, let $\mathfrak{F}_{\textrm{ext}}\big(K\times (a,b)\big)$ be the $\sigma$-field generated by the curves $K\times (a,b)$. A pair $(\mathfrak{a},\mathfrak{b})$ of random variables is called a $K$-stopping domain if
$\big\{\mathfrak{a} \leq a, \mathfrak{b}\geq b\big\} \in \mathfrak{F}_{\textrm{ext}}\big(K\times (a,b)\big)$.
 Denote the set of continuous $K$-indexed functions $(f_{k_1},\ldots, f_{k_2})$, each from $(a,b)\to \R$ by $C^{K}(a,b)$ and furthermore, write
$$
C^K := \Big\{ (a,b,f_{k_1},\ldots, f_{k_2}):a<b \textrm{ and } (f_{k_1},\ldots, f_{k_2})\in C^K(a,b)\Big\}\, .
$$
Let $\mathcal{B}(C^K)$ be the set of all Borel measurable functions from $C^K$ to $\R$. We will say a $K$-stopping domain $(\mathfrak{a},\mathfrak{b})$ satisfies the \emph{strong $\mathbf{H}$-Brownian Gibbs property} if, for all $F\in \mathcal{B}(C^K)$, $\mathbb{P}$-almost surely
$$
\mathbb{E} \Big[ F\big(\mathfrak{a},\mathfrak{b},\mathcal{L}\big\vert_{K\times (\mathfrak{a},\mathfrak{b})}\big) \Big\vert \mathfrak{F}_{\textrm{ext}}\big(K\times (a,b)\big)\Big] = \mathbb{E}^{k_1,k_2, (\ell,r),\vec{x}, \vec{y}, f, g}_{\mathbf{H}} \Big[F\big(\ell,r,\mathcal{L}_{k_1},\ldots, \mathcal{L}_{k_2}\big)\Big] \, ,
$$
where, on the right-hand side, $\ell = \mathfrak{a}$, $r= \mathfrak{b}$, $\vec{x} =(\mathcal{L}_i(\mathfrak{a}))_{i\in K}$, $\vec{y} =(\mathcal{L}_i(\mathfrak{b}))_{i\in K}$, $f = \mathcal{L}_{k_1-1}$ (or $+\infty$ if $k_2+1\notin \Sigma$), $g=\mathcal{L}_{k_2+1}$ (or $-\infty$ if $k_2+1\notin \Sigma$), and the curves $\mathcal{ L}_{k_1},\ldots, \mathcal{L}_{k_2}$ have law $\mathbb{P}^{k_1,k_2, (\ell,r), \vec{x}, \vec{y}, f, g}_{\mathbf{H}}$.
\end{definition}

The following lemma demonstrates a sufficient condition under which the strong H-Brownian
Gibbs property holds. 

\begin{lemma}[Lemma 2.5 of \cite{CH16}]\label{lem:strongBGP}
Any line ensemble which enjoys the $\mathbf{H}$-Brownian Gibbs property also enjoys the strong $\mathbf{H}$-Brownian Gibbs property.
\end{lemma}

Line ensembles with the $\mathbf{H}$-Brownian Gibbs property  benefit from certain  stochastic monotonicities of the underlying measures. The following definition formally defines such monotonicity of measures. Afterwards, we state a proposition showing that two line ensembles
with the same index set can be coupled in such a way that if the boundary conditions of one
ensemble dominates the other, then same is true for laws of the restricted curves.

\begin{definition}[Domination of measure]\label{DomM}
Let $\mathcal{L}_1$ and $\mathcal{L}_2$ be two  $(\Sigma\times \Lambda)$-indexed line ensembles with respective laws $\mathbb{P}_1$ and $\mathbb{P}_2$. We say that $\mathbb{P}_{1}$ dominates $\mathbb{P}_{2}$ if there exists a coupling of  $\mathcal{L}_1$ and~$\mathcal{L}_2$ such that $\mathcal{L}^{j}_{1}(x)\geq \mathcal{L}^{j}_{2}(x)$ for all $j\in \Sigma$ and $x\in \Lambda$.
\end{definition}


\begin{lemma}[Stochastic monotonicity: Lemmas~$2.6$ and $2.7$ of \cite{CH16}]\label{Coupling1}
Fix finite intervals $K\subset \Sigma $ and $(a,b)\subset \Lambda$; and, for $i\in \{1,2\}$, vectors $\vec{x}_{i} = \big( x_i^{(k)}: k \in K \big)$ and  $\vec{y}_{i} = \big( y_i^{(k)}: k \in K \big)$ in $\R^K$ that satisfy $x^{(k)}_{2}\leq x^{(k)}_{1}$ and $y^{(k)}_{2}\leq y^{(k)}_{1}$ for $k\in K$; as well as measurable functions $f_i: (a,b)\to\R \cup \{+\infty\} $ and $g_i: (a,b) \to \R \cup \{-\infty\}$ such that $f_2(s)\leq f_1(s)$ and  $g_2(s)\leq  g_1(s)$ for $s\in (a,b)$. For $i \in \{1,2\}$, let
$\Pr_i$ denote the law $\mathbb{P}^{k_1, k_2, (a,b), \vec{x}_i, \vec{y}_i, f_i,g_i}_{\mathbf{H}}$, so that a $\Pr_i$-distributed random variable
$\mathcal{L}_i= \{\mathcal{L}_i^{k}(s)\}_{k\in K, s\in(a,b)}$ is a $K\times (a, b)$-indexed line ensemble.
If $\mathbf{H}:[0,\infty)\to \R$ is convex, then $\mathbb{P}_{1}$ dominates $\mathbb{P}_{2}$ -- that is, a common probability space $(\Omega,\mathcal{B},\mathbb{P})$ may be constructed on which the two measures are supported such that, almost surely, $\mathcal{L}_{1}^{k}(s)\geq \mathcal{L}_{2}^{k}(s)$ for $k\in K$ and $s\in (a,b)$.
\end{lemma}

Recall that the Hamiltonians $\mathbf{H}^{\mathrm{long}}_t(x)$ and $\mathbf{H}^{\mathrm{short}}_t(x)$ in \eqref{eq:Ht} are convex. Thus, Lemma~\ref{Coupling1} applies to any $\mathbf{H}^{\mathrm{long}}_t$ or, $\mathbf{H}^{\mathrm{short}}_t$-Brownian Gibbs line ensemble.

The following proposition recalls the unscaled and scaled KPZ line ensemble constructed in~\cite{CH16} which satisfies $\mathbf{H}^{\mathrm{long}}_t$-Brownian Gibbs property and introduces the short time KPZ line ensemble which exhibits $\mathbf{H}^{\mathrm{short}}_t$-Brownian Gibbs property.

\begin{lemma} \label{line-ensemble}
	Let $t>0$. There exists an $\N\times \R$-indexed line ensemble $\mathcal{H}_t =\{\mathcal{H}^{(n)}_{t}(x)\}_{n\in \N, x\in \R}$ such that:
	\begin{enumerate}
		\item the lowest indexed curve $\mathcal{H}^{(1)}_{t}(x)$ is equal in distribution (as a process in $x$) to the Cole-Hopf solution $\mathcal{H}^{\mathbf{nw}}(t,x)$ of KPZ started from the narrow wedge initial data;
		\item $\mathcal{H}_t$ satisfies the $\mathbf{H}^{\mathrm{long}}_{1}$-Brownian Gibbs property;
		\item the scaled KPZ line ensemble $\{\mathfrak{h}^{(n)}_t(x)\}_{n\in \N,x\in \R}$, defined by 
		\begin{align*}
		\mathfrak{h}^{(n)}_t(x) \, = \,  t^{-1/3} \Big( \mathcal{H}^{(n)}_{t}\big(t^{2/3} x\big)+ t/24 \Big) \, ,
		\end{align*}
		satisfies the $\mathbf{H}^{\mathrm{long}}_{t}$-Brownian Gibbs property.
		\item and the scaled short time line ensemble $\{\mathfrak{g}^{(n)}_t(x)\}_{n\in \N,x\in \R}$, defined by 
		\begin{align}\label{eq:UsilonNDefd}
		\mathfrak{g}^{(n)}_t(x) \, = \,  (\pi t/4)^{-1/4} \Big( \mathcal{H}^{(n)}_{t}\big((\pi t/4)^{1/2} x\big)+ \log\sqrt{2\pi t} \Big) \, ,
		\end{align}
		satisfies the $\mathbf{H}^{\mathrm{short}}_{t}$-Brownian Gibbs property.
	\end{enumerate}
\end{lemma}
\begin{proof}
The part $(1)$, $(2)$ and $(3)$ follow from the part $(1)$, $(2)$ and $(3)$ of Theorem 2.15 of \cite{CH16} respectively. For the proof of part $(4)$, we rely on the proof of part $(3)$ in \cite[Theorem 2.15]{CH16}. The main ingredients of their proof were the one point tail probabilities and spatial stationarity of $\h^{(1)}_t(\cdot)$. These two properties are also present for the lowest indexed curve for the short time line ensemble $\{\mathfrak{g}^{(n)}_t(x)\}_{n\in \N,x\in \R}$ (see Theorem~\ref{uptail-upb},~\ref{short:lowertail} and Lemma~\ref{stationary}). With these in hand, the  part $(4)$ can be proved exactly in the same way as part $(3)$ of \cite[Theorem 2.15]{CH16}. For brevity, we skip the details. 
\end{proof}

The above result envisages the Brownian Gibbs property of the ensemble to be brought to bear as a tool for analysing the spatial profiles $\mathfrak{h}_t(x)$ and $\sh_t(x)$ by demonstrating that the lowest indexed curves $\mathfrak{h}_t^{(1)}$ and $\sh_t^{(1)}$ in the scaled long time and short time KPZ line ensemble have the laws of the centered and scaled narrow wedge solution $\mathfrak{h}_t(x):=\mathfrak{h}_t(1,x)$ and $\sh_t(x):= \sh_t(1,x)$ of the KPZ equation defined in~\eqref{def:kpz-short-scale}. 



\subsection{Barlow-Taylor's macroscopic fractal theory}\label{sec:BT}

\begin{definition}[Hausdorff content and dimension]\label{bfHausdorffCont}
For any Borel set $\mathcal{A}\subset \mathbb{R}$, the $n$\emph{-th shell} of $\mathcal{A}$ is defined as $\mathcal{A}\cap \big\{(-e^{n+1},-e^{n}]\cup [e^{n},e^{n+1})\big\}$. Let us fix a number $c_0>0$, and the set $\mathcal{A}\subset \mathbb{R}$ and $\rho>0$, define $\rho$-dimensional \emph{Hausdorff content} of the $n$-th shell of $\mathcal{A}$ as 
\begin{align*}
\nu_{n,\rho}(\mathcal{A}) := \inf \sum_{i=1}^{m} \Big(\frac{\mathrm{Length}(Q_i)}{e^n}\Big)^{\rho}
\end{align*} 
where the infimum is taken over all sets of intervals $Q_1,\ldots , Q_m$ of length greater than $c_0$ and covering $n$-th shell of $\mathcal{A}$. Define the $\rho$-dimensional Hausdorff content of the set $\mathcal{A}$ as a sum total of $\nu_{n,\rho}(\mathcal{A})$ as $n$ varies over the set of all positive integers. Then, the Barlow-Taylor \emph{macroscopic Hausdorff dimension} the set $\mathcal{A}$ is defined as the infimum over all $\rho>0$ such that the $\rho$-dimensional Hausdorff content of $\mathcal{A}$ is finite, i.e.,
\begin{align*}
\mathrm{Dim}_{\mathbb{H}}(\mathcal{A}) := \inf \big\{\rho>0: \sum_{n=0}^{\infty} \nu_{n,\rho}(\mathcal{A})<\infty\big\}.
\end{align*} 
\end{definition}
From the definition, it follows that the macroscopic Hausdorff dimension of a bounded set is $0$. Just as in the microscopic case, one has $\mathrm{Dim}_{\mathbb{H}}(E)\leq \mathrm{Dim}_{\mathbb{H}}(F)$ when $E$ is contained in $F$. Furthermore, it has been observed in \cite[Lemma~2.3]{KKX17} that the macroscopic Hausdorff dimension does not depend on the value of $c_0$. These observations are summarized in the following proposition.

\begin{proposition}[\cite{BT1,BT2,KKX17}]\label{ppn:EssHaus}
Consider $E\subset \mathbb{R}$. Then, $\mathrm{Dim}_{\mathbb{H}}(E)$ does not depend on the value of $c_0$ of Definition~\ref{bfHausdorffCont} and $\mathrm{Dim}_{\mathbb{H}}(E)\leq \mathrm{Dim}_{\mathbb{H}}(F)$ for $F\supset E$. Moreover, $\mathrm{Dim}_{\mathbb{H}}(E)=0$ if $E$ is bounded.
\end{proposition} 
Since the choice of $c_0>0$ does not matter, we will work with the choice of Barlow and Taylor \cite{BT1,BT2,KKX17} and from now on, we set $c_0:= 1$.

We next mention a technical estimate on the Hausdorff content of any set. The following proposition, as stated in \cite{KKX17} is a macroscopic analogue of the classical Frostman lemma for microscopic Hausdorff dimension.

\begin{proposition}[Lemma~2.5 of \cite{KKX17}]\label{ppn:Frostman}
Fix $n\in \mathbb{R}_{\geq 1}$, and suppose $E$ is a subset of the shell $[-e^{n+1}, -e^{n}) \cup (e^{n}, e^{n+1}]$. Denote the Lebesgue measure of a Borel set $B\subset \mathbb{R}$ by $\mathrm{Leb}(B)$. Let $\mu$ be a finite Borel measure on $\mathbb{R}$ and define for $\rho>0$,
\begin{align}\label{eq:Kdef}
K_{n,\rho} := \sup\Big\{\frac{\mu(Q)}{\mathrm{Leb}(Q)}: Q \text{ is a Borel set in }[-e^{n+1}, -e^{n}) \cup (e^{n}, e^{n+1}], \quad \mathrm{Leb}(Q)\geq 1\Big\}.
\end{align}
Then, we have $\nu_{n,\rho}(E)\geq K^{-1}_{n,\rho}e^{-n\rho}\mu(E)$.
\end{proposition}

The above proposition will be used in Section~\ref{sec:MuonoMult} to show lower bound to the macroscopic Hausdorff dimension of the level sets of the KPZ equation. In the following, we introduce a notion of \emph{thickness} of a set, another important tool to bound the Hausdorff dimension from below. 
\begin{definition}[$\theta$-Thickness]\label{bd:Thickness}
 Fix $\theta\in (0,1)$ and define $$\Pi_n(\theta):=\bigcup_{\substack{0\le j\le e^{n(1-\theta)+1}-e^{n(1-\theta)} \\ j\in \Z}} \{e^n+je^{n\theta}\}.$$
We say $E\subset \R$ is $\theta$-thick if there exist integer $M=M(\theta)$ such that $E\cap [x,x+e^{\theta n}]\neq \varnothing$ for all $x\in \Pi_{n}(\theta)$ and for all $n\geq M$.
\end{definition}
The following result from \cite{KKX17} provides a lower bound to the Hausdorff dimension of a given set in terms of its thickness.   

\begin{proposition}[Corollary~4.6 in \cite{KKX17}]\label{thick} If $E\subset \R$ is $\theta$-thick for some $\theta\in (0,1)$, then $\operatorname{Dim}_{\mathbb{H}}(E)\ge 1-\theta$. 
\end{proposition}

\subsection{KPZ equation results}\label{sec:kpzeq} We start with introducing the space-time scaling of the KPZ height function appropriate for the short time regime, i.e., the case when the time variable goes to $0$.
\begin{align}\label{def:kpz-short-scale}
\sh_t(\alpha,x) :=\frac{\calH^{\mathbf{nw}}(\alpha t,(\pi t/4)^{1/2}x)+\log\sqrt{2\pi\alpha t}}{(\pi t/4)^{1/4}}.
\end{align}
We will often use the shorthand notation $\sh_t(x):=\sh_t(1,x)$. In addition, we simply write $\sh_t:=\sh_t(1,0)$ when $x=0$. The following lemma shows the spatial stationarity of the process $\h_t(\cdot)$ and $\sh_t(\cdot)$.   

\begin{lemma}[Stationarity]\label{stationary} The one point distribution of $\h_t(x)+\frac{x^2}{2}$ is independent of $x$ and converges weakly to Tracy-Widom GUE distribution as $t\uparrow \infty$. On the other hand, the one point distribution of $\sh_t(x)+\frac{(\pi t/4)^{3/4}x^2}{2t}$ is independent of $x$ and converges weakly to standard Gaussian distribution as $t\downarrow 0$. 
\end{lemma}
\begin{proof} The first part was proved in Proposition 1.7 of \cite{CH16}. By Proposition~1.4 of \cite{ACQ11}, we know $\calH^{\mathbf{nw}}(t,z)+\frac{z^2}{2t}$ is stationary in $z$. As a result,  \begin{align*}
	\sh_t(x)+\frac{(\pi t/4)^{3/4}x^2}{2t}=(\pi t/4)^{-1/4}\left[\calH^{\mathbf{nw}}(t,(\pi t/4)^{1/2}x)+\frac{(\pi t/4)x^2}{2t}+\log\sqrt{2\pi t}\right]
\end{align*}
is stationary in $x$. From Proposition 1.8 in \cite{ACQ11}, it follows that $\sh_t(0)$ converges weakly to standard Gaussian distribution as $t\downarrow 0$.
\end{proof}
Our next result provides a multipoint composition law of the KPZ temporal process. In latter sections, this will be used to infer properties of multipoint distributions of $\h_t$. Our proof of the multipoint composition law resembles the one in \cite[Proposition~2.9]{corwin2019kpz} which proves the two point composition law. For stating the law, we introduce the following notation.
For $t>0$, define a $t$-indexed composition map $I_t(f,g)$ between two functions $f(\cdot)$ and $g(\cdot)$ as
\begin{align}\label{eq:Icomp}
I_t(f,g):=t^{-1/3}\log \int^{\infty}_{-\infty} e^{t^{1/3}\big(f(t^{-2/3}y) +g(-t^{-2/3} y)\big)} dy \, .
\end{align}

\begin{proposition}\label{ppn:MulpointComposition}
For any fixed $t>0$, $k\in \mathbb{N}$ and $1<\alpha_1<\alpha_2<\ldots <\alpha_k$, there exist independent spatial processes $\mathfrak{h}_{\alpha_1 t \downarrow t}, \mathfrak{h}_{\alpha_2 t \downarrow \alpha_1 t}, \cdots ,\mathfrak{h}_{\alpha_{k} t \downarrow \alpha_{k-1}t} $ supported on the same probability space as the KPZ equation solution such that:
\begin{enumerate}
\item $\mathfrak{h}_{\alpha_i t\downarrow \alpha_{i-1}t}(\cdot)$ is distributed according to the law of the process $\mathfrak{h}_{\alpha_{i-1}t}((\alpha_i-\alpha_{i-1})/\alpha_{i-1}, \cdot)$;
\item $\mathfrak{h}_{\alpha_i t\downarrow \alpha_{i-1}t}(\cdot)$ is independent of $\mathfrak{h}_{\alpha_{i-1} t}(\cdot)$; and
\item $\mathfrak{h}_{\alpha_{i-1}t}(\frac{\alpha_i}{\alpha_{i-1}},0) = I_{\alpha_{i-1}t}\big(\mathfrak{h}_{\alpha_{i-1} t}, \mathfrak{h}_{\alpha_i t\downarrow \alpha_{i-1}t}\big)$.
\end{enumerate} 
\end{proposition}

\begin{proof}
For $s<t$ and $x,y\in \mathbb{R}$, let $\mathcal{Z}^{\mathbf{nw}}_{s,x}(t,y)$ be the solution at time $t$ and position $y$ of the SHE started at time $s$ with Dirac delta initial data at position $x$. 
We will show that for any $0<t_1<\ldots <t_k$ and $y_1,\ldots ,y_k\in \mathbb{R}$, there exists independent spatial processes $\mathcal{Z}_{y_2}(t_2\downarrow t_1,\cdot), \ldots , \mathcal{Z}_{y_k}(t_{k}\downarrow t_{k-1}, \cdot)$ coupled on a probability space upon which the space-time white noise of the KPZ equation is defined such that  
\begin{align}\label{eq:Conv}
\mathcal{Z}^{\mathbf{nw}}(t_i,y_i) = \mathcal{Z}^{\mathbf{nw}}_{0,0}(t_i,y_i) = \int_{\mathbb{R}} \mathcal{Z}^{\mathbf{nw}}_{0,0}(t_{i-1},x) \mathcal{Z}_{y_i}(t_{i}\downarrow t_i,x)dx \, ,
\end{align}
  and the law of $\mathcal{Z}_{y_{i}}(t_i\downarrow t_{i-1},\cdot)$ is same as that of $\mathcal{Z}^{\mathbf{nw}}(t_i-t_{i-1},y_i-\cdot)$ for $2\leq i\leq k$. Expressing the convolution and interchange properties in terms of $\mathfrak{h}_t$ immediately yields the proposition.


 We now return to show \eqref{eq:Conv}. The above convolution formula is known when $k=2$ (see \cite{corwin2019kpz}). We extend the proof given in \cite{corwin2019kpz} for $k>2$ using the chaos series for the SHE (see \cite{Cor18,Walsh86,alberts2014} for background). 
Here we have used the following notations. We write $\vec{s}=(s_1,\ldots, s_\ell)\in \mathbb{R}^{\ell}_{\geq 0}$,  $\vec{x}= (x_1,\ldots, x_\ell)\in \mathbb{R}^{\ell}$ and define the set of ordered times
$$\Delta_{\ell}(s,t) = \{\vec{s}: s\leq s_1\leq s_2\leq \ldots \leq s_\ell\leq t\}.$$ 
 For any $0\leq s<t$ and $x,y\in \mathbb{R}$, $\mathcal{Z}^{\mathbf{nw}}_{s,x}(t,y)$ is given as  the following chaos series expansion (see Theorem~2.2 of \cite{Cor18}):
\begin{align}
\mathcal{Z}^{\mathbf{nw}}_{s,x}(t,y) = \sum_{\ell=0}^{\infty} \int_{\Delta_{\ell}(s,t)} \int_{\mathbb{R}^{\ell}} P_{\ell;s,x;t,y}(\vec{s},\vec{x})d\xi^{\otimes_\ell}(\vec{s},\vec{x})\label{eq:ZEq}.
\end{align}
The integration in \eqref{eq:ZEq} is a multiple It\^o stochastic integral against the white noise $\xi$ and the term $P_{\ell;s,x;t,y}(\vec{s},\vec{x})$ is the density function for a one-dimensional Brownian motion starting from $(s,x)$ to go through the time-space points $(s_1,x_1),\ldots, (s_\ell,x_\ell)$ and ends up at $(t,y)$. This transition density has the following product formula using the Gaussian heat kernel $p(s,y) := (2\pi s)^{-1/2}\exp(-y^2/2s)$ and the conventions $s_0=s$,  $s_{\ell+1}=t$, $x_0=x$ and $x_{\ell+1}=y$:
\begin{align*}
P_{\ell;s,x;t,y}(\vec{s},\vec{x}) = \prod_{j=0}^{\ell} p(s_{i+1}-s_i,x_{i+1}-x_i).
\end{align*}
For any $0\leq s<t$, the heat kernel $p(\cdot,\cdot)$ satisfies the simple convolution identity
\begin{align}\label{eq:pconv}
p(t,x) = \int p(s,y) p(t-s,x-y)  dy.
\end{align}

Fix $2\leq i\leq k$. By using the fact that the sum of indicator functions gives the value one, we may replace $\int_{\Delta_\ell(0,t_i)}$  in \eqref{eq:ZEq} by the quantity $\sum_{j=0}^\ell \int_{\Delta_\ell(0,t_i)} \mathbf{1}_{s_j\leq t_i<s_{j+1}}$. As a consequence, we get 
\begin{align}\label{eq:indinsert}
\mathcal{Z}^{\mathbf{nw}}_{0,0}(t_i,y_i) =  \sum_{\ell=0}^{\infty}\sum_{j=0}^{\ell} \int_{\Delta_{\ell}(s_j,t_i)} \int_{\mathbb{R}^{k}} \mathbf{1}_{s_j\leq t_{i-1}<s_{j+1}}  P_{\ell;0,0;t_i,y_i}(\vec{s},\vec{x})d\xi^{\otimes_\ell}(\vec{s},\vec{x}).
\end{align}
For $1\leq a\leq b\leq k$,  $\vec{s}_{[a,b]}$ denotes $(s_{a},\ldots, s_b)$ and likewise for $\vec{x}$. Using these notations and \eqref{eq:pconv}, we may write 
$$
 \mathbf{1}_{s_j\leq t_{i-1}<s_{j+1}}  P_{\ell;0,0;t_i,y_i}(\vec{s},\vec{x}) =  \mathbf{1}_{s_j\leq t_{i-1}<s_{j+1}}  \int_{\mathbb{R}} P_{\ell;0,0;t_{i-1},z}(\vec{s}_{[1,j]},\vec{x}_{[1,j]})P_{\ell-j;t_{i-1},z;t_i,y}(\vec{s}_{[j+1,\ell]},\vec{x}_{[j+1,\ell]})dz.
$$

We now insert the above display into \eqref{eq:indinsert}. We also replace $\int_{\Delta_{\ell}(0,t_i)} \mathbf{1}_{s_j\leq t_{i-1}<s_{j+1}}$ by the product of the integral $\int_{\Delta_{j}(0,t_{i-1})}\int_{\Delta_{\ell-j}(t_{i-1},t_i)}$ and relabel $\vec{s}_{[1,j]}=\vec{u}$, $\vec{s}_{[j+1,\ell]}=\vec{v}$, $\vec{x}_{[1,j]}=\vec{a}$, $\vec{x}_{[j+1,\ell]}=\vec{b}$, $P_{j;0,0;t_{i-1},z}(\vec{u},\vec{a}) = P^{0,0}_{j;t_{i-1},z}(\vec{u},\vec{a})=P_{j;t_{i-1},z}(\vec{u},\vec{a})$ and $P_{\ell-j;t_{i-1},z;t_i,y}(\vec{v},\vec{b}) = P^{t_{i-1},z}_{\ell-j;t_i,y}(\vec{v},\vec{b})$. Using of the fact that the white noise integration can be split since the times range over disjoint intervals, we find
$$
\mathcal{Z}^{\mathbf{nw}}_{0,0}(t_i,y_i) =  \sum_{\ell=0}^{\infty}\sum_{j=0}^{\ell} \int_{\Delta_j(0,t_{i-1})}\int_{\Delta_{\ell-j}(t_{i-1},t_i)} \int_{\mathbb{R}^{i}}\int_{\mathbb{R}^{\ell-j}} \int_{\mathbb{R}} P_{j;t_{i-1},z}(\vec{u},\vec{a})P^{t_{i-1},z}_{\ell-j;t_i,y}(\vec{v},\vec{b})dz d\xi^{\otimes_j}(\vec{u},\vec{a})d\xi^{\otimes_{\ell-j}}(\vec{v},\vec{b}).
$$
By the change of variables $m=\ell-j$, the double sum $\sum_{\ell=0}^{\infty}\sum_{j=0}^\ell$ can be replaced by $\sum_{j=0}^{\infty}\sum_{m=0}^\infty$. We bring  the integral in $z$ to the outside resumming and reordering of integrals is readily justified since all sums are convergent in $L^{2}$ (with respect to the probability space on which $\xi$ is defined -- see, for example, \cite[Theorem 2.2]{Cor18} for details). As a result, we get 
\begin{align*}
\mathcal{Z}^{\mathbf{nw}}_{0,0}(t_i,y_i) = \int_{\mathbb{R}} dz &\Big(\sum_{j=0}^{\infty}\int_{\Delta_j(0,t_{i-1})}P_{j;t_{i-1},z}(\vec{u},\vec{a}) d\xi^{\otimes_j}(\vec{u},\vec{a})\Big)\nonumber\\
& \times \Big(\sum_{m=0}^{\infty}\int_{\Delta_{\ell-j}(t_{i-1},t_i)}P^{t_{i-1},z}_{\ell-j;t_i,y}(\vec{v},\vec{b})d\xi^{\otimes_{\ell-j}}(\vec{v},\vec{b})\Big) .
\end{align*}
Comparing with \eqref{eq:ZEq}, we may now recognize that 
\begin{align*}
\mathcal{Z}^{\mathbf{nw}}_{0,0}(t_{i-1}, z) = \sum_{j=0}^{\infty}\int_{\Delta_j(0,t_{i-1})}P_{j;t_{i-1},z}(\vec{u},\vec{a}) d\xi^{\otimes_j}(\vec{u},\vec{a}),
\end{align*}
for any $z\in \mathbb{R}$ whereas the stochastic process 
$$\mathcal{Z}_{y_i}(t_i\downarrow t_{i-1},z) :=  \sum_{m=0}^{\infty}\int_{\Delta_{\ell-j}(t_{i-1},t_i)}P^{t_{i-1},z}_{\ell-j;t_i,y_i}(\vec{v},\vec{b})d\xi^{\otimes_{\ell-j}}(\vec{v},\vec{b}) $$
is same in distribution with $\mathcal{Z}^{\mathbf{nw}}_{t_{i-1},z}(t_i,y_i)$. Furthermore, $\mathcal{Z}^{\mathbf{nw}}_{0,0}(t_{i-1}, \cdot)$ and $\mathcal{Z}_{y_i}(t_i\downarrow t_{i-1},\cdot)$ are independent since they are defined with respect to disjoint portions of the space-time white noise. Due to the same reason, $\mathcal{Z}_{y_i}(t_i\downarrow t_{i-1},\cdot)$ and $\mathcal{Z}_{y_j}(t_i\downarrow t_{j-1},\cdot)$ for any $1\leq i<j\leq k$. Recall the {\em interchange} property of the SHE: namely that, for $s<t$ and $y\in \mathbb{R}$ fixed,  $\mathcal{Z}^{\mathbf{nw}}_{s,x}(t,y)$ is equal in law as a process in $x$ to $\mathcal{Z}^{\mathbf{nw}}_{s,y}(t,x)$ -- the change between the two expressions is in the interchange of $x$ and $y$. By the interchange property, the spatial process $\mathcal{Z}^{\mathbf{nw}}_{t_{i-1},\cdot}(t_i,y_i)$ has same law as $\mathcal{Z}^{\mathbf{nw}}_{0,0}(t_i-t_{i-1},y_i-\cdot)$. This completes the proof of \eqref{eq:Conv}.  
\end{proof}

Our next proposition which is taken from \cite[Proposition~2.7]{corwin2019kpz} states a FKG type inequality for the KPZ equation. 

\begin{proposition}\label{ppn:FKG}
For any $t_1,t_2>\mathbb{R}_{>0}$ and $s_1,s_2\in \mathbb{R}$, we have 
\begin{align*}
\mathbb{P}(\h_{t_1}\geq s_1, \h_{t_2}\geq s_2\big)\geq \mathbb{P}(\h_{t_1}\geq s_1)\mathbb{P}(\h_{t_2}\geq s_2).
\end{align*}
\end{proposition}

%
%

In the following two results, we will see the one point tail probabilities of the temporal process $\h_t$ which are proved in \cite{CG18a, CG18b}. We state the results from \cite{corwin2019kpz} which has used same notations as ours. These results hold for any finite time $t>0$. Since the short time scaling of the KPZ equation has the Gaussian limit, the same tail bounds as in the forthcoming result does not hold as $t$ goes to $0$. The short time tail bounds which are tackled in Theorem~\ref{short-uptail} and~\ref{short:lowertail} should be contrasted with the following two propositions.

\begin{proposition}[Proposition 2.12 from \cite{corwin2019kpz}] \label{onepointuptail}
For any $t_0>0$, there exist $s_0 = s_0(t_0)>0$ and $c_1(t_0)>c_2(t_0)>0$ such that, for $t\geq t_0$, $s>s_0$ and $x\in \R$,
\begin{align}\label{eq:UpperTail}
 \exp\big(-c_1 s^{3/2}\big)\le \mathbb{P}\Big(\mathfrak{h}_t(x)+\frac{x^2}{2}\geq s\Big) \leq \exp\big(-c_2 s^{3/2}\big) \, .
\end{align}
\end{proposition}
\begin{proposition}[Proposition 2.11 from \cite{corwin2019kpz}]\label{onepointlowtail}
	For any  $t_0>0,\varepsilon>0$, there exist $s_0 = s_0(t_0)>0$ and $c=c(t_0)>0$ such that, for  $t>t_0$, $s>s_0$ and $x\in \R$,
	\begin{align}
 \mathbb{P}\Big(\mathfrak{h}_{t}(x)+\frac{x^2}{2} \leq -s \Big) &\leq \exp\big(- cs^{5/2}\big) \, . \label{eq:LowTailUp}
	\end{align}
\end{proposition}
As one may notice, the constants of the tail bound in the above two propositions are left imprecise. For deriving tail bounds of Section~\ref{sec:spatialtail} and~\ref{sec:diffprob}, we do not need the exact values of those constants. However, in Section~\ref{sec:lil} and~\ref{sec:MuonoMult}, we require precise description of those constants only in the case when the time variable $t$ is large. The following proposition quotes relevant tail bounds from \cite{CG18a,CG18b,CC19,Z19} for large values of $t$. 

\begin{proposition}\label{ppn:LargeTime}
Fix  $t_0>0$ large  and $\varepsilon\in (0,1)$. Then, there exist $s_0 = s_0(t_0,\varepsilon)>0$ and $c=c(t_0,\varepsilon)>0$ such that, for  $t>t_0$, $c(\log t)^{2/3}> s>s_0$ and $x\in \R$,
\begin{align}
\exp\big(-\frac{4\sqrt{2}}{3}(1+\varepsilon)s^{3/2}\big)\leq \mathbb{P}\Big(\mathfrak{h}_{t}(x)+\frac{x^2}{2} \geq s\Big)\leq \exp\big(-\frac{4\sqrt{2}}{3}(1-\varepsilon)s^{3/2}\big) \label{eq:uptailTlarge}
\end{align}
and,
\begin{align}
\exp\big(-\frac{1}{6}(1+\varepsilon)s^3\big)\leq \mathbb{P}\Big(\mathfrak{h}_{t}(x)+\frac{x^2}{2} \leq -s\Big)\leq \exp\big(-\frac{1}{6}(1-\varepsilon)s^3\big). \label{eq:lowtailTlarge}
\end{align}

\end{proposition}  

\begin{proof}
 Since $\h_{t}(x)+ \frac{x^2}{2}$ is stationary in $x$, it suffices to prove \eqref{eq:uptailTlarge} and \eqref{eq:lowtailTlarge} for $x=0$. From the specifications of the upper and lower bounds of the upper tail probabilities in Theorem~1.10 (part (a)) of \cite{CG18b}, \eqref{eq:uptailTlarge} follows immediately. It remains to show \eqref{eq:lowtailTlarge}. Theorem~1.1 of \cite{CG18a} which is recently been strengthened in \cite{CC19,Z19} proves that for any given $\varepsilon, t_0>0$, there exists $s_0=(t_0,\varepsilon)>0$ such that for all $s\geq s_0$ and $t\geq t_0$,
\begin{align}\label{eq:LowUp}
 \mathbb{P} \big(\h_t(0)\leq -s\big)\leq \exp\Big(-\frac{4\sqrt{2}(1-2\varepsilon)}{15}t^{\frac{1}{3}}s^{\frac{5}{2}}\Big) +\exp(-Ks^3-\varepsilon st^{\frac{1}{3}})+ \exp\Big(-\frac{(1-2\varepsilon)}{6}s^3\Big)
\end{align}
   and,
\begin{align}\label{eq:HighUp}
\mathbb{P} \big(\h_t(0)\leq -s\big)\geq   \exp\big(-\frac{4\sqrt{2}}{15}(1+\varepsilon)t^{\frac{1}{3}}s^{\frac{5}{2}}\big)+ \exp\big(-\frac{1}{6}(1+\varepsilon)s^3\big).
\end{align} 
 The first inequality of \eqref{eq:lowtailTlarge} follows from \eqref{eq:HighUp}. Note that $s^{5/2}t^{1/3}\gg s^3$ and $\varepsilon st^{1/3}\gg s^3$ when we have $(\log t)^{2/3}\gg s$. By choosing $s_0$ and $c$ large, we may bound the right hand side of \eqref{eq:LowUp} by $\exp(-(1-\varepsilon)s^3/6)$ for all $s\geq s_0$ satisfying $c(\log t)^{2/3}> s$. This proves \eqref{eq:lowtailTlarge}.     
\end{proof}

The next two results which are proved in \cite{corwin2019kpz} provide tail bounds on the supremum and infimum of the spatial process $\h_t(\cdot)$ for any fixed time $t>0$.

\begin{proposition}[Proposition 4.1 from \cite{corwin2019kpz}]
	\label{ppn:LowerTail}
	For any $t_0>0$ and $\nu\in (0,1]$, there exist $s_0= s_0(t_0,\nu)>0$ and $c=c(t_0,\nu)>0$ such that, for $t\geq t_0$ and $s>s_0$,
	\begin{align*}
	\mathbb{P}(\mathsf{B}) \leq \exp\big(- cs^{5/2}\big)\qquad \textrm{where}\quad \mathsf{B}= \Big\{\inf_{x\in \R}\Big( \mathfrak{h}_t(x)+\frac{(1-\nu)x^2}{2}\Big) \leq -s\Big\}.
	\end{align*}
\end{proposition}

\begin{proposition}[Proposition 4.2 from \cite{corwin2019kpz}]
	\label{ppn:UpperTail}
	For any $t_0>0$ and $\nu\in (0,1]$, there exist $s_0= s_0(t_0,\nu)>0$ and $c_1=c_1(t_0,\nu)>c_2=c_2(t_0,\nu)>0$ such that, for $t\geq t_0$ and $s>s_0$,
	\begin{align*}
	\exp\big(- c_1s^{3/2}\big) \leq \mathbb{P}(\mathsf{A}) \leq \exp\big(- c_2 s^{3/2}\big)\quad \textrm{where}\quad \mathsf{A}= \Big\{\sup_{x\in \R}\Big( \mathfrak{h}_t(x)+\frac{(1-\nu)x^2}{2}\Big) \geq s\Big\}.
	\end{align*}
\end{proposition}

We end this section with the Paley-Zygmund inequality which is applicable for any positive random variable. 

\begin{proposition}[Paley-Zygmund Inequality]\label{ppn:Paley-Zygmund}
Fix $\delta\in (0,1)$. For any positive random variable $X$, 
\begin{align*}
\mathbb{P}\big(X\geq \delta \mathbb{E}[X]\big)\geq \frac{(1-\delta)^2 (\mathbb{E}[X])^2}{\mathbb{E}[X^2]}.
\end{align*} 
\end{proposition} 

\section{Short Time Tail Bounds} \label{sec:onepoint}
The main goal of this section is to prove Theorem~\ref{uptail-upb} and~\ref{short:lowertail} which describe uniform bounds to the one point tail probabilities of the KPZ height function as time variable $t$ goes to $0$. The proof of Theorem~\ref{uptail-upb} which is given in Section~\ref{sec:UpperTail} will use the exact formulas of the integer moments of the SHE. These formulas are put forward by Kardar \cite{Kardar87} using the techniques of replica Bethe ansatz. See \cite{ghosal2018moments} for a discussion on different approaches to prove those formulas rigorously. On the other hand, the proof of Theorem~\ref{short:lowertail} which is contained in Section~\ref{sec:LowerTail} will be based on core probabilistic aspect like Gaussian concentration. 

\subsection{Upper Tail}\label{sec:UpperTail}
 
 Our starting point which is the content of the following proposition is to provide upper bounds to the exponential moments of $\sh_t$. Using these moment estimates, the proof of Theorem~\ref{uptail-upb} will be completed in the ensuing subsection.
\begin{proposition}\label{short-mom-bound} Fix $\e>0$. There exist $t_0=t_0(\e)>0$, $C=C(\e)>0$, and $s_0=s_0(\e)>0$, such that for all $t\le t_0$, $s\ge s_0$ and $k:=\lfloor s(\pi t/4)^{-1/4} \rfloor$ we have 
	\begin{align}\label{e:short-mom-bound}
	\Ex\big[\exp\big(k(\pi t/4)^{1/4}\sh_{2t}\big)\big] &  \le \exp\big(C(s^3t^{1/4-4\e}+s^2)\big).
	\end{align}
\end{proposition}
	\begin{proof} For any positive integer $k$, we recall the $k$-moment formula for $\calZ^{\mathbf{nw}}(2t,0)$ (see \cite{borodin2014macdonald, ghosal2018moments})
	\begin{align*} 
	\Ex\left[\calZ^{\mathbf{nw}}(2t,0)^ke^{\frac{kt}{12}}\right] & =\sum_{\substack{\lambda \vdash k \\ \lambda=1^{m_1}2^{m_2}\ldots}} \frac{k!}{\prod m_j!}\prod_{i=1}^{\ell(\lambda)} \frac{e^{\frac{t\lambda_i^3}{12}}}{2\pi}\int\limits_\mathbb{R}\cdots\int\limits_\mathbb{R} \prod_{i=1}^{\ell(\lambda)}\frac{dz_i 	e^{-t^{\frac13}\lambda_i z_i^2}}{t^{\frac13}\lambda_i}\prod_{i<j}^{\ell(\lambda)}\frac{\frac{t^{\frac23}(\lambda_i-\lambda_j)^2}{4}+(z_i-z_j)^2}{\frac{t^{\frac23}(\lambda_i+\lambda_j)^2}{4}+(z_i-z_j)^2}.
	\end{align*}
	Note that each terms of the product inside the integral is less than $1$.
Bounding those terms by $1$ and evaluating the left over Gaussian integral, we have
	\begin{align*}
	\Ex\left[\calZ^{\mathbf{nw}}(2t,0)^ke^{\frac{kt}{12}}\right] & \le \sum_{\substack{\lambda \vdash k \\ \lambda=1^{m_1}2^{m_2}\ldots}} \frac{k!}{\prod m_j!}\prod_{i=1}^{\ell(\lambda)} \frac{e^{\frac{t\lambda_i^3}{12}}}{2\pi} \prod_{i=1}^{\ell(\lambda)}\frac{\sqrt{\pi}}{t^{\frac12}\lambda_i^{3/2}}   \le \sum_{\substack{\lambda \vdash k \\ \lambda=1^{m_1}2^{m_2}\ldots}} \frac{k!e^{\frac{tk^3}{12}}}{(4\pi t)^{\frac{\ell(\lambda)}2}\prod m_j!}
	\end{align*}	
	The last inequality in above equation follows by using $\lambda_i^{3/2}\ge 1$ and $\sum_i\lambda_i^3 \le k^3$.
	 Expressing the left hand side of the above display in terms of $\sh_{2t}$ we get 
	\begin{align}\label{e:mom-bound4}
	\Ex e^{k(\pi t/2)^{1/4}\sh_{2t}}=\Ex\left[(\calZ^{\mathbf{nw}}(2t,0)\sqrt{4\pi t})^k\right] &  \le e^{\frac{tk^3-tk}{12}}\sum_{\substack{\lambda \vdash k \\ \lambda=1^{m_1}2^{m_2}\ldots}} (4\pi t)^{\frac{k-\ell(\lambda)}{2}}\frac{k!}{\prod m_j!}
	\end{align}
	We choose $t_0$ and $s_0$ such that $2^{5/2}t_0^{\e}(\pi/2)^{1/4}\le \frac12$ and $s_0 \ge 2(\pi t_0/2)^{1/4}$. Then for all $t\le t_0$ and $s\ge s_0$, we set $k=k(t):=\lfloor s(\pi t/2)^{-1/4} \rfloor$. By the condition on $t_0,s_0$ and $k$, we always have $k \ge 2$. We further have $k\le s(\pi t/2)^{-1/4}$ which implies $t \le \frac{2s^4}{\pi k^4}$. Bounding $t$ with this inequality, combining it with the estimate $k! \le k^{k-m_1}m_1!$ and using those in the right hand side of \eqref{e:mom-bound4} yields
	\begin{align}
	\Ex e^{k(\pi t/2)^{1/4}\sh_{2t}} \le e^{\frac{s^{4}(k^{-1}-k^{-3})}{6\pi}}\sum_{\substack{\lambda \vdash k \\ \lambda=1^{m_1}2^{m_2}\ldots}} \left(2^{3/4}s\right)^{2k-2\ell(\lambda)}\frac{k^{k-m_1}k^{2\ell(\lambda)-2k}}{\prod\limits_{j\ge 2} m_j!} \label{e:mom-bound1} 
	\end{align}
	Throughout the rest, we provide bound for the right hand side of \eqref{e:mom-bound1}. We separate our analysis into three cases depending on the location of $s$.

	\noindent\textbf{Case-1.} $s\le t^{-1/4+\e}$. Observe that $k-\ell(\lambda)= \sum_{j\geq 2} (j-1)m_j$ and $2\ell(\lambda)-m_1-k = -\sum_{j\geq 3} (j-2)m_j$. 
	
%


	We extend the range of $m_2,m_3,m_4,\ldots$ over all non-negative integers in \eqref{e:mom-bound1}. Taking first the sum w.r.t. $m_2$ shows  
	\begin{align}
	\mbox{r.h.s.~of }\eqref{e:mom-bound1} &  \le e^{\frac{s^{4}(k^{-1}-k^{-3})}{6\pi}}\sum_{j=3}^{\infty} \sum_{m_j=0}^{\infty} \left(2^{3/2}s^2\right)^{\sum\limits_{j\ge 3}(j-1)m_j}\frac{k^{-\sum\limits_{j\ge 3}(j-2)m_j}}{\prod\limits_{j\ge 3} m_j!}\sum_{m_2=0}^{\infty} \frac{(2^{3/2}s^2)^{m_2}}{m_2!} \label{e:mom-bound3}
 \end{align}	
 Note that the inner sum w.r.t. $m_2$ is equal to $\exp(2^{3/2}s^2)$. We may now write the right hand side of the above display as
 \begin{align*}
	  & e^{\frac{s^{4}(k^{-1}-k^{-3})}{6\pi}} e^{2^{3/2}s^2}\prod_{j=3}^{\infty} \sum_{m_j=0}^{\infty} \left(2^{3/2}s^2\right)^{(j-1)m_j}\frac{k^{-(j-2)m_j}}{m_j!}  \\
	  & =  \exp\left(\frac{s^{4}(k^{-1}-k^{-3})}{6\pi}+2^{3/2}s^2+\frac{2^{3}s^4k^{-1}}{1-2^{3/2}s^2k^{-1}}\right)
	\end{align*}
	where the equality is obtained by taking sum w.r.t. $m_3,m_4,...$ separately and simplifying the product. With this equality, we get 
	\begin{align*}
	\mbox{r.h.s.~of }\eqref{e:mom-bound3} &  \le \exp\left(\frac{2s^{3}(\pi t/2)^{1/4}}{6\pi}+2^{3/2}s^2+\frac{2^{4}s^3(\pi t/2)^{1/4}}{1-2^{5/2}s(\pi t/2)^{1/4}}\right)  \le \exp\big(Cs^3t^{1/4}+Cs^2\big)
	\end{align*}
	where the last inequality is obtained by using the facts $k^{-1}\le 2s^{-1}(\pi t/2)^{1/4}$, $s\le t^{-1/4+\e}$ and $t\le t_0$ with $2^{5/2}t_0^{\e}(\pi/2)^{1/4}\le \frac12$.
	This proves \eqref{e:short-mom-bound} for $s\le t^{-1/4+\e}$.
	\smallskip 
	
	\noindent\textbf{Case-2.} $s\ge t^{-1/4-\e}$. We assume $t_0\le\frac1{4\pi}$. Recall the definition of $k$. Since $k+1\geq s/(\pi t/2)^{1/4}$, we may bound $s^{2k-2\ell(\lambda)}$ by $(\pi t/2)^{(k-\ell(\lambda))/2}(k+1)^{2k-2\ell(\lambda)}$. Combining this with the facts $k!\le k^k$ and $\prod m_j! \ge 1$, we get
	\begin{align*}
	\mbox{r.h.s.~of }\eqref{e:mom-bound4} \le e^{\frac{s^{4}(k^{-1}-k^{-3})}{6\pi}}\sum_{\substack{\lambda \vdash k \\ \lambda=1^{m_1}2^{m_2}\ldots}} \left(4\pi t\right)^{\frac{k-\ell(\lambda)}{2}}k^{k-m_1}\cdot k!\big(1+\frac{1}{k}\big)^{2(k-\ell(\lambda))} & \le e^{\frac{s^{4}(k^{-1}-k^{-3})}{6\pi}}k^{2k} 
	\end{align*}
	where we bound $(1+1/k)^{2(k-\ell(\lambda))}$ by $1$ and the number of partitions of $k$ by $k^{k}$ to get the last inequality. Since we are in the case $s\ge t^{-1/4-\e}$, we have $s^4k^{-1} \le s^3(\pi t/2)^{1/4}$ and $k\ln k \le cst^{-1/4}\ln(st^{-1/4})  \le cs^3t^{1/4}$. Due to these inequalities, the right hand side of the above display is bounded by $\exp(cs^3t^{1/4})$ for some constant $c>0$. Combining this with \eqref{e:mom-bound1} shows $\Ex \exp(k(\pi t/2)^{1/4}\sh_{2t})   \le \exp(cs^3t^{1/4}).$
	\smallskip

	\noindent\textbf{Case-3.} $t^{-1/4+\e} \le s\le t^{-1/4-\e}$. Define $\tilde{s} = t^{-1/4-\varepsilon}$ and $\tilde{k}:=\lfloor \tilde{s}(\pi t/2)^{-1/4} \rfloor$. Note that $k\leq \tilde{k}$ since $s\leq t^{-1/4-\e}$. Using the H\"older's inequality, we know $\Ex \exp(k(\pi t/2)^{1/4}\sh_{2t})$ is bounded by $\big(\Ex \exp(\tilde{k}(\pi t/2)^{1/4}\sh_{2t})\big)^{k/\tilde{k}}$. By \textbf{Case-2}, we know $\Ex \exp(\tilde{k}(\pi t/2)^{1/4}\sh_{2t}) \le \exp(c\tilde{s}^3t^{1/4})$ for all $t\leq t_0 =\frac{1}{4\pi}$. Combining these observations shows 
	\begin{align*}
	\Ex \exp(k(\pi t/2)^{1/4}\sh_{2t}) \le \exp(ck\tilde{s}^3t^{1/4}/\tilde{k}) \le \exp(ct^{-3/4-3\e}t^{1/4}st^{1/4+\e}) = \exp(cst^{-1/4-2\e}) 
	\end{align*}
	where the second inequality follows from the definition of $\tilde{k}$ and $\tilde{s}$. Since $s\geq t^{-1/4+\e}$, the last term of the above display is bounded by $\exp(cs^3t^{1/4-4\e})$. This completes the proof  for \textbf{Case-3}. 
	
	Combining all cases we get \eqref{e:short-mom-bound}. This completes the proof.
	\end{proof}

\subsubsection{Proof of Theorem~\ref{uptail-upb}}
	We introduce the notations $f_{t,s}:=\frac{1}{C+\sqrt{C^2+3Cst^{1/4-4\e}}}$, $\tilde{s} := sf_{t,s}$ and $\tilde{k}:=\lfloor\tilde{s}(\pi t/2)^{-1/4} \rfloor$ where the constant $C$ is same as in \eqref{e:short-mom-bound}. By Markov's inequality,
	\begin{align}
	\Pr(\sh_{2t}\ge s)  = \Pr(e^{\tilde{k}(\pi t/2)^{1/4}\sh_{2t}}\ge e^{\tilde{k}{s}(\pi t/2)^{1/4}})   & \le \exp(-\tilde{k}{s}(\pi t/2)^{1/4})\mathbb{E}\big[\exp(\tilde{k}(\pi t/2)^{1/4}\sh_{2t})\big]\nonumber\\ \le &\exp\big(Cs^3f_{t,s}^3t^{1/4-4\e}+Cs^2f_{t,s}^2-\tilde{k}{s}(\pi t/2)^{1/4}\big) \label{e-uptail}
	\end{align}
	where the last inequality follows from Proposition \ref{short-mom-bound}. We choose $s_0$ large enough such that for all $s\ge s_0$ and $t\le t_0$ we have $\tilde{k}{s}(\pi t/2)^{1/4} \ge \frac{11}{12}s^2f_{t,s}$. From the definition of $f_{t,s}$, it follows $$Cf_{t,s} \le C\frac{1}{2C}=\frac12, \qquad Cst^{1/4-4\e}f_{t,s}^2 \le Cst^{1/4-4\e}\frac{1}{3Cst^{1/4-4\e}} = \frac13 .$$ Plugging all these inequalities in the right side of \eqref{e-uptail} yields
	\begin{align*}
	\Pr(\sh_{2t}\ge s) \le  \exp\big(-\frac{s^2f_{t,s}}{12}\big) \le \exp\Big(-\frac{C's^2}{1+\sqrt{1+st^{1/4-4\e}}}\Big).
	\end{align*}
 for all $t\leq t_0$, $s\geq s_0$ and some constant $C'>0$.	
	This completes the proof.

\subsection{Lower Tail}\label{sec:LowerTail}
Our proof of Theorem~\ref{short:lowertail} will utilize ideas from \cite{flores2014strict}. In \cite{flores2014strict}, the author provided an upper bound to the lower tail probability of $\mathcal{H}^{\mathbf{nw}}$. However, it was not clear whether the same bound holds for $\sh_t$, i.e., centering $\mathcal{H}^{\mathbf{nw}}$ with $\log \sqrt{2\pi t}$ and scaling by $(\pi t/4)^{1/4}$. Our analysis will demonstrate that it is indeed possible to derive similar tail bound for $\sh_t$.    

The main tool of our proof of Theorem~\ref{short:lowertail} are some properties of the directed random polymer partition functions and its convergence to the solution of the SHE. Below, we introduce relevant notations. 

Let $\Xi:=\{\mathcal{E}(i,x):i\in\mathbb{N},x\in\mathbb{Z}\}$ be a collection of independent standard normal random variables. We call such collections as \emph{lattice environment}. Let $\{S_i\}_{i\ge 0}$ be a simple symmetric random walk on $\mathbb{Z}$ starting at $S_0=0$ independent of $\Xi$. Denote the law of $\{S_i\}_{i\ge 0}$ by $\mathbb{P}_{S}$. At inverse temperature $\beta>0$, the directed polymer partition function $Z_n^{(\Xi)}(\beta)$ is defined as
\begin{align*}
Z^{(\Xi)}_n(\beta)=\Ex_S\left[\exp\left\{{\beta\sum_{i=1}^n \mathcal{E}(i,S_i)}\right\}\ind_{S_n=0}\right].
\end{align*}
where the expectation $\mathbb{E}_{S}$ is taken w.r.t. $\mathbb{P}_{S}$. From \cite{alberts2014}, we know as $n\to \infty$
\begin{align}\label{eq:WC}
\frac{1}{(\pi t/2)^{1/4}}\left[\log\sqrt{\frac{n\pi}{2}}+\log\frac{Z_{n}^{(\Xi)}((t/2n)^{1/4})}{\Ex Z_{n}^{(\Xi)}((t/2n)^{1/4})}\right] \Rightarrow  \sh_t, \quad \text{for each }t>0
\end{align}
 Here, `$\Rightarrow$' denotes the weak convergence. To complete the proof of Theorem~\ref{short:lowertail}, we need the following two lemmas. Lemma~\ref{lem:Overlap} is originally from a part of the proof of Theorem~1.5 in \cite{carmona2002partition}. 

\begin{lemma}[Lemma~1 of \cite{flores2014strict}]\label{lem:Overlap}
 Let $\Xi$ and $\Xi^{\prime}$ be two independent lattice environments. Let $S^{(1)}$ and $S^{(2)}$ be  two independent simple symmetric random walks starting at origin. Denote the expectation w.r.t. the joint law of $S^{(1)}$ and $S^{(2)}$ by $\Ex_{S^{(1)},S^{(2)}}$.Then, we have 
	\begin{align*}
	\log Z_n^{(\Xi)}(\beta)\ge \log Z_n^{(\Xi^{\prime})}(\beta)-\beta d_n(\Xi,\Xi^{\prime})\sqrt{\mathfrak{Overlap}_{\Xi^{\prime}}(S^{(1)}, S^{(2)})}
	\end{align*}
	where $d_n(\Xi,\Xi^{\prime})^2:=\sum_{i=1}^n \sum_{|x|\le i}(\mathcal{E}(i,x)-\mathcal{E}^{\prime}(i,x))^2$ and, 
	\begin{align*}
	\mathfrak{Overlap}_{\Xi^{\prime}}(S^{(1)}, S^{(2)}): = \frac1{Z_n^{(\Xi^{\prime})}(\beta)^2}\Ex_{S^{(1)},S^{(2)}}\left[\sum_{i=1}^n \ind_{S_i^{(1)}=S_i^{(2)}}e^{\beta\sum_{i=1}^n(\mathcal{E}^{\prime}(i,S_i^{(1)})+\mathcal{E}^{\prime}(i,S_i^{(2)}))}\ind_{S_n^{(1)}=S_n^{(2)}=0}\right].
	\end{align*}	
\end{lemma}
The next lemma is similar to Lemma 2 of \cite{flores2014strict}. To state the lemma, we introduce for any $n\in \mathbb{N}$, $t>0$
 and $C>0$
 \begin{align*}
A_{n,t,C}:=\left\{\Xi^{\prime}:Z_n^{(\Xi^{\prime})}((t/2n)^{1/4})\ge \sqrt{\frac2{n\pi}}\Ex Z_n^{(\Xi)}((t/2n)^{1/4}), \mathfrak{Overlap}_{\Xi^{\prime}}(S^{(1)}, S^{(2)})\le C\sqrt{n}\right\}.
\end{align*}
\begin{lemma}\label{antc-lemma} For any given $\varepsilon>0$, there exist constants $t_0=t_0(\varepsilon)\in (0,2]$ and $C=C(\varepsilon)>0$ satisfying the following: for any $t\le t_0$, there exists $n_t\in\mathbb{N}$ such that for all $n\ge n_t$, we have $\Pr(A_{n,t,C})\ge \frac12-\varepsilon$.
\end{lemma} 
Our proof of the above lemma uses some of the ideas from the proof of Lemma~2 of \cite{flores2014strict}. However, there is a major difference between these two results. Unlike Lemma~2 of \cite{flores2014strict}, Lemma~\ref{antc-lemma} provides a lower bound to $\Pr(A_{n,t,C})$ which does not depend on $t$. On the other hand, the lower bound of Lemma 2 of \cite{flores2014strict} is valid for all $n\geq 1$ which is not the case in Lemma~\ref{antc-lemma}. Since we are interested in the evolution of tail probabilities of $Z_n^{(\Xi)}((t/2n)^{1/4})$ as $n$ grows large, the probability bound of $A_{n,t,C}$ for large $n$ is more relevant to our analysis than a uniform bound for all $n\geq 1$. Furthermore, the independence of the lower bound of $\Pr(A_{n,t,C})$ from $t$ enables us in Theorem~\ref{short:lowertail} to derive bounds on the lower tail probability of $\sh_t$ uniform in $t$. Before proceeding to the proof of Lemma~\ref{antc-lemma}, we will show Theorem~\ref{short:lowertail} by assuming Lemma~\ref{antc-lemma}.   

\subsubsection{Proof of Theorem~\ref{short:lowertail}}

Fix $\epsilon\in (0,\frac{1}{2})$. We choose $t_0=t_0(\varepsilon)\in (0,2]$ as defined in Lemma \ref{antc-lemma}. Fix $t\le t_0$. From Lemma \ref{antc-lemma} we pick $C>0$ and $n_t\in\N$ such that for all  $n\ge n_t$, $\Pr(A_{n,t,C})\ge \frac14$. Fix $n\ge n_t$. Consider any $\Xi^{\prime}\in A_{n,t,C}$. By Lemma~\ref{lem:Overlap}, we have 
	\begin{align*}
	\log Z_n^{(\Xi)}((t/2n)^{1/4}) & \ge \log Z_n^{(\Xi^{\prime})}((t/2n)^{1/4})-(t/2n)^{1/4} d_n(\Xi,\Xi^{\prime})\sqrt{\mathfrak{Overlap}_{\Xi^{\prime}}(S^{(1)}, S^{(2)})} \\ & \ge \log\sqrt{\frac2{n\pi}}+\log\Ex Z_n^{(\Xi)}((t/2n)^{1/4})-(t/2n)^{1/4} d_n(\Xi,\Xi^{\prime})\sqrt{C\sqrt{n}}.
	\end{align*}
	where the second inequality follows since $\Xi^{\prime}\in A_{n,t,C}$. 
	Rearranging the above inequality and using the fact that it holds for any $\Xi^{\prime}\in A_{n,t,C}$ shows 
	\begin{align*}
	\frac1{(\pi t/2)^{1/4}}\left[\log\sqrt{\frac{n\pi}2}+\log \frac{Z_n^{(\Xi)}((t/2n)^{1/4})}{\Ex Z_n^{(\Xi)}((t/2n)^{1/4})}\right]& \ge-C^{1/2}\pi^{-1/4} \inf_{\Xi^{\prime}\in A_{n,t,C} }d_n(\Xi,\Xi^{\prime}).
	\end{align*}
	Thus, for all $s>0$,
	\begin{align}\label{eq:dist}
	\Pr\left(\frac1{(\pi t/2)^{1/4}}\left[\log\sqrt{\frac{n\pi}2}+\log \frac{Z_n^{(\Xi)}((t/2n)^{1/4})}{\Ex Z_n^{(\Xi)}((t/2n)^{1/4})}\right] \le -s\right)  
	\le \Pr( d_n(\Xi,A_{n,t,C})\ge s\pi^{1/4}C^{-1/2})
	\end{align}
	where $d_n(\Xi,A_{n,t,C}):= \inf_{\Xi^{\prime}\in A_{n,t,C} }d_n(\Xi,\Xi^{\prime})$.
	Since $\Pr(A_{n,t,C})\ge \frac12 -\varepsilon$, applying Theorem~3 of \cite{flores2014strict} ( Talagrand's inequality) shows $\Pr(d_n(\Xi,A_{n,t,C})\ge u+\sqrt{4\log 2})\le e^{-u^2/2}$. Applying this probability bound into the right hand side of the above display yields 
	\begin{align}\label{eq:UnifBd}
	\mbox{r.h.s.~of }\eqref{eq:dist} & \le \exp\left({-\frac12\left\{s\pi^{-1/4}C^{-1/2}-\sqrt{4\log 2}\right\}^2}\right) \le e^{-cs^2}
	\end{align}
	for some positive constant $c>0$ and for all $s\ge s_0$ where neither $s_0$ nor $c$ does depend on $n$ or $t$. Due to the weak convergence of \eqref{eq:WC}, we have 
	\begin{align*}
	\Pr\left(\frac1{(\pi t/2)^{1/4}}\left[\log\sqrt{\frac{n\pi}2}+\log \frac{Z_n^{(\Xi)}((t/2n)^{1/4})}{\Ex Z_n^{(\Xi)}((t/2n)^{1/4})}\right] \le -s\right) \stackrel{n\to\infty}{\to} \Pr(\sh_t\le -s)
	\end{align*} 
	Combining this convergence with \eqref{eq:dist} and \eqref{eq:UnifBd} shows the desired conclusion.

\subsubsection{Proof of Lemma~\ref{antc-lemma}}

Define $\beta_n=(t/2n)^{1/4}$. By Proposition~ of \cite{ACQ11}, $\sh_t$ converges weakly to the standard Gaussian distribution implying $\lim\limits_{t\to 0}\Pr(\sh_t\ge 0)=\frac12.$ We choose the largest $t_0=t_0(\varepsilon)\in (0,2]$ such that $\Pr(\sh_t\ge 0)\ge \frac12-\frac{\varepsilon}{2}$ for all $t\le t_0$. For simplicity in notations, we set
	\begin{align*}
	\mathfrak{L}_n:=\sum_{i=1}^n \ind_{S_i^{(1)}=S_i^{(2)}}\cdot \ind_{S_n^{(1)}=S_n^{(2)}=0}\cdot e^{\beta_n\sum_{i=1}^n(\mathcal{E}^{\prime}(i,S_i^{(1)})+\mathcal{E}^{\prime}(i,S_i^{(2)}))}, \quad L_n=\sum_{i=1}^n \ind_{S_i^{(1)}=S_i^{(2)}}.
	\end{align*}
	Recall that $\mathfrak{Overlap}_{\Xi^{\prime}}(S^{(1)}, S^{(2)})$ is equal to $\mathbb{E}_{S^{(1)}, S^{(2)}}[\mathfrak{L}_n]/(Z_n^{(\Xi^{\prime})}(\beta))^2$. By simple probability bounds, we get
	\begin{align}
	\Pr(A_{n,t,C})   
	 & \ge  \Pr\left(Z_n^{(\Xi^{\prime})}(\beta_n)\ge \sqrt{\frac2{n\pi}}\Ex Z_n^{(\Xi)}(\beta_n), \Ex_{S^{(1)}S^{(2)}}(\mathfrak{L}_n)\le \frac{2C}{\sqrt{n\pi^2}}(\Ex Z_n^{(\Xi)}(\beta_n))^2\right) \nonumber\\ & \label{antc} \ge  \Pr\left(Z_n^{(\Xi^{\prime})}(\beta_n)\ge \sqrt{\frac2{n\pi}}\Ex Z_n^{(\Xi)}(\beta_n)\right)- \Pr\left( \frac{\Ex_{S^{(1)}S^{(2)}}(\mathfrak{L}_n)}{(\Ex Z_n^{(\Xi)}(\beta_n))^2}> \frac{2C}{\sqrt{n\pi^2}}\right)
	\end{align}
	We claim that for any $t\leq t_0$, there exists $n_t\in \mathbb{N}$ such that for all $n\geq n_t$,
	\begin{align}
	\Pr\left(Z_n^{(\Xi^{\prime})}(\beta_n)\ge \sqrt{\frac2{n\pi}}\Ex Z_n^{(\Xi)}(\beta_n)\right)\geq \frac{1}{2}-\frac{3\varepsilon}{4}, \quad \Pr\left( \frac{\Ex_{S^{(1)}S^{(2)}}(\mathfrak{L}_n)}{(\Ex Z_n^{(\Xi)}(\beta_n))^2}> \frac{2C}{\sqrt{n\pi^2}}\right)\leq \frac{\varepsilon}{4}.\label{eq:Lastwo}
	\end{align}
	Substituting the above inequalities into the right hand side of \eqref{antc} completes the proof of Lemma~\ref{antc-lemma}. Thus, it suffices to show that the above inequalities hold for all large $n$. To see the first inequality of \eqref{eq:Lastwo}, we first note that $\Ex Z_n^{(\Xi)}(\beta_n)= \Ex Z_n^{(\Xi^{\prime})}(\beta_n)$ and write
	\begin{align*}
	 \Pr\left(Z_n^{(\Xi^{\prime})}(\beta_n)\ge \sqrt{\frac2{n\pi}}\Ex Z_n^{(\Xi^{\prime})}(\beta_n)\right)  = \Pr\left(\frac1{(\pi t/2)^{\frac14}}\left[\log\sqrt{\frac{n\pi}2}+\log \frac{Z_n^{(\Xi^{\prime})}(\beta_n)}{\Ex Z_n^{(\Xi^{\prime})}(\beta_n)}\right]\ge 0\right). 
	\end{align*}
	By the weak convergence in \eqref{eq:WC} and $\mathbb{P}(\sh_t\geq 0)\geq \frac{1}{2}-\frac{\varepsilon}{2}$, it follows that the right side of the above display is greater than $\frac{1}{2}-\frac{3\epsilon}{4}$ for all large $n$. This proves the first inequality of \eqref{eq:Lastwo}.
	
	Now, we show the second inequality of \eqref{eq:Lastwo}. Note that $\Ex Z_n^{(\Xi)}(\beta_n)=e^{n\beta^2_n/2}$. By Fubini, we have 
	\begin{align*}
	\Ex_{\Xi} \Ex_{S^{(1)}S^{(2)}}[\mathfrak{L}_n]  & = \Ex_{S^{(1)}S^{(2)}}\Big[\sum_{i=1}^n \ind_{S_i^{(1)}=S_i^{(2)}}\cdot\ind_{S_n^{(1)}=S_n^{(2)}=0}\cdot \prod_{j=1}^n\Ex_{\Xi}\Big(e^{\beta_n(\mathcal{E}(j,S_j^{(1)})+\mathcal{E}(j,S_j^{(2)}))}\Big)\Big] \\ & = e^{n\beta^2_n}\Ex_{S^{(1)}S^{(2)}}\Big[\sum_{i=1}^n \ind_{S_i^{(1)}=S_i^{(2)}}\cdot \ind_{S_n^{(1)}=S_n^{(2)}=0}\cdot \exp\big(\beta^2_n\sum\limits_{i=1}^n \ind_{S_i^{(1)}=S_i^{(2)}}\big)\Big]
	 \\ & = (\Ex Z_n^{(\Xi)}(\beta)_n)^2\Ex_{S^{(1)}S^{(2)}}\big[L_ne^{\beta^2_nL_n}\ind_{S_n^{(1)}=S_n^{(2)}=0}\big]
	\end{align*}
	Applying Markov's inequality and using the above expression of $\Ex_{\Xi} \Ex_{S^{(1)}S^{(2)}}[\mathfrak{L}_n]$ shows 
	\begin{align*}
	\Pr\left( \frac{\Ex_{S^{(1)}S^{(2)}}(\mathfrak{L}_n)}{(\Ex Z_n^{(\Xi)}(\beta_n))^2}> \frac{2C}{\sqrt{n\pi^2}}\right) & \le \frac{\sqrt{n\pi^2}}{2C}\Ex_{S^{(1)}S^{(2)}}\left[L_ne^{\beta^2_nL_n}\ind_{S_n^{(1)}=S_n^{(2)}=0}\right]\\ & \le \frac{\sqrt{n\pi^2}}{2C}\Pr(S_n^{(1)}=S_n^{(2)}=0)\Ex_{S^{(1)}S^{(2)}}\left[L_ne^{\beta^2_nL_n}\mid S_n^{(1)}=S_n^{(2)}=0\right]
	\end{align*}
By Stirling's approximation, there exists constant $a>0$ such that $\Pr(S_n^{(1)}=S_n^{(2)}=0)=\frac{1}{2^{2n}}\binom{n}{n/2}^2 \le \frac{a}{n}$ for all $n$. Since $\beta_n=(t/2n)^{1/4}$, we have $L_ne^{\beta^2_nL_n} = L_ne^{(t/2n)^{1/2}L_n} \le L_ne^{n^{-1/2}L_n}$ for all $t\le t_0\le 2$. Furthermore, Lemma 3 in \cite{flores2014strict} proves
	\begin{align*}
	\sup_{N\ge 1}\frac{1}{\sqrt{n}}\Ex_{S^{(1)}S^{(2)}}\left[L_ne^{n^{-1/2}L_n}\mid S_n^{(1)}=S_n^{(2)}=0\right]=K<\infty.
	\end{align*} 
	Thus for all $t\le t_0$ we have a constant $K'>0$ (free of $t$) so that
	\begin{align*}
	\Pr\left( \frac{\Ex_{S^{(1)}S^{(2)}}(\mathfrak{L}_n)}{(\Ex Z_n^{(\Xi)}(\beta_n))^2}> \frac{2C}{\sqrt{n\pi^2}}\right) \le \frac{K'}{C}
	\end{align*}
	Letting $C$ large shows the second inequality of \eqref{eq:Lastwo} for all large $n$. This completes the proof.
	

\section{Tail Bounds of the KPZ Spatial Process} \label{sec:spatialtail}
In this section, we prove delicate tail bounds on several functionals of the long and short time spatial processes $\h_t(\cdot)$ and $\g_t(\cdot)$ respectively. Four propositions will be proved in this section; two of them are about the supremum and the infimum of the spatial process $\h_t$ and other two are devoted on similar results about $\sh_t$. One may notice similarities between Proposition~\ref{longtime:smdiffbd},~\ref{longtime:diffbd} and Theorem~1.3 of \cite{corwin2019kpz} since both bound the tail probabilities of the supremum and/or infimum of the KPZ height differences between spatial points. However, in comparison to \cite[Theorem~1.3]{corwin2019kpz}, the bounds on the tail probabilities in Proposition~\ref{longtime:smdiffbd} and \ref{longtime:diffbd} improve on multiple aspects (e.g., decay exponents) which turn out to be extremely useful for proving the results of Section~\ref{sec:diffprob}. The main ingredients of the proofs of this sections are: $(1)$ tail bounds from Section~\ref{sec:onepoint} and $(2)$ Brownian Gibbs property of the line ensemble discussed in Section~\ref{sec:tools}. From this time forth, we will denote complement of any set $\m{B}$ by $\neg \m{B}$.

\begin{proposition}\label{longtime:smdiffbd} Fix $\kappa>0$ and $\alpha \in[\frac32,2]$. There exist constant $c>0$, $t_0>0$ such that for all $t \ge t_0$ and $\beta\in (0,1]$ and $s\ge s_0(t_0)$ we have
	\begin{align}
	\label{inf-smdiffbd} \Pr\Big(\inf_{|y|\le \beta^{2\kappa}s^{2-\alpha}} (\h_t(y)-\h_t(0))\le -\frac{7}{8}\beta^{\kappa}s\Big) & \le e^{-cs^{\alpha}}.
	\end{align}
\end{proposition}

\begin{proof} Let us define 
	\begin{align*}
	\mathsf{A}:=\left\{ \inf_{y\in [0,\beta^{2\kappa}s^{2-\alpha}]} (\h_t(y)-\h_t(0))\le -\frac{7}{8}\beta^{\kappa}s\right\}, \ \ \mathsf{B}:= \left\{\h_t(\beta^{\kappa}s^{1-\frac{\alpha}3})-\h_t(0)\le -\frac{3s^{2\alpha/3}}{4}\right\}.
	\end{align*}
	We seek to show that $\Pr(\m{A})$ is bounded above by $\exp(-cs^{\alpha})$ for all large $s$ with some constant $c>0$. Observe that $\Pr(\m{A}) \le \Pr(\m{A} \cap \neg \m{B})+\Pr(\m{B})$. In what follows, we show that there exists $s_0=s_0(t_0)$, $c>0$ such that for all $s\geq s_0$ and $t>t_0$, 
	\begin{align*}
	\underbrace{\Pr(\m{B})\leq \exp(-cs^{\alpha})}_{(\mathbf{I})}, \qquad \underbrace{\Pr(\m{A}\cap \neg \m{B}) \leq  \exp(-cs^{\alpha})}_{(\mathbf{II})}.
	\end{align*}
	Combining $(\mathbf{I})$ and $(\mathbf{II})$ will bound $\Pr(\m{A})$. By repeating the same argument for the interval $[-\beta^{2\kappa}s^{2-\alpha},0]$, one can show $$\mathbb{P}\Big(\inf_{y\in [-\beta^{2\kappa}s^{2-\alpha},0]} (\h_t(y)-\h_t(0))\le -\frac{7}{8}\beta^{\kappa}s\Big)\leq e^{-cs^{\alpha}}. $$
 Combining this inequality with the upper bound on $\Pr(\m{A})$ will complete the proof of this proposition. Throughout the rest, we prove the inequalities $(\mathbf{I})$ and $(\mathbf{II})$.
\smallskip	
	
	We first show the inequality  $(\mathbf{I})$. Note that $\m{B}$ is contained in the union of $\{\h_t(\beta^{\kappa}s^{1-\alpha/3})\leq -5s^{2\alpha/3}/8\}$ and $\{\h_t(0)\geq s^{2\alpha/3}/8\}$. By the union bound, 
	\begin{align}\label{eq:BBound}
	\Pr(\m{B}) & \le \Pr\Big(\h_t(\beta^{\kappa}s^{1-\frac{\alpha}3})+ \frac{\beta^{2\kappa}s^{2-\frac{2\alpha}3}}{2}\le -\frac{5s^{2\alpha/3}}{8}+\frac{\beta^{2\kappa}s^{2-\frac{2\alpha}3}}{2}\Big)+\Pr\Big(\h_t(0)\ge \frac{s^{2\alpha/3}}{8}\Big)
	\end{align}
	Due to the stationarity, $\h_t(\beta^{\kappa}s^{1-\frac{\alpha}3})+ \beta^{2\kappa}s^{2-\frac{2\alpha}3}/2$ is same in distribution with $\h_t(0)$. Furthermore we have the inequality $-5s^{2\alpha/3}8+\beta^{2\kappa}s^{2-\frac{2\alpha}3}/2\leq -s^{2\alpha/3}/8$ because $\alpha\geq 3/2$ and $\beta \leq 1$. Combining we get 
	\begin{align*}
	\text{r.h.s. of \eqref{eq:BBound}}\leq \Pr\Big(\h_t(0)\le -\frac{s^{2\alpha/3}}{8}\Big)+\Pr\Big(\h_t(0)\ge \frac{s^{2\alpha/3}}{8}\Big).
	\end{align*}
 Using \eqref{eq:LowTailUp} of Proposition~\ref{onepointlowtail} and \eqref{eq:UpperTail} of Proposition~\ref{onepointuptail}, we bound $\Pr(\h_t(0)\le -s^{2\alpha/3}/8)$ and $\Pr(\h_t(0)\ge s^{2\alpha/3}/8)$ by $\exp(-cs^{\alpha})$ for some constant $c>0$. Substituting these bound into the right side of the above display yields $\Pr(\m{B})\le 2\exp(-cs^{\alpha})$. 
 \smallskip
 
  Next we show $(\mathbf{II})$. For this, we use the Brownian Gibbs Property of the KPZ line ensemble. Let us denote $\mathcal{I}_{s,\beta}:=(0,\beta^{2\kappa}s^{1-\alpha/3})$. Recall that $\h_t$ is the lowest indexed curve $
\h^{(1)}_t$ of the KPZ line ensemble $\{\h^{(n)}_{t}\}_{n\in \mathbb{N}}$.  
  Let $\mathcal{F}_s:=\mathcal{F}_{\mathbf{ext}}(\{1\},\mathcal{I}_{s,\beta})$ be the $\sigma$-algebra generated by $\{\h^{(1)}_t(x): x\in \mathbb{R}\backslash \mathcal{I}_{s,\beta}\}$ and $\{\h_t^{(n)}(x):x\in \mathbb{R}\}_{n\in \mathbb{N}_{\geq 2}}$. Note that $\neg \m{B}$ is measurable w.r.t. $\mathcal{F}_s$. Thus, we may write  
   \begin{align}\label{eq:Prob}
	\Pr\big(\m{A}\cap \neg \m{B}\big)=\Ex \left[\ind_{\neg \m{B}}\Ex[\ind_{\m{A}}|\mathcal{F}_s]\right]=\Ex \left[\ind_{\neg \m{B}}\Pr_{s}(\m{A})\right].
	\end{align}
	where $\Pr_s:=\Pr_{\mathbf{H}^{\mathrm{long}}_t}^{1,1,(0,\beta^{\kappa}s^{1-\frac{\alpha}3}),\h_t^{(1)}(0),\h_t^{(1)}(\beta^{\kappa}s^{1-\frac{\alpha}3}),+\infty,\h_t^{(2)}}$. By the monotone coupling (Lemma \ref{Coupling1}) $\Pr_s(\m{A})\le \Pr_{\mathbf{free}}(\m{A})$, where $\Pr_{\mathbf{free}}:=\Pr_{\mathbf{H}^{\mathrm{long}}_t}^{1,1,(0,\beta^{\kappa}s^{1-\frac{\alpha}3}),\h_t^{(1)}(0),\h_t^{(1)}(\beta^{\kappa}s^{1-\frac{\alpha}3}),+\infty,-\infty}$ is the law of a Brownian Bridge $\mathfrak{B}$ on $[0,\beta^{\kappa}s^{1-\frac{\alpha}3}]$ with $\mathfrak{B}(0):=\h_t(0)$ and $\mathfrak{B}(\beta^{\kappa}s^{1-\frac{\alpha}3}):=\h_t(\beta^{\kappa}s^{1-\frac{\alpha}3})$. Since $\beta \in (0,1]$ and $\alpha\geq 3/2$, we have $\beta^{\kappa}s^{1-\alpha/3}\geq \beta^{2\kappa} s^{2-\alpha}$. By the affine equivariance of the law of Brownian bridges,
$$\big\{\mathfrak{B}(x): x\in \mathcal{I}_{s,\beta}\big\} \stackrel{d}{=} \big\{\widetilde{\mathfrak{B}}(x)+ \frac{\h_t(\beta^{\kappa}s^{1-\frac{\alpha}3})-\h_t(0)}{\beta^{\kappa}s^{1-\frac{\alpha}3}}y: x\in \mathcal{I}_{s,\beta}\big\}$$
where $\widetilde{\mathfrak{B}}$ is a Brownian Bridge on $[0,\beta^{\kappa}s^{1-\frac{\alpha}3}]$ starting and ending at $0$. Combining these observations with \eqref{eq:Prob} shows  
	\begin{align*}
	\Pr(\m{A}\cap \neg \m{B})& \le \mathbb{E}[\ind_{\neg \m{B}}\Pr_{s}(\m{A})]
	\\ & = \Ex\Big[\ind_{\neg \m{B}}\Pr\big(\inf_{y\in [0,\beta^{2\kappa}s^{2-\alpha}]} \big[\widetilde{\mathfrak{B}}(y)+\frac{\h_t(\beta^{\kappa}s^{1-\frac{\alpha}3})-\h_t(0)}{\beta^{\kappa}s^{1-\frac{\alpha}3}}y\big]\le -\frac78\beta^{\kappa}s \big|\mathcal{F}_s\big)\Big] \\ & \le \Ex\Big[\ind_{\neg \m{B}}\Pr\big(\inf_{y\in [0,\beta^{2\kappa}s^{2-\alpha}]} \widetilde{\mathfrak{B}}(y)-\frac{3s^{2\alpha/3}\beta^{2\kappa}s^{2-\alpha}}{4\beta^{\kappa}s^{1-\frac{\alpha}3}}\le -\frac78\beta^{\kappa}s \big|\mathcal{F}_s\Big)\Big] 
	\\ & \le \Pr\Big(\inf_{y\in [0,\beta^{2\kappa}s^{2-\alpha}]} \widetilde{\mathfrak{B}}(y)\le -\frac{\beta^{\kappa}s}8 \Big)=\Pr\Big(\frac1{\beta^{\kappa}s^{1-\frac{\alpha}{2}} }\inf_{y\in [0,\beta^{2\kappa}s^{2-\alpha}]} \widetilde{\mathfrak{B}}(y)\le -\frac{s^{\alpha/2}}8 \Big)
	\end{align*}
	The inequality in the third line follows since 
	$$\inf_{y\in [0,\beta^{2\kappa}s^{2-\alpha}]} \Big\{\widetilde{\mathfrak{B}}(y)+\frac{\h_t(\beta^{\kappa}s^{1-\frac{\alpha}3})-\h_t(0)}{\beta^{\kappa}s^{1-\frac{\alpha}3}}y\Big\} \geq  \inf_{y\in [0,\beta^{2\kappa}s^{2-\alpha}]} \widetilde{\mathfrak{B}}(y) - \frac{3s^{2\alpha/3}\beta^{2\kappa}s^{2-\alpha}}{4\beta^{\kappa}s^{1-\frac{\alpha}3}} $$ on the event $\neg \m{B}$. 
 The next inequality follows by neglecting the indicator. The last probability is clearly bounded by $\exp(-cs^{\alpha})$ by tail estimates of Brownian motion. This proves $(\mathbf{II})$ and hence, completes the proof of this proposition. 
\end{proof}

%
\begin{proposition}\label{longtime:diffbd} Fix $\kappa>0$. There exist constant $c>0$, $t_0>0$ such that for all $t \ge t_0$ and $\beta\in (0,1]$ and $s\ge s_0(t_0)$ we have
	\begin{align}
	\label{sup-diffbd}\Pr\Big(\sup_{|y|\le \frac1{16}\beta^{2\kappa}\sqrt{s}} (\h_t(y)-\h_t(0))\ge \beta^{\kappa}s\Big) & \le e^{-cs^{3/2}}.
	\end{align}
\end{proposition}
\begin{proof}  
Let $\mathsf{Sup}_1$ and $\mathsf{Sup}_2$ be the supremum value of $\h_t(y)-\h_t(0)$ for $y \in [0,\frac1{16}\beta^{2\kappa}\sqrt{s}]$ and $y\in [-\frac1{16}\beta^{2\kappa}\sqrt{s},0]$ respectively. In what follows, we only bound $\Pr(\mathsf{Sup}_1\geq \beta^{\kappa}s)$. One can bound $\Pr(\mathsf{Sup}_2\geq \beta^{\kappa}s)$ analogously. Let $\chi$ be the infimum of $y$ in $[0,\frac1{16}\beta^{2\kappa}\sqrt{s}]$ such that $\h_t(y) - \h_t(0)\geq \beta^{\kappa}s$. If there is no such $y$, define $\chi$ to be $+\infty$. Note that $\Pr(\mathsf{Sup}_1\geq \beta^{\kappa}s)= \Pr(\chi\le \frac{1}{16}\beta^{2\kappa}\sqrt{s})$. We can write the event $\{\chi\le \frac{1}{16}\beta^{2\kappa}\sqrt{s}\}$ as a disjoint union of $\m{A}_1$ and $\m{A}_2$ which are defined as
$$\m{A}_1:=\Big\{\chi\le \frac{\beta^{2\kappa}\sqrt{s}}{16},\Big(\h_t(\chi)-\h_t\big(\frac{\beta^{2\kappa}\sqrt{s}}{16}\big)\Big)< \frac{\beta^{\kappa}s}{8}\Big\}, \quad \m{A}_2:=\Big\{\chi\le \frac{\beta^{2\kappa}\sqrt{s}}{16},\Big(\h_t(\chi)-\h_t\big(\frac{\beta^{2\kappa}\sqrt{s}}{16}\big)\Big)\geq  \frac{\beta^{\kappa}s}{8}\Big\}.$$  
In what follows, we show there exist $s_0=s_0(t_0)>0$ and constant $c>0$ such that for all $s\geq s_0$ and $t\geq t_0$, we have 
\begin{align}\label{eq:TwoIneq}
\Pr(A_1)\leq \exp(-cs^{3/2}), \qquad \qquad \Pr(A_2)\leq \frac{1}{2}\Pr(\chi\leq \frac{1}{16}\beta^{2\kappa}\sqrt{s})+ \exp(-cs^{3/2}).
\end{align}
Since $\Pr(\chi\le \frac{1}{16}\beta^{2\kappa}\sqrt{s})= \Pr(\m{A}_1)+ \Pr(\m{A}_2)$, combining the above two inequalities shows $2^{-1}\Pr(\chi\le \frac{1}{16}\beta^{2\kappa}\sqrt{s})\leq 2\exp(-cs^{\alpha})$. Thus, proving \eqref{sup-diffbd} boils down to showing \eqref{eq:TwoIneq}.
	
	We first prove $\Pr(\m{A}_1)\leq \exp(-cs^{3/2})$. By the continuity of the spatial process $\h_t(\cdot)$, we have $\h_t(\chi)=\h_t(0)+\beta^{\kappa}s$ on the event $\{\chi\le \frac{1}{16}\beta^{2\kappa}\sqrt{s}\}$. Thus
	\begin{align*}
	\Pr(A_1) & \le \Pr\Big(\h_t(0)-\h_t(\tfrac{\beta^{2\kappa}\sqrt{s}}{16})\le -\frac{7\beta^\kappa s}8\Big)\le \Pr\left(\inf_{y\in [0,\beta^{2\kappa}\sqrt{s}/16]}\left(\h_t(y)-\h_t(\tfrac{\beta^{2\kappa}\sqrt{s}}{16})\right)\le -\frac{7\beta^\kappa s}{8}\right)
	\end{align*}
	The right hand side of the above inequality is bounded by $\exp(-cs^{3/2})$ due to Proposition~\ref{longtime:smdiffbd} and the stationarity of spatial process $\h_t(x)+ \frac{x^2}{2}$. This proves the first inequality of \eqref{eq:TwoIneq}. 
	
	Now we turn to show the second inequality of \eqref{eq:TwoIneq}. Consider the following event 
 $$\m{B}:=\Big\{\h_t(0)\in [-s/4,s/4], \h_t(\beta^{\kappa}\sqrt{s})\in [-3s/4,s/4]\Big\}.$$
 	Observe that $\Pr(\m{A}_2) \le \Pr(\m{A}_2 \cap \m{B})+\Pr(\neg \m{B})$. By Proposition~\ref{onepointlowtail} and~\ref{onepointuptail}, we get $\Pr(\neg \m{B})\leq \exp(-cs^{3/2})$ for some constant $c>0$ and all large $s$ and $t$. 
 	It suffices to show 
 	\begin{align}\label{eq:A_2B}
 	\Pr(\m{A}_2\cap \m{B})\le 2^{-1}\Pr(\chi\leq \beta^{2\kappa}\sqrt{s}/16)
 	\end{align} 
 	which is proven below.
 	
 	To bound the $\Pr(\m{A}_2\cap \m{B})$ we use the strong Brownian Gibbs property of the KPZ line ensemble. Let $\mathcal{F}_s=\mathcal{F}_{\mathbf{ext}}(\{1\},(\chi,\beta^{\kappa}\sqrt{s}))$ be the $\sigma$-algebra generated by $\{\h_t^{(n)}(x)\}_{n\in \mathbb{N},x\in \mathbb{R}}$ outside $\{\h_t^{(1)}(x)\}_{x\in (\chi,\beta^{\kappa}\sqrt{s})}$. By the tower property of the conditional expectation, we have
	\begin{align}\label{eq:Tower}
	\Pr(\m{A}_2\cap \m{B})=\Ex \left[\ind_{\left\{\chi \le \frac1{16}\beta^{2\kappa}\sqrt{s}\right\} \cap \m{B}}\Ex(\ind_{\m{D}}|\mathcal{F}_s)\right]=\Ex \left[\ind_{\left\{\chi \le \frac1{16}\beta^{2\kappa}\sqrt{s}\right\} \cap \m{B}}\Pr_{s}(\m{D})\right].
	\end{align}
	where $\m{D}:=\{\h_t(\chi)-\h_t(\frac1{16}\beta^{2\kappa}\sqrt{s}) \ge \frac18\beta^{\kappa}s\}$ and $\Pr_s:=\Pr_{\mathbf{H}^{\mathrm{long}}_t}^{1,1,(\chi,\beta^{\kappa}\sqrt{s}),\h_t^{(1)}(\chi),\h_t^{(1)}(\beta^{\kappa}\sqrt{s}),+\infty,\h_t^{(2)}}$. We show that $\Pr_s(\m{D})\leq \frac12$ on the event $\{\chi\leq \frac{1}{16}\beta^{2\kappa}\sqrt{s}\}\cap \m{B}$. By Lemma~\ref{Coupling1}, $\Pr_s(\m{D})\le \Pr_{\mathbf{free}}(\m{D})$, where $\Pr_{\mathbf{free}}:=\Pr_{\mathbf{H}^{\mathrm{long}}_t}^{1,1,(\chi,\beta^{\kappa}\sqrt{s}),\h_t^{(1)}(\chi),\h_t^{(1)}(\beta^{\kappa}\sqrt{s}),+\infty,-\infty}$ is the law of a Brownian Bridge $\mathfrak{B}(\cdot)$ on $[\chi,\beta^{\kappa}\sqrt{s}]$ with $\mathfrak{B}(\chi):=\h_t(\chi)$ and $\mathfrak{B}(\beta^{\kappa}\sqrt{s}):=\h_t(\beta^{\kappa}\sqrt{s})$. Let us define $$\mathfrak{B}_{\mathrm{Interpole}}(y) =\frac{\beta^{\kappa}\sqrt{s}-y}{\beta^{\kappa}\sqrt{s}-\chi} \h_t(\chi) + \frac{y-\chi}{\beta^{\kappa}\sqrt{s}-\chi} \h_t(\beta^{\kappa}\sqrt{s}) \quad \text{for }y \in  [\chi,\beta^{\kappa}\sqrt{s}].$$
Note that $\mathfrak{B}_{\mathrm{Interpole}}(y)- \h_t(\chi)$ is equal to $(y-\chi)(\h_t(\beta^{\kappa}\sqrt{s})-\h_t(\chi))/(\beta^{\kappa}\sqrt{s}-\chi)$. On the event $\{\chi\leq \beta^{\kappa}s/16\}\cap \m{B}$, we have $\h_t(\chi)= \h_t(0)+\beta^{\kappa}s$ by the continuity of the spatial process $\h_t$ and hence, $\h_t(\beta^{\kappa}\sqrt{s})-\h_t(\chi)$ is bounded below by $\frac{-3s}{4}-\h_t(0)-\beta^{\kappa}s$ which is further lower bounded by $-2s$. This shows $\mathfrak{B}_{\mathrm{Interpole}}(\beta^{2\kappa}s/16)\geq \h_t(\chi)- \beta^{\kappa}s/8$. However, we know that $\mathfrak{B}(\beta^{2\kappa}s/16)\leq \h_t(\chi)- \beta^{\kappa}s/8$ on the event $\m{D}$. This shows $\mathfrak{B}(\beta^{2\kappa}s/16)\leq \mathfrak{B}_{\mathrm{Interpole}}(\beta^{2\kappa}s/16)$ on the event $\{\chi\leq \beta^{\kappa}s/16\}\cap \m{B}$. However, since $\mathfrak{B}$ is a Brownian bridge and $\mathfrak{B}_{\mathrm{Interpole}}$ is the linear interpolation of the end points of $\mathfrak{B}$, the probability of the event $\{\mathfrak{B}(\beta^{2\kappa}s/16)\leq \mathfrak{B}_{\mathrm{Interpole}}(\beta^{2\kappa}s/16)\}$ is equal to $1/2$. This implies  $\Pr_s(\m{D})\leq \frac12$ on the event $\{\chi \le\beta^{2\kappa}\sqrt{s}/16\} \cap \m{B}$. Substituting this bound into the right hand side of \eqref{eq:Tower} shows $\Pr(\m{A}_2\cap \m{B})\leq \Pr(\chi\leq \beta^{2\kappa}\sqrt{s}/16)/2$. This completes the proof.

\end{proof}

\begin{proposition}\label{shortsupbd} Fix $a\in \R$ and $\delta>0$. There exist $t_0\in (0,1)$ and an absolute constant $c>0$ such that for all $t\le t_0$, $s\ge s_0(t_0)$ satisfying $(|a|+|\delta|)^2-a^2\le \frac{s}{2^8}$, we have 
	\begin{align}\label{short:smallbd}
	\Pr\Big(\sup_{x\in[a,a+\delta]}\Big(\sh_t((4^3t/\pi^3)^{1/8}x)+\frac{x^2}{2}\Big)\ge s\Big)\le e^{-cs^{3/2}}.
	\end{align}
	\begin{align}\label{e:shortinfbd}
	\Pr\Big(\inf_{x\in[a,a+\delta]}\Big(\sh_t((4^3t/\pi^3)^{1/8}x)+\frac{x^2}{2}\Big)\le -s\Big)\le e^{-cs^{2}}+e^{-cs^2t^{-1/8}\delta^{-1}}.
	\end{align}
\end{proposition}
\begin{proof}  
We introduce the shorthand $\tilde{\g}_t(x):= \g_t((4^3t/\pi^3)^{1/8}x)$ which will be used throughout this proof. We divide the proof of this proposition in two stages. We prove \eqref{short:smallbd} and \eqref{e:shortinfbd} in \emph{Stage-1} and \emph{Stage-2} respectively. 
\smallskip 

\emph{Stage-1: Proof of \eqref{short:smallbd}.} Assume $[a,a+\delta]\subset \R_{\ge 0}$. Define 
\begin{align*}
{\m{C}}_{[a,a+\delta]}  :=\Big\{\sup_{x\in[a,a+\delta]}(\tilde{\g}_t(x)+\frac{x^2}{2})\ge s \Big\}, &\quad
\m{C}'_{[a,a+\delta]}  :=\Big\{\sup_{x\in[a,a+\delta]}(\tilde{\g}_t(x)-\tilde{\g}_t(a))\ge \frac{s}{4} \Big\}, \\
{\m{D}}_{w}  := \Big\{-\frac{s}{32}\le &\tilde{\g}_t(w)+\frac{w^2}{2}\le \frac{s}{32} \Big\}
\end{align*}
where $w\in \mathbb{R}$. We seek to show that $\Pr({\m{C}}_{[a,a+\delta]})\leq \exp(-cs^{3/2})$ for all large $s$ and small $t$. Combining the stationarity in $x$ of the process $\tilde{\sh}_t(x)+\frac{x^2}{2}$ (follows from Lemma~\ref{stationary}) with Corollary~\ref{short-uptail} and Theorem~\ref{short:lowertail} yields $\Pr(\neg\m{D}_{w})\le \exp(-cs^{3/2})$ for all $w\in \mathbb{R}$. This will be used throughout the proof.  
On the event $\m{C}_{[a,a+\delta]}\cap \m{D}_{a}$, there exists $x\in [a,a+\delta]$ such that
$$\tilde{\sh}_t(x) \geq s-\frac{x^2}{2}\geq s -\frac{(a+\delta)^2}2\geq \frac{31s}{32}+\tilde{\sh}_t(a)+\frac{a^2}{2}-\frac{(a+\delta)^2}{2}\geq \frac{s}{4}+\tilde{\sh}_t(a)$$
where the second inequality follows since $x\leq a+\delta$, the third inequality follows since $\tilde{\sh}_t(a)+\frac{a^2}{2}\leq s/32$ on $\m{D}_a$ and the last inequality holds since $(a+\delta)^2-a^2\leq s/2^{8}$. The above inequalities shows $\m{C}_{[a,a+\delta]}\cap \m{D}_{a}\subset \m{C}'_{[a,a+\delta]}$ which implies 
\begin{align*}
\Pr({\m{C}}_{[a,a+\delta]}) \le \Pr(\neg\m{D}_{a})+\Pr(\m{C}'_{[a,a+\delta]}).
\end{align*}
Recall that $\Pr(\neg\m{D}_{a})\le \exp(-cs^{3/2})$. To complete the proof, it suffices to show that $\Pr(\m{C}'_{[a,a+\delta]})\leq \exp(-cs^{3/2})$ for large $s$ and small $t$. This we do as follows.

Let $\sigma$ be the infimum of $y\in [a,a+\delta]$ such that $\tilde{\g}_t(y)-\tilde{\g}_t(a)\ge \frac{s}{4}$, with the understanding that $\sigma$ is equal to $+\infty$ if no such point exists. Let us define $\m{B}:= \{\tilde{\g}_t(a+\delta)-\tilde{\g}_t(\sigma)\le -\frac{s}{8}\}$ and write 
\begin{align*}
\Pr(\m{C}'_{[a,a+\delta]}) =\Pr(\sigma\le a+\delta) = \Pr(\{\sigma\le a+\delta\}\cap \m{B})+\Pr(\{\sigma\le a+\delta\}\cap \neg \m{B}) 
\end{align*}
 On the event $\{\sigma \le a+\delta\}$, we have $\tilde{\g}_t(\sigma)=\tilde{\g}_t(a)+\frac{s}{4}$. This implies $\tilde{\sh}_t(a+\delta)- \tilde{\sh}_t(a)=\tilde{\sh}_t(a+\delta)- \tilde{\sh}_t(\sigma)+ s/4\geq -s/8$ on $\{\sigma \le a+\delta\}\cap \neg\m{B}$ and hence, 
\begin{align}
\Pr(\{\sigma \le a+\delta\}\cap \neg\m{B}) &\le \Pr\big(\tilde{\g}_t(a+\delta)+\frac{(a+\delta)^2}2- \tilde{\g}_t(a)- \frac{a^2}2\geq -\frac{s}{8}\big)\nonumber\\
&\leq\Pr\big(\tilde{\g}_t(a+\delta)+\frac{(a+\delta)^2}2>\frac{s}{16}\big)+P\big(\tilde{\g}_t(a)+ \frac{a^2}2\le -\frac{s}{16}\big) \le \exp(-cs^{3/2}).\label{eq:LastIneq}
\end{align}
where the second inequality follows from the union bound and the last inequality follows by combining the stationarity of $\tilde{\sh}_t(x)+\frac{x^2}{2}$ with Corollary~\ref{short-uptail} and Theorem~\ref{short:lowertail}.   

 Now we proceed to bound $\Pr(\{\sigma \le a+\delta\}\cap \m{B})$. By the union bound, we have  
 \begin{equation}\label{eq:sigma}
 \Pr(\{\sigma \le a+\delta\}\cap \m{B}) \le \Pr(\{\sigma \le a+\delta\}\cap \m{B}\cap \m{D}_{a}\cap \m{D}_{a+4\delta})+\Pr(\neg \m{D}_{a})+\Pr(\neg \m{D}_{a+4\delta}). 
 \end{equation} 
 We know $\Pr(\neg \m{D}_{a})+\Pr(\neg \m{D}_{a+4\delta})$ is bounded above by $\exp(-cs^{3/2})$ for some constant $c>0$. In what follows, we show that  
 \begin{align}\label{eq:halfProb}
 \Pr(\{\sigma \le a+\delta\}\cap \m{B}\cap \m{D}_{a}\cap \m{D}_{a+4\delta})\leq \frac12 \Pr(\sigma \le a+\delta).
 \end{align} 
Combing this inequality with \eqref{eq:sigma} and \eqref{eq:LastIneq} show that $\Pr(\m{C}'_{[a,a+\delta]})\leq 2^{-1}\Pr(\m{C}'_{[a,a+\delta]})+ \exp(-cs^{3/2})$ for all large $s$ and small $t$. By simplifying this aforementioned inequality, we get the desired result. 

It remains to show \eqref{eq:halfProb} whose proof is similar to that of \eqref{eq:A_2B}. To avoid the repetition, we sketch the underlying idea without details. The main tool that we use is the Brownian Gibbs property of the short time KPZ line ensemble $\{\sh^{(n)}_t\}_{n\in \mathbb{N}}$ (Recall its definition from \eqref{eq:UsilonNDefd}). By the tower property, we write the left hand side of \eqref{eq:halfProb} as $\Ex[\ind_{\left\{\sigma \le a+\delta\right\}\cap \calD_{a}\cap \m{D}_{a+4\delta}} \Pr_s(\m{B})]$ where $$\Pr_s:=\Pr_{\mathbf{H}^{\mathrm{short}}_t}^{1,1,(4^3t/\pi^3)^{1/8}(\sigma,a+4\delta),\tilde{\g}_t^{(1)}((4^3t/\pi^3)^{1/8}\sigma),\tilde{\g}_t^{(1)}((4^3t/\pi^3)^{1/8}(a+4\delta)),+\infty,\tilde{\sh}_t^{(2)}}.$$ By monotone coupling, $\Pr_s(\m{B})\leq \Pr_{\mathbf{free}}(\m{B})$ where $\Pr_{\mathbf{free}}$ is the law of a free Brownian bridge between $(4^3t/\pi^3)^{1/8}\sigma$ and $(4^3t/\pi^3)^{1/8}(a+4\delta)$ with the value of the end points being $\tilde{\sh}_t(\sigma)$ and $\tilde{\sh}_t(a+4\delta)$. On the event $\{\sigma \le a+\delta\}\cap \m{D}_{a}\cap \m{D}_{a+4\delta}\cap \m{B}$, the value of the Brownian bridge at $(4^3t/\pi^3)^{1/8}(a+\delta)$ has to be lower than the value of the line joining two end points of the Brownian bridge. The probability of this is bounded by $1/2$ which shows $ \Pr_{{\mathbf{free}}}(\m{B})\leq 1/2$ on $\left\{\sigma \le a+\delta\right\}\cap \m{D}_{a}\cap \m{D}_{a+4\delta}$. Hence, we get $\Ex[\ind_{\left\{\sigma \le a+\delta\right\}\cap \m{D}_{a}\cap \m{D}_{a+4\delta}} \Pr_s(\m{B})]$ is less than $\Pr(\sigma\leq a+\delta)/2$. This shows \eqref{eq:halfProb} and hence, completes the proof of \eqref{short:smallbd}.

	\medskip

	\noindent \emph{Stage-2: Proof of \eqref{e:shortinfbd}.} Let us define the following two events:
	\begin{align*}
	\m{B}_{[a,a+\delta]}  =\Big\{\frac{a^2}{2}+\inf_{x\in[a,a+\delta]}\tilde{\sh}_t(x)\le -s \Big\},\qquad  	\m{E}_{w} : = \Big\{\tilde{\sh}_t(w)+ \frac{w^2}{2} \ge -\frac{s}{4}\Big\}
	\end{align*}
	for $w\in \mathbb{R}$. Note that  
\begin{align*}
\Pr\big(\inf_{x\in[a,a+\delta]}\big(\tilde{\sh}_t(x)+\frac{x^2}{2}\big)\le -s\big)\leq\Pr(\m{B}_{[a,a+\delta]}) \le \Pr(\neg\m{E}_a)+\Pr(\neg\m{E}_{a+\delta})+\Pr(\m{B}_{[a,a+\delta]}\cap \m{E}_a\cap \m{E}_{a+\delta}).
\end{align*}
Due to the spatial stationarity of the process $\tilde{\sh}_t(x)+x^2/2$ (see Lemma \ref{stationary}) and Theorem~\ref{short:lowertail}, we have $\Pr(\neg\m{E}_{a+\delta})=\Pr(\neg\m{E}_0)\le \exp(-cs^{2})$ for all large $s$ and small $t$. To complete the proof of \eqref{e:shortinfbd}, it suffices to show 
\begin{align}\label{eq:Intmed}
\Pr(\m{B}_{[a,a+\delta]}\cap \m{E}_a\cap \m{E}_{a+\delta})\leq \exp(-cs^2t^{-1/8}\delta^{-1}).
\end{align}


To show the above inequality, we use the Brownian-Gibbs property of the short time KPZ line ensemble. Recall from \eqref{def:kpz-short-scale} and \eqref{eq:UsilonNDefd} that $\left\{\tilde{\sh}_t((4^3t/\pi^3)^{-1/8}w)\right\}_{w\in\mathbb{R}}$ is same in distribution with $\g_{t}^{(1)}(\cdot)$ where $\g_{t}^{(1)}$ is the lowest indexed curve of the short-time KPZ line ensemble defined in (4) of Lemma \ref{line-ensemble}. Let us set $a':=(4^3t/\pi^3)^{1/8}a$ and $\delta':= (4^3t/\pi^3)^{1/8}\delta$ for convenience. Let $\mathcal{F}_s:=\mathcal{F}_{\mathbf{ext}}(\{1\},(a',a'+\delta'))$ be the $\sigma$-algebra generated by $\{\tilde{\g}_{t}^{(n)}(x)\}_{n\in \mathbb{N}_{\geq 2},x\in \mathbb{R}}$ outside $\{\tilde{\g}_{t}^{(1)}(x)\}_{x\in \mathbb{R}\backslash (a',a'+\delta')}$. Consider the following two measures
$$\Pr_s:=\Pr_{\mathbf{H}^{\mathrm{short}}_t}^{1,1,(a',a'+\delta'),\tilde{\sh}_t(a),\tilde{\sh}_t(a+\delta),\infty,\tilde{\g}_t^{(2)}},\quad \Pr_{\mathbf{free}}:=\Pr_{\mathbf{H}^{\mathrm{short}}_t}^{1,1,(a',a'+\delta'),\tilde{\sh}_t(a),\tilde{\sh}_t(a+\delta),\infty,-\infty}$$ where $\Pr_{\mathbf{free}}$ denotes the law of a Brownian bridge on $[a',a'+\delta']$ with the boundary values $\tilde{\sh}_t(a)$ and $\tilde{\sh}_t(a+\delta)$ respectively.
By the strong Brownian Gibbs property for the short-time KPZ line ensemble,
\begin{align*}
\Pr(\m{B}_{[a,a+\delta]}\cap\m{E}_{a}\cap\m{E}_{a+\delta})=\Ex \left[\ind_{\m{E}_{a}}\ind_{\m{E}_{a+\delta}}\Ex(\m{B}_{[a,a+\delta]}|\mathcal{F}_s)\right]=\Ex \left[\ind_{\m{E}_{a}}\ind_{\m{E}_{a+\delta}}\Pr_{s}(\m{B}_{[a, a+\delta]})\right].
\end{align*}
 Due to the monotone coupling, we know $\Pr_s(\m{B}_{[a,a+\delta]})\le \Pr_{\mathbf{free}}(\m{B}_{[a,a+\delta]})$. Let $\mathfrak{B}$ be a Brownian bridge on $[0,\delta']$ with $\mathfrak{B}(0)=\mathfrak{B}(\delta')=0$. Then, the law of $\mathfrak{B}(x)+\tilde{\sh}_t(a)\frac{\delta'-x}{\delta'}+\tilde{\sh}_t(a+\delta)\frac{x}{\delta'}$ is same as $\Pr_{\mathbf{free}}$. So, we have 
\begin{align}
& \Pr(\m{E}_{a}\cap\m{E}_{a+\delta}\cap\m{B}_{[a,a+\delta]})\nonumber \\ & \le \Ex\Big[\ind_{\m{E}_{a}}\ind_{\m{E}_{a+\delta}} \Pr\Big(\frac{a^2}{2}+\inf_{x\in [0,\delta']} \big(\mathfrak{B}(x)+\tilde{\sh}_t(a)\frac{\delta'-x}{\delta'}+\tilde{\sh}_t(a+\delta)\frac{x}{\delta'}\big)\le -s\Big)\Big] \nonumber\\ & \label{e0}\le \Ex\Big[\ind_{\m{E}_{a}}\ind_{\m{E}_{a+\delta}} \Pr\Big(\frac{a^2}{2}+\inf_{x\in [0,\delta']} \big[\mathfrak{B}(x)+\big(-\frac{s}{4}-\frac{a^2}{2}\big)\frac{\delta'-x}{\delta'}+\big(-\frac{s}{4}-\frac{(a+\delta)^2}{2}\big)\frac{x}{\delta'}\big]\le -s\Big)\Big] \\ & \label{e1} \le \Pr\Big(\frac{a^2-(a+\delta)^2}{2}-\frac{s}{4}+\inf_{x\in [0,\delta']} \big[\mathfrak{B}(x)+\frac{[(a+\delta)^2-a^2](\delta'-x)}{2\delta'}\big]\le -s\Big). 
\end{align} 
The inequality in \eqref{e0} follows by noting that $\tilde{\sh}_t(a)+a^2/2$ and $\tilde{\sh}_t(a+\delta)+(a+\delta)^2/2$ are at least $-s/4$ on the event on $(\m{E}_a\cap \m{E}_{a+\delta})$. The last inequality in \eqref{e1} follows by dropping the indicators $\ind_{\m{E}_{a}}$ and $\ind_{\m{E}_{a+\delta}}$ from inside the expectation. 
Recall that $(|a|+|\delta|)^2 - a^2 \leq s/2^8$. Using this inequality to bound in the last line of the above display yields 
\begin{align}\label{eq:e2}
\text{r.h.s. of \eqref{e1}}\leq\Pr\Big(\inf_{x\in [0,\delta']} \big[\mathfrak{B}(x)+\frac{[(a+\delta)^2-a^2](\delta'-x)}{2\delta'}\big]\le -\frac{3s}{4}+\frac{s}{2^9}\Big). 
\end{align}

We seek to bound the right hand side of the inequality. To this end, we divide the rest of the analysis into three cases: \emph{Case:} (a) when $[a,a+\delta]\subset \mathbb{R}_{\geq 0}$, \emph{Case:} (b) $[a,a+\delta]\subset \mathbb{R}_{\leq 0}$ and, \emph{Case:} (c) when both $[a,a+\delta]\cap \mathbb{R}_{> 0}$ and $[a,a+\delta]\cap \mathbb{R}_{<0}$ are nonempty. 
\smallskip 

\noindent \emph{Case:} (a) 
Note that the drift term of the Brownian bridge in the above display is always positive when $[a,a+\delta]\subset \mathbb{R}_{\geq 0}$. Ignoring the drift term of the Brownian bridge in \eqref{eq:e2} and upper bounding $-\tfrac{3s}{4}+ \tfrac{s}{2^9}$ by $-\frac{s}{4}$, we get
$$\text{r.h.s. of \eqref{eq:e2}}\leq \mathbb{P}\Big(\inf_{x\in [0,\delta^\prime]} \mathfrak{B}(x)\leq -\frac{s}{4}\Big)=   \Pr\Big(\inf_{x\in [0,1]} \widetilde{\mathfrak{B}}(x)\le -\frac{s}{4\sqrt{\delta'}}\Big) \le \exp(-cs^2/\delta')$$
 Here, $\widetilde{\mathfrak{B}}$ is a Brownian bridge on $[0,1]$ with $\widetilde{\mathfrak{B}}(0)= \widetilde{\mathfrak{B}}(1)=0$. The equality in the above display follows from the scale invariance property of the Brownian bridge. The last inequality is obtained by bounding the tail probability of the infimum of a Brownian bridge using reflection principle. Noting that $\delta'\le 2t^{1/8}\delta$, we get \eqref{eq:Intmed} from \eqref{eq:e2}  and hence, obtain \eqref{short:smallbd} when $[a,a+\delta]\subset \R_{\ge 0}$. 
\smallskip
 
\noindent \emph{Case:} (b) The drift term of the Brownian bridge in \eqref{eq:e2} is negative. Nevertheless, the absolute value of the drift term is bounded above by $s/{2^9}$. Adjusting the bound on the drift term in \eqref{eq:e2}, we get 
$$\text{r.h.s. of \eqref{e1}}\leq \mathbb{P}\Big(\inf_{x\in [0,\delta^\prime]} \mathfrak{B}(x)\leq -\frac{3s}{4}+ \frac{s}{2^8}\Big).$$
From the above inequality, the proof of \eqref{eq:Intmed} follows using the similar argument as in \emph{Case} (a). This completes the proof of \eqref{short:smallbd} when  $[a,a+\delta]\subset \R_{\le 0}$.
\smallskip
 
\noindent \emph{Case:} (c) For this case, the drift term of the Brownian bridge could be positive or, negative. When the drift is positive (i.e., $|a+\delta|>|a|$), one can complete the proof of \eqref{eq:Intmed} (and consequently, \eqref{short:smallbd}) using similar argument as in \emph{Case} (a). When the drift is negative, one can use similar argument as in \emph{Case} (b). This completes the proof.  
\end{proof}
Our next and final proposition of this section bounds the tail probabilities of the supremum and infimum of the spatial process $\sh_t(x)+(\pi t/4)^{3/4}x^2/(2t)$ as $x$ varies in $\R$. Proof of this proposition is similar to that of Proposition 4.1 and 4.2 of \cite{corwin2019kpz}. These results proved tail probability bound for the supremum and infimum of the spatial process $\h_t(x)+x^2/2$. The key tools for the proof of those propositions were the one point tail probabilities of $\h_t$ and the Brownian Gibbs property of the long time KPZ line ensemble. In a similar way, proving the following proposition would require one point tail probabilities of $\sh_t$ from Corollary~\ref{short-uptail} and Theorem~\ref{short:lowertail} and the Brownian Gibbs property of the short-time KPZ line ensemble. For brevity, we state the result without giving its proof.

\begin{proposition}\label{sup-proc} Let $\nu>0$. There exist $t_0=t_0(\nu)\in (0,1)$, $c=c(\nu)>0$ and $s=s(\nu)>0$ such that for all $t\le t_0$ and $s\ge s_0$, we have 
\begin{align*}
\Pr\Big(\sup_{x\in\R} \big(\sh_t(x)+\frac{(\pi t/4)^{3/4}(1-\nu)x^2}{2t}\big)\ge s\Big)  \le \exp(-cs^{3/2}), \\ \Pr\Big(\inf_{x\in\R} \big(\sh_t(x)+\frac{(\pi t/4)^{3/4}(1+\nu)x^2}{2t}\big)\le -s\Big)  \le \exp(-cs^{2}). 
\end{align*}
\end{proposition}

\section{Spatio-Tempral Modulus of Continuity} \label{sec:diffprob}
 The main goal of this section is to study the temporal modulus of continuity of the KPZ equation and use it for proving Theorem~\ref{thm:ModCont}. The proof of Theorem~\ref{thm:ModCont} requires detailed study of the tail probabilities for difference of the KPZ height function at two distinct time points. This will be explored in Proposition \ref{shortdiff} and~\ref{longdiff}. In particular, Proposition~\ref{shortdiff} will study the tail estimates when two time points are close to each other and Proposition~\ref{longdiff} will focus on the case when the time points are far apart. With these result in hand, we show the H\"older continuity of the sample path of $\h_t$ in Proposition~\ref{temp-modulus}. Below, we first state those propositions; prove Theorem~\ref{thm:ModCont}; and then, complete proving those proposition in three ensuing subsections.  
\smallskip

\begin{proposition}\label{shortdiff} Fix $\e\in (0,\frac14)$. There exist $t_0=t_0(\e) \ge 1, c=c(\e)>0$, and $s_0=s_0(\e)> 0$ such that for all $t \ge t_0$, $s\ge s_0$ and $\beta \le (0,1]$ satisfying $\beta t\le \frac1{t_0}$, we have
	\begin{align} \label{sd-up}
	\Pr(\h_t(1+\beta,0)-\h_t(1,0)\ge \beta^{1/4-\e}s) & \le \exp(-cs^{3/2}), \\ \label{sd-low}
	\Pr(\h_t(1+\beta,0)-\h_t(1,0)\le -\beta^{1/4-\e}s) & \le \exp(-cs^{2}).
	\end{align}
\end{proposition}

\begin{proposition}\label{longdiff} Fix $t_0>0$. There exist $c=c(t_0)>0$, and $s_0=s_0(t_0)> 0$ such that for all $t \ge t_0$ satisfying $\beta t\ge t_0$, $\beta\in (0,1]$ and $s\ge s_0$, 
	\begin{align} \label{ld-up}
	\Pr(\h_t(1+\beta,0)-\h_t(1,0)\ge \beta^{1/4}s) & \le \exp(-cs^{3/2}), \\ \label{ld-low}
	 \Pr(\h_t(1+\beta,0)-\h_t(1,0)\le -\beta^{1/4}s) & \le \exp(-cs^{2}).
	\end{align}
\end{proposition}

\begin{remark} 
Note that Proposition~\ref{shortdiff} and~\ref{longdiff} together bounds the upper and lower tail probabilities of the difference of the KPZ height function at any two time points irrespective of their distance. This is in sharp contrast with Theorem~1.5 of \cite{corwin2019kpz} which was able to prove some tail bounds of the KPZ height difference only under the assumption that the two associated time points are far apart. While Proposition~\ref{longdiff} may appear to share the same spirit as \cite[Theorem~1.5]{corwin2019kpz} since they both work under the assumption of the time points being distant from each other, however, the tail bounds of Proposition~\ref{longdiff} (see \eqref{ld-up} and \eqref{ld-low}) improve on the decay exponents in comparison with those in \cite{corwin2019kpz}. That being said, we expect that same tail bounds as in \eqref{ld-up} and \eqref{ld-low} hold even when the exponent of $\beta$ is $\tfrac{1}{3}$ instead of $\tfrac{1}{4}$. Nevertheless, the present tail bounds of Proposition~\ref{shortdiff} and~\ref{longdiff} are sufficient for proving main results of this paper. 
\end{remark}

Proposition~\ref{shortdiff} and~\ref{longdiff} will be proved in Section~\ref{sec:Prop5.1} and~\ref{sec:Prop5.2} respectively. The following proposition is in the same vein as Proposition~\ref{longdiff}.

\begin{proposition}\label{lngdiff} Fix $t_0>0$. For any given $\beta>0$, recall the spatial process $\h_{(1+\beta)t\downarrow t}(\cdot)$ from Proposition~\ref{ppn:MulpointComposition}. There exist $c=c(t_0)>0$, and $s_0=s_0(t_0)> 0$ such that for all $t \ge t_0$, $s\ge s_0, \beta\ge 1$ with $t\ge t_0$ we have
	\begin{align*}
	\Pr(\h_t(1+\beta,0)-\h_{(t+\beta t)\downarrow t}(0)\ge s) & \le \exp(-cs^{3/2}) \\ 
	\Pr(\h_t(1+\beta,0)-\h_{(t+\beta t)\downarrow t}(0) \le -s) & \le \exp(-cs^{2}).
	\end{align*}
\end{proposition}
 The proofs of Proposition~\ref{longdiff} and Proposition~\ref{lngdiff}, both use the representation $h_t(1+\beta,0)= I_t(\h_t, \h_{(1+\beta)t\downarrow t})$ (see the definition of $I_t$ in \eqref{eq:Icomp}). In fact, the proof of Proposition~\ref{lngdiff} is ditto to that of Proposition~\ref{longdiff} upto switching the role of $\h_t$ and $\h_{(1+\beta)t\downarrow t}$. With the aforementioned switching, the rest of the argument can be carried out exactly in the same way thanks to the fact that the spatial process $\h_{(1+\beta)\downarrow t}(\cdot)$ has the same law as $\h_t(\beta,0)$.  For avoiding repetitions, we will only prove Proposition~\ref{longdiff} and skip the details of the proof of Proposition~\ref{lngdiff}.

Our next result which will proved in Section~\ref{sec:Prop5.4} is on the tail bounds of the modulus of continuity of the KPZ temporal process.
%


\begin{proposition}[Temporal moulus of continuity]\label{temp-modulus} Fix $\e\in (0,\frac14)$. There exist $t_0=t_0(\e), s_0=s_0(\e)>0$ and $c=c(\e)>0$, such that for all $a,t\ge 0$ with $at\ge t_0$ and $s\ge s_0$,
	\begin{align}\label{temp-sup}
	\Pr\left(\sup_{\tau\in [0,a]} \frac{\h_t(a+\tau,0)-\h_t(a,0)}{(\tau/a)^{\frac14-\e}\log^{2/3}\frac{a}{\tau}} \ge a^{1/3}s\right) & \le e^{-cs^{3/2}}, \\ \label{temp-inf}
	\Pr\left(\inf_{\tau\in [0,a]} \frac{\h_t(a+\tau,0)-\h_t(a,0)}{(\tau/a)^{\frac14-\e}\log^{1/2}\frac{a}{\tau}} \le -a^{1/3}s\right) & \le e^{-cs^{2}}.
	\end{align}
\end{proposition}

\subsection{Proof of Theorem~\ref{thm:ModCont}}
 Our proof of Theorem~\ref{thm:ModCont} is built upon Proposition~\ref{temp-modulus} and the spatial modulus of continuity result of the KPZ equation from Theorem~1.4 of \cite{corwin2019kpz}. The mainstay of the proof can be divided into two  conceptual steps: the first step is to construct a dyadic mesh of spatio-temporal points  in $[a,b]\times [c,d]$ (recall $[a,b]$, $[c,d]$ from Theorem~\ref{thm:ModCont}) and prove the modulus of continuity result over those discrete set of points. This last part will achieved through Proposition~\ref{temp-modulus} and Theorem~1.4 of \cite{corwin2019kpz}.  The second step is to extend the modulus of continuity over the dyadic mesh to the set of all spatio-temporal pair of points from $[a,b]\times [c,d]$. Although such proofs are standard in the literature, we present here a full argument for the sake of completeness. 
 
  Consider the dyadic partition $\big\{\cup^{2^{n}}_{k_1,k_2=1}\mathcal{J}^{(n)}_{k_1,k_2}\big\}_{n\in \mathbb{N}}$ of $[a,b]\times [c,d]$ as 
 \begin{align*}
 \mathcal{J}^{(n)}_{k_1,k_2} = [\alpha^{(n)}_{k_1-1}, \alpha^{(n)}_{k_1}]\times [x^{(n)}_{k_2-1}, x^{(n)}_{k_2}], \quad \alpha^{(n)}_{k} = a+ \frac{k}{2^{n}}(b-a),  x^{(n)}_{k} = c+ \frac{k}{2^{n}}(d-c), \quad k=1,2,\ldots , 2^{n}-1.
 \end{align*}
 We introduce the following shorthand notations: 
 \begin{align*}
 \mathfrak{h}^{\nabla,\epsilon_1,\epsilon_2}_{t, k_1,k_2} & := \h_{t}(\alpha^{(n)}_{k_1+\epsilon_1},x^{(n)}_{k_1+\epsilon_2})+\frac{(x^{(n)}_{k_2+\epsilon_1})^2}{2\alpha^{(n)}_{k_1+\epsilon_2}} - \h_{t}(\alpha^{(n)}_{k_1},x^{(n)}_{k_2})- \frac{(x^{(n)}_{k_2})^2}{2\alpha^{(n)}_{k_1}} \\
 \mathrm{Sup}\mathfrak{h}^{\nabla,\epsilon_1,\epsilon_2}_{t, k_1,k_2} &:= \sup_{(\alpha,x)\in \mathcal{J}^{(n)}_{k_1,k_2}}\Big|\h_{t}(\alpha,x)+\frac{x^2}{2\alpha} - \h_{t}(\alpha^{(n)}_{k_1},x^{(n)}_{k_2})- \frac{(x^{(n)}_{k_2})^2}{2\alpha^{(n)}_{k_1}}\Big|
 \end{align*}
 for any $k_1,k_2= 1,\ldots , 2^{n}$ and $\epsilon_1,\epsilon_2\in \{-1,0,1\}$. 
 
 We consider the following event
 \begin{align*}
 \mathfrak{A}_{\mathrm{up}}:= \bigcup_{n=1}^{\infty}\bigcup_{k_1,k_2= 1}^{2^n} \Big\{\mathrm{Sup}\mathfrak{h}^{\nabla,1,1}_{t,k_1,k_2}\geq s\big((d-c)^{1/2}2^{-\frac{n}{2}}+ (b-a)^{\frac{1}{4}-\varepsilon}2^{-n(\frac{1}{4}-\varepsilon)}\big)\big(n\log 2\big)^{\frac{2}{3}}\Big\}
\end{align*}  

 By the union bound, we write 
 \begin{align}
 \mathbb{P}(\mathfrak{A}_{\mathrm{up}}) \leq \sum_{n=1}^{\infty}\sum_{k_1,k_2=1}^{2^n} \mathbb{P}\big(\mathrm{Sup}\mathfrak{h}^{\nabla,1,1}_{t,k_1,k_2}\geq s\big((d-c)^{1/2}2^{-\frac{n}{2}}+ (b-a)^{\frac{1}{4}-\varepsilon}2^{-n(\frac{1}{4}-\varepsilon)}\big)\big(n\log 2\big)^{\frac{2}{3}}\big).  \label{eq:UnionBd}
 \end{align}
 Below, we claim and prove that there exist constant $c>0$ and $s_0>0$ such that for all $s>s_0$, $n\in \mathbb{N}$, $k\in \{1,\ldots , 2^{n}\}$ and $\epsilon_1, \epsilon_2\in \{-1,0,1\}$, 
 \begin{align}\label{eq:SupBd}
 \mathbb{P}\big(\mathrm{Sup}\mathfrak{h}^{\nabla,\epsilon_1,\epsilon_2}_{t,k_1,k_2}\geq s\big((d-c)^{1/2}2^{-\frac{n}{2}}+ (b-a)^{\frac{1}{4}-\varepsilon}2^{-n(\frac{1}{4}-\varepsilon)}\big)\big(n\log 2\big)^{\frac{2}{3}}\big) \leq e^{-cs^{3/2}}
 \end{align}
Before proceeding to the proof of \eqref{eq:SupBd}, we first complete the proof of Theorem~\ref{thm:ModCont} by assuming \eqref{eq:SupBd}. Substituting the probability bound of \eqref{eq:SupBd} into \eqref{eq:UnionBd} and summing shows
\begin{align}\label{eq:AupBd}
\mathbb{P}(\mathfrak{A}_{\mathrm{up}})\leq \sum_{n=1}^{\infty}2^{2n}(2\times e^{-ncs^{3/2}\log 2}) = \sum_{n=1}^{\infty}  2^{-n(cs^{3/2}-2)}\leq e^{-c's^{3/2}}, \quad \forall s>(3/c)^{3/2}
\end{align} 
for some constant $c'>0$. Recall the definition of $\mathcal{C}$ from \eqref{eq:Cgg}. In light of \eqref{eq:AupBd}, the proof of the tail bound of $\mathcal{C}$ in \eqref{eq:CIneq} follows if one can show that $$\{\mathcal{C}\geq Ks\}\subset \mathfrak{A}_{\mathrm{up}}$$ for all large $s$ and some constant $K>0$. This is shown as follows. 

We will prove $\neg \mathfrak{A}_{\mathrm{up}}\subset \{\mathcal{C}\geq Ks\}$. For any given $\alpha_1<\alpha_2$ and $x_1<x_2$, there exists $n_0$ such that $(b-a)2^{-n_0-1}\leq \alpha_2- \alpha_1\leq (b-a) 2^{-n_0}$ and $(d-c)2^{-n_0-1}\leq x_2- x_1\leq (d-c) 2^{-n_0}$. This tells that $\alpha_1,\alpha_2$ may belong to the same dyadic interval $[\alpha^{(n_0)}_{k},\alpha^{(n_0)}_{k+1}]$ or, they belong to the two consecutive dyadic intervals $[\alpha^{(n_0)}_{k-1},\alpha^{(n_0)}_{k}]$
and $[\alpha^{(n_0)}_{k},\alpha^{(n_0)}_{k+1}]$. Similarly, there are two cases possible for $x_1,x_2$. Combination of these two sets of possibilities yields $4$ different cases. We only focus on the case when $\alpha_1,\alpha_2\in [\alpha^{(n_0)}_{k_1}, \alpha^{(n_0)}_{k_1+1}]$ and $x_1,x_2\in [x^{(n_0)}_{k_2}, x^{(n_0)}_{k_2+1}]$ for some $k_1,k_2$. The following conclusion in other cases would be same. Note that 
\begin{align}
\big|&\h_{t}(\alpha_1,x_1) + \frac{x^2_1}{2\alpha_1} - \h_{t}(\alpha_2,x_2) - \frac{x^2_2}{2\alpha_2} \big|\leq 2\sup_{\epsilon_1,\epsilon_2\in \{-1,0,1\}}\mathrm{Sup}\mathfrak{h}^{\nabla,\epsilon_1,\epsilon_2}_{t,k_1,k_2}\label{eq:TwoDiff}
\end{align}
On the event $\neg\mathfrak{U}_{\mathrm{up}}$, we have 
\begin{align}
\text{r.h.s. of \eqref{eq:TwoDiff}} &\leq  2s \big((d-c)^{1/2}2^{-\frac{n}{2}}+ (b-a)^{\frac{1}{4}-\varepsilon}2^{-n(\frac{1}{4}-\varepsilon)}\big)\big(n\log 2\big)^{\frac{2}{3}}\nonumber\\
& \leq 2Cs \Big((\alpha_2-\alpha_1)^{\frac{1}{2}}\big(\log \frac{(b-a)}{(\alpha_2-\alpha_1)}\big)^{\frac{2}{3}}+ (x_2-x_1)^{\frac{1}{4}-\varepsilon}\big(\log \frac{(d-c)}{(x_2-x_1)}\big)^{\frac{2}{3}}\Big)\nonumber\\
& = 2Cs \mathrm{Norm}(\alpha_1,x_1:\alpha_2,x_2) \label{eq:LastDis}
\end{align}
for some constant $C>0$ where the definition of $\mathrm{Norm}(\alpha_1,x_1:\alpha_2,x_2)$ can be found in \eqref{eq:NormCgg}. The second inequality in the above display follows since $(b-a) 2^{-n_0-1}\leq \alpha_2- \alpha_1 $ and $ (d-c) 2^{-n_0}\leq x_2- x_1$. Recall $\mathcal{C}$ from \eqref{eq:Cgg}. Substituting \eqref{eq:LastDis} into \eqref{eq:TwoDiff} and recalling the definition $\mathcal{C}$ yields $\neg \mathfrak{A}_{\mathrm{up}}\subset \neg \{\mathcal{C}\geq 2Cs\}$. Combining this with \eqref{eq:AupBd} shows  \eqref{eq:CIneq}. This proves  Theorem~\ref{thm:ModCont}. It remains to show \eqref{eq:SupBd} which we proves as follows. 

Fix $n\in \mathbb{N}$ and $k_1,k_2 \in \{1,2,\ldots , 2^{n}\}$. Consider the event $\tilde{\mathfrak{A}}_{\mathrm{up}}$ defined as 
\begin{align*}
\tilde{\mathfrak{A}}_{\mathrm{up}}:= \bigcup_{m=n+1}^{\infty}\bigcup_{k'_1= k_1}^{2^{m-n}+k_1}\bigcup_{k'_2= k_2}^{2^{m-n}+k_2} \Big\{\mathfrak{h}^{\nabla,1,1}_{t,k'_1,k'_2}\geq s\big((d-c)^{1/2}2^{-\frac{n}{2}}+ (b-a)^{\frac{1}{4}-\varepsilon}2^{-n(\frac{1}{4}-\varepsilon)}\big)\big(n\log 2\big)^{\frac{2}{3}}\Big\}
\end{align*}
By the triangle inequality, we know $|\mathfrak{h}^{\nabla,1,1}_{t,k'_1,k'_2}|\leq |\mathfrak{h}^{\nabla,1,0}_{t,k'_1,k'_2}|+ |\mathfrak{h}^{\nabla,0,1}_{t,k'_1,k'_2}|$. From Theorem 1.4 of \cite{corwin2019kpz} and Proposition~\ref{shortdiff} respectively, we have 
 \begin{align*}
 \mathbb{P}\big(|\mathfrak{h}^{\nabla,-1,0}_{t,k'_1,k'_2}|\geq s(b-a)^{\frac{1}{4}-\varepsilon}2^{-m(\frac{1}{4}-\varepsilon)}(m\log 2)^{\frac{2}{3}}\big)&\leq e^{-mcs^{3/2}\log 2}\\
  \mathbb{P}\big(|\mathfrak{h}^{\nabla,0,-1}_{t,k'_1,k'_2}|\geq s(d-c)^{\frac{1}{2}}2^{-\frac{m}{2}}(m\log 2)^{\frac{2}{3}}\big)&\leq e^{-mcs^{3/2}\log 2}
 \end{align*}
for some constant $c>0$ and $k'_1 \in \{k_1+1,\ldots , 2^{m-n}+k_1\}$ and $k'_2\in \{k_2+1, \ldots , 2^{m-n}+k_1\}$. Applying these inequalities shows 
\begin{align*}
\mathbb{P}\big(\tilde{\mathfrak{A}}_{\mathrm{up}}\big) \leq \sum_{m=n+1}^{\infty} \sum_{k'_1=k_1}^{2^{m-n}+k_1}\sum_{k'_2=k_2}^{2^{m-n}+k_2} e^{-mcs^{3/2}\log2} = \sum_{m=n+1} 2^{2(m-n)} e^{-mcs^{3/2}\log 2}\leq e^{-nc's^{3/2}}, \quad \forall s>(2/c)^{2/3}
\end{align*}
for some constant $c'>0$. Fix any $\alpha\in [\alpha^{(n)}_{k_1},\alpha^{(n)}_{k_1+1}]$ and $x\in [\alpha^{(n)}_{k_2},\alpha^{(n)}_{k_2+1}]$, we choose four sequences $\{\alpha^{(m)}_{k_{m}}\}_{m>n}$ and $\{x^{(m)}_{k_{m}}\}_{m>n}$ such that $\alpha^{(m)}_{k_{m}}\uparrow \alpha$ and $  x^{(m)}_{k_{m}}\uparrow x$
as $m\to \infty$. On the event $\neg \tilde{\mathfrak{A}}_{\mathrm{up}}$, we have  
\begin{align*}
\Big| &\h_{t}(\alpha^{(m)}_{k_{m}},x^{(m)}_{k_{m}}) +\frac{\big(x^{(m)}_{k_{m}}\big)^2}{2\alpha^{(m)}_{k_{m}}}- \h_{t}(\alpha^{(n)}_{k_1}, x^{(n)}_{k_2}) - \frac{\big(x^{(n)}_{k_{2}}\big)^2}{2\alpha^{(n)}_{k_{1}}} \Big|\\ &\leq  
 \sum_{m=n+1}^{\infty} \sum_{k'_1 = k_1}^{2^{m-n}+k_1} \sum_{k'_2= k_2+1}^{2^{m-n}+k_2} s\big((d-c)^{1/2}2^{-\frac{m}{2}}+ (b-a)^{\frac{1}{4}-\varepsilon}2^{-m(\frac{1}{4}-\varepsilon)}\big)\big(m\log 2\big)^{\frac{2}{3}}\\
&\leq  Cs\big((d-c)^{1/2}2^{-\frac{n}{2}}+ (b-a)^{\frac{1}{4}-\varepsilon}2^{-n(\frac{1}{4}-\varepsilon)}\big)\big(n\log 2\big)^{\frac{2}{3}}
\end{align*}
for some constant $C>0$ which does not depend on $\alpha$ or, $x$. 
Since the space time-process $\h_t(\cdot, \cdot)$ is continuous with probability $1$ (see \cite{BC95}), by letting $n\to \infty$ into the above display and taking supremum over $\alpha\in [\alpha^{(n)}_{k_1}, \alpha^{(n)}_{k_1+1}]$ and $x\in [x^{(n)}_{k_1}, x^{(n)}_{k_2+1}]$, we get 
\begin{align*}
\mathrm{Sup}\mathfrak{h}^{\nabla,1,1}_{t,k_1,k_2}\leq  Cs\big((d-c)^{1/2}2^{-\frac{n}{2}}+ (b-a)^{\frac{1}{4}-\varepsilon}2^{-n(\frac{1}{4}-\varepsilon)}\big)\big(n\log 2\big)^{\frac{2}{3}}. 
\end{align*} 
This implies 
\begin{align*}
\mathbb{P}\Big(\mathrm{Sup}\mathfrak{h}^{\nabla,1,1}_{t,k_1,k_2}\geq   Cs\big((d-c)^{1/2}2^{-\frac{n}{2}}+ (b-a)^{\frac{1}{4}-\varepsilon}2^{-n(\frac{1}{4}-\varepsilon)}\big)\big(n\log 2\big)^{\frac{2}{3}}\Big)\leq \mathbb{P}(\tilde{\mathfrak{A}}_{\mathrm{up}})\leq e^{-nc's^{3/2}} 
\end{align*}
for all large $s$. This completes the proof of Theorem~\ref{thm:ModCont}. 
 
\subsection{Proof of Proposition~\ref{shortdiff}}\label{sec:Prop5.1}  We will prove \eqref{sd-low} and \eqref{sd-up} in \emph{Stage-1} and \emph{Stage-2} respectively. We start with introducing relevant notations which will be used throughout the proof. Fix $t_0>0$, $\e\in (0,\frac14)$ and set $\kappa=\frac14-\e$. By the composition law,
	\begin{align} \label{e:comp}
	\h_t(1+\beta,0)-\h_t(1,0)=\frac{1}{t^{\frac13}}\log\int_{\R} \exp\left(t^{\frac13}\left(\h_t(1,t^{-\frac23}y)+\h_{(t+\beta t)\downarrow t}(-t^{-\frac23}y)-\h_t(1,0)\right)\right)dy
	\end{align}
	where $\h_{(t+\beta t)\downarrow t}(\cdot)$ is independent of $\h_t(1,\cdot)$ and is distributed as $\h_t(\beta,\cdot)$. We define $\widetilde{\h}_t(\beta,\cdot):\mathbb{R}\to \mathbb{R} $ and $\widetilde{\sh}_{\beta t}(\cdot):\mathbb{R}\to \mathbb{R}$ by $\widetilde{\h}_t(\beta,\cdot):= \h_{(t+\beta t)\downarrow t}(\cdot)$ and
	\begin{align*}
	\widetilde{\h}_t(\beta,t^{-\frac23}y)=\frac1{t^{\frac13}}\left(\frac{\pi \beta t}4\right)^{\frac14}\left(\widetilde{\sh}_{\beta t}(z)+\frac{z^2}{2}\right)+\frac{\frac{\beta t}{24}-\log\sqrt{2\pi \beta t}}{t^{\frac13}}-\frac{y^2}{2\beta t^{\frac43}},
	\end{align*}
	where $z=(\pi\beta^5t^5/4)^{-1/8}y$. Note that $\widetilde{\sh}_{\beta t}(x)$ is distributed as $\sh_{\beta t}((4^3t/\pi^3)^{1/8}x)$ and independent of $\h_t(1,\cdot)$. Writing the right hand side of \eqref{e:comp} in terms of of $\widetilde{\sh}_{\beta t}$ yields 
	\begin{equation}
	\begin{aligned}
	\h_t(1+\beta,0)-\h_t(1,0) & =\frac{\beta t^{2/3}}{24}+\frac{1}{t^{1/3}}\log\int_{\R} \frac1{\sqrt{2\pi\beta t}}\exp\left\{-\frac{y^2}{2\beta t}+
	t^{\frac13}\left(\h_t(t^{-\frac23}y)-\h_t(0)\right)\right. \nonumber\\ & \left.\hspace{2cm}+\left(\frac{\pi \beta t}4\right)^{\frac14}\left[\widetilde{\sh}_{\beta t}\left(\frac{-y}{(\pi\beta^5t^5/4)^{1/8}}\right)+\frac{y^2}{2(\pi\beta^5t^5/4)^{1/4}}
	\right]\right\}dy \nonumber\\ & =: \frac{\beta t^{2/3}}{24}+\frac1{t^{1/3}}\log \int_\R X_t(\beta,y) dy. \label{diffrel}
	\end{aligned}
	\end{equation}
	where the space-time stochastic process $X_t(\beta, y):\mathbb{R}_{>0}\times \mathbb{R} \to \mathbb{R}_{\geq 0}$ is defined by the above relation. 
	 We seek for an upper bound and a lower bound for the r.h.s.~of \eqref{diffrel} which will prove \eqref{sd-up} and \eqref{sd-low} respectively. 
	 
	 \medskip
	 
	 \noindent \textbf{Stage-1}. Define $\mathfrak{Int}(\beta,t):= [-t^{2/3}\beta^{2\kappa},t^{2/3}\beta^{2\kappa}]$. From \eqref{diffrel}, $\h_t(1+\beta,0)-\h_t(1,0)$ is bounded below by $t^{-1/3}\log \int_{\mathfrak{Int}(\beta,t)} X_t(\beta, y)dy $. This implies  \begin{align}\label{eq:BaseIneq}
	 \mathbb{P}\big(\h_t(1+\beta,0)- \h_t(1,0)\leq - \beta^{1/4-\varepsilon} s\big)\leq \mathbb{P}\Big(\log \int_{\mathfrak{Int}(\beta,t)} X_t(\beta, y)dy\leq -\beta^{1/4-\varepsilon}t^{1/3}s\Big).
	 \end{align}
	 Below, we find the upper bound to the right hand side of the above inequality. 
	The following inequality is straightforward from the definition of $X_t(\cdot,\cdot)$
\begin{align}\label{a:def}
	\frac{1}{t^{1/3}}\log \int_{\mathfrak{Int}(\beta,t)} X_t(\beta, y)dy & \ge \frac1{t^{1/3}}\log\int_{\mathfrak{Int}(\beta,t)} \frac{e^{-y^2/(2\beta t)}}{\sqrt{2\pi\beta t}}dy+\inf_{|y|\le \beta^{2\kappa}} (\h_t(y)-\h_t(0)) \nonumber\\ & +t^{-\frac13}\big(\frac{\pi \beta t}4\big)^{\frac14}\inf_{|y|\le (\pi\beta^5t^5/4)^{-1/8}t^{2/3}\beta^{2\kappa}}\Big(\widetilde{\sh}_{\beta t}(y)+\frac{y^2}{2}\Big).
	\end{align}
	The first term on the right hand side is deterministic. Using the Gaussian integral bound, we can write    
	\begin{align}\label{non-randomA}
	\frac1{t^{1/3}}\log\int_{\mathfrak{Int}(\beta,t)} \frac{e^{-y^2/(2\beta t)}}{\sqrt{2\pi\beta t}}dy & \ge \frac1{t^{1/3}}\log\Big(1-\exp\big(-t^{1/3}\beta^{4\kappa-1}/2\big)\Big) 
	\ge -\frac{2e^{-t^{1/3}\beta^{4\kappa-1}/2}}{t^{1/3}}.
	\end{align}
	where the last inequality follows since $\log(1-x)\geq -x$ for any $x\in (0,1)$. Note that $4\kappa-1<0$. For any given $s_0(\e)$, choosing $t_0(\e)$ large, we may bound $t^{-1/3}e^{-t^{1/3}\beta^{4\kappa-1}/2}$ by $\beta^{\kappa}s_0/8$ for all $t\ge t_0$, and $\beta \le t_0^{-2}$. This shows there exists $t_0(\e)$ large such that the right hand side of \eqref{non-randomA} is bounded below by $-\beta^{\kappa}s/4$ for all $t\geq t_0$, $\beta t\leq t^{-1}_0$ and $s\geq s_0$.

  By the inequality \eqref{a:def}, \eqref{non-randomA} and the union bound, the right side of \eqref{eq:BaseIneq} is bounded by $\mathbb{P}(\calA_1)+ \mathbb{P}(\calA_2)$ for all $t\geq t_0$, $\beta t\leq t^{-1}_0$ and $s\geq s_0$ where 
	\begin{align*}
	\m{A}_1 & := \Big\{\inf_{|y|\le\beta^{2\kappa}} (\h_t(y)-\h_t(0)) \le -\frac{\beta^{\kappa}s}{8}\Big\}, \\
	\m{A}_2 & := \Big\{\inf_{|y|\le (\pi\beta^5t^5/4)^{-1/8}t^{2/3}\beta^{2\kappa} }\Big(\widetilde{\sh}_{\beta t}(y)+\frac{y^2}{2}\Big) \le -\frac18\beta^{\kappa
		-\frac14}t^{1/12}s\Big\}. 
	\end{align*}
	By setting $\alpha=2$ in Lemma \ref{longtime:smdiffbd}, we get $\Pr(\m{A}_1)\le \exp(-cs^2)$ from \eqref{inf-smdiffbd}. In order to bound $\Pr(\m{A}_2)$, we use Lemma \ref{shortsupbd}. 	Mapping $a\mapsto -(\pi\beta^5t^5/4)^{-1/8}t^{2/3}\beta^{2\kappa}$, $\delta \mapsto 2(\pi\beta^5t^5/4)^{-1/8}t^{2/3}\beta^{2\kappa}$ and $s\mapsto -\frac18\beta^{\kappa
		\frac14}t^{1/12}s$ and choosing $s_0(\e)$ large, we note $(|a|+|\delta|)^2-a^2\leq s/2^{8}$ for all $s\geq s_0$. With those choice of $a,\delta, s$ in hand, the condition of Lemma~\ref{shortsupbd} is satisfied and hence, \eqref{e:shortinfbd} yields
	\begin{align*}
	\Pr(\m{A}_2) & \le \exp(-cs^{2}t^{1/6}\beta^{2\kappa-\frac{1}{2}})+\exp(-cs^{2}t^{1/6}\beta^{2\kappa-\frac{1}{2}}(\beta t)^{-1/8}t^{-1/24}\beta^{\kappa-\frac{3}8}) \\ & \le \exp(-cs^{2}t^{1/6}\beta^{2\kappa-\frac{1}{2}})+\exp(-cs^{2}\beta^{3\kappa-1}) \le \exp(-cs^2). 
	\end{align*}
	Combining the upper bounds on $\mathbb{P}(\m{A}_1)$ and $\mathcal{P}(\m{A}_2)$ and using those to bound the right side of \eqref{eq:BaseIneq} completes the proof of \eqref{sd-low}. 
	
	\medskip

	\noindent\textbf{Stage-2:} Here, we prove \eqref{sd-up}. According to \eqref{diffrel}, $\h_t(1+\beta,0) - \h_t(1,0)$ is a sum of $\beta t^{2/3}/24 + t^{-1/3}\log \int X_t(\beta ,y) dy$. For all $t\geq t_0$ and $\beta>0$ satisfying $\beta t\leq t^{-1}_0$, $\beta t^{2/3}$ is less than $\beta^{1/3}t^{-2/3}_0$. We can choose $s_0(\e)>0$ large such that $\beta t^{2/3}/24 \leq \beta^{1/4-\e}s/2$ for all $s\geq s_0$, $t\geq t_0$ and $\beta$ satisfying $\beta t\leq t^{-1}_0$. Thus, for all $s\geq s_0$, we have 
	\begin{align}\label{eq:X_tRep}
	\mathbb{P}(\h_t(1+\beta,0) - \h_t(1,0)\geq \beta^{1/4-\e}s)\leq \mathbb{P}\big(t^{-1/3} \log \int X_t(\beta, y)dy\geq \beta^{1/4-\e}s/2\big).
	\end{align}
%
	Our objective is to the upper bound the right hand side of the above inequality. To this end, let us denote $\mathfrak{Int}_{s}(\beta,t):=  [-\frac1{64}t^{2/3}\beta^{2\kappa}\sqrt{s},\frac1{64}t^{2/3}\beta^{2\kappa} \sqrt{s}]$. By the union bound, we may write 
	\begin{align*}
	\text{r.h.s. of \eqref{eq:X_tRep}}\leq \underbrace{\Pr\big(\int_{\mathfrak{Int}_{s}(\beta,t)} X_t(\beta, y)dy\geq e^{\frac{s}{2}t^{\frac{1}{3}}\beta^{\frac{1}{4}-\e}}\big)}_{=:(\mathbf{I})}+ \underbrace{\Pr\big(\int_{\mathbb{R}\backslash\mathfrak{Int}_{s}(\beta,t)} X_t(\beta, y)dy\geq e^{\frac{s}{2}t^{\frac{1}{3}}\beta^{\frac{1}{4}-\e}}\big)}_{=:(\mathbf{II})}.
	\end{align*}
 We will show that $(\mathbf{I})$ and $(\mathbf{II})$ are bounded above by $\exp(-cs^{3/2})$ for some constant $c>0$ in \emph{Step I} and \emph{Step II} respectively. Substituting these bounds into the right side of the above inequality completes the proof of \eqref{sd-up}.
\medskip 

\noindent \emph{Step I:} Using similar ideas as in \eqref{a:def}, we have 
	\begin{align*}
	\frac1{t^{1/3}}\log \int_{\mathfrak{Int}_{s}(\beta,t)} X_t(\beta,y)dy & \le  \frac1{t^{1/3}}\log\int_{\mathfrak{Int}_s(\beta,t)} \frac{e^{-y^2/(2\beta t)}}{\sqrt{2\pi\beta t}}dy+ \sup_{|y|\le \frac1{64}\beta^{2\kappa} \sqrt{s}} (\h_t(y)-\h_t(0)) \nonumber\\ & +  \frac1{t^{1/3}}\left(\frac{\pi \beta t}4\right)^{\frac14}\sup_{|y|\le \frac1{64}(\pi\beta^5t^5/4)^{-1/8}t^{2/3}\beta^{2\kappa} \sqrt{s}}\left(\widetilde{\sh}_{\beta t}(y)+\frac{y^2}{2}\right).
	\end{align*}
	Since $(2\pi \beta t)^{-1/2}\int_{\mathfrak{Int}_{s}(\beta,t)} e^{-y^2/2\beta t}dy<1$, from the above inequality and the union bound, it follows that $(\mathbf{I})\leq \Pr(\m{A}_3)+ \Pr(\m{A}_4)$ where 
	\begin{align*}
	\m{A}_3  & := \left\{\sup_{|y|\le\frac1{64}\beta^{2\kappa} \sqrt{s}} (\h_t(y)-\h_t(0)) \ge \frac{\beta^{\kappa}s}{8}\right\}, \\ 
	\m{A}_4 & := \left\{\sup_{|y|\le \frac1{64}(\pi\beta^5t^5/4)^{-1/8}t^{2/3}\beta^{2\kappa} \sqrt{s}}\left(\widetilde{\sh}_{\beta t}(y)+\frac{y^2}{2}\right) \ge \frac18\beta^{\kappa
		-\frac14}t^{1/12}s\right\}. 
	\end{align*}
  Indeed, from Lemma \ref{longtime:diffbd}, we know $\Pr(\m{A}_3) \le \exp(-cs^{3/2})$. In what follows, we claim and prove that  $\Pr(\m{A}_4) \le \exp(-cs^{3/2})$ for all large $s$ and some constant $c>0$. 
	\smallskip 
	
	 Let us denote $\mathfrak{M}:= \frac1{64}(4/\pi)^{1/8}\beta^{2\kappa-\frac58}t^{1/24}\sqrt{s}$ and $\delta:=\frac1{2^{14}}\beta^{\frac{3}8-\kappa}t^{1/24}\sqrt{s}$. Define $N:= \lceil \mathfrak{M}/\delta\rceil$. For any  $a\in \mathbb{R}$, define 
	$$
		\m{B}_{[a,a+\delta]}  =\left\{\sup_{y\in [a,a+\delta]} \left(\widetilde{\sh}_{\beta t}(y)+\frac{y^2}{2}\right) \ge \frac18\beta^{\kappa
			-\frac14}t^{1/12}s \right\}.
	$$
		Notice that $ \m{A}_4 \subset \cup_{i=-N-1}^{N} \m{B}_{[i\delta,(i+1)\delta]}$. Hence, by the union bound 
		\begin{align}\label{eq:A_4Bd}
		\Pr(\m{A}_4)\leq \sum_{i=-N-1}^{N} \Pr(\m{B}_{[i\delta,(i+1)\delta]}).
		 \end{align}
		In what follows, we seek to bound $\Pr(\m{B}_{[a,a+\delta]})$ for $a\in \{-(N+1)\delta,-N\delta,\ldots , N\delta\}$. To this end, we wish to apply Lemma \ref{shortsupbd}. It is readily checked that we have $|(|a|+|\delta|)^2-a^2|\le \frac{\beta^{\kappa-\frac14}t^{1/12}s}{2^{11}}$ for $a\in \{-(N+1)\delta,-N\delta,\ldots , N\delta\}$. Thus with the substitutions $t\mapsto \beta t$, $s\mapsto \beta^{\kappa-\frac14}t^{1/12}s$, and $\delta \mapsto \frac1{2^{14}}\beta^{\frac{3}8-\kappa}t^{1/24}\sqrt{s}$ in Lemma \ref{shortsupbd} we have
		\begin{align*}
		\Pr(\m{B}_{[a, a+\delta]}) & \le \exp(-cs^{3/2}t^{1/8}\beta^{\frac{3\kappa}{2}-\frac{3}{8}})+\exp(-cs^{2}t^{1/6}\beta^{2\kappa-\frac12}(\beta t)^{-1/8}t^{-1/24}\beta^{\kappa-\frac38}s^{-1/2}) \\ & \le \exp(-cs^{3/2}t^{1/8}\beta^{\frac{3\kappa}{2}-\frac{3}{8}})+\exp(-cs^{3/2}\beta^{3\kappa-1}). 
		\end{align*}
Substituting this upper bound into the right hand side of \eqref{eq:A_4Bd} and using the fact that $2(N+1)\le 4N \le 2^{11}\beta^{3\kappa-1}$, we get 
		\begin{align*}
		\Pr(\m{A}_4)  & \le 2^{11}\beta^{3\kappa-1}\left[\exp(-cs^{3/2}t^{1/8}\beta^{\frac{3\kappa}{2}-\frac{3}{8}})+\exp(-cs^{3/2}\beta^{3\kappa-1})\right] \le \exp(-cs^{3/2}).
		\end{align*}
		This completes the proof of the claim. Combining the bounds on $\Pr(\calA_3)$ and $\Pr(\calA_4)$ shows $(\mathbf{I})\leq \exp(-cs^{3/2})$ for all large $s$.

	\medskip	
	\noindent \emph{Step II:} Define $\tilde{y}:=y/(\pi\beta^5t^5/4)^{1/8}$. Recall the definition of $X_{t}(\beta,y)$ from \eqref{diffrel}. Adjusting the parabolic term inside the exponential of $X_t(\beta,y)$, we may rewrite
	\begin{align*}
	X_t(\beta,y)& = \frac{1}{\sqrt{2\pi\beta t}}\exp\Big\{-\frac{y^2}{4\beta t}+
	t^{\frac13}\big(\h_t\big(\frac{y}{t^{\frac23}}\big)-\h_t(0)\big)+\Big(\tfrac{\pi \beta t}4\Big)^{\frac14}\big[\widetilde{\sh}_{\beta t}(\tilde{y})+\frac{\tilde{y}^2}{4}\big]\Big\}\\
	 & \le \exp \Big\{ t^{1/3}\sup_{z\in\R} (\h_t(z)-\h_t(0)) +\Big(\tfrac{\pi \beta t}4\Big)^{\frac14}\sup_{z\in\R} \big[\widetilde{\sh}_{\beta t}(z)+\frac{z^2}{4}
	\big]\Big\} \frac{1}{\sqrt{2\pi\beta t}}\exp\Big(-\frac{y^2}{4\beta t}\Big).
	\end{align*}
	where the last inequality follows by fixing the quadratic term in $y$ and taking supremum of the rest of the terms as $y$ varies in $\mathbb{R}$. 
	Integrating both sides of the last inequality over $\mathbb{R}\backslash \mathfrak{Int}_s(\beta,t)$ and taking log on both sides yields shows 
\begin{equation}\label{uddd}
\begin{aligned}
\frac1{t^{\frac13}}\log \int_{\mathbb{R}\backslash \mathfrak{Int}_s(\beta,t)} X(\beta,t,y)dy & \le -\frac{s\beta^{4\kappa-1}}{2^{15}}+\sup_{z\in \R} (\h_t(z)-\h_t(0))+\frac1{t^{\frac13}}\Big(\frac{\pi \beta t}4\Big)^{\frac14}\sup_{z\in \mathbb{R}}\big(\widetilde{\sh}_{\beta t}(z) +\frac{z^2}{4}\big).
\end{aligned}
\end{equation}
where $-\frac{1}{2^{15}}s\beta^{4\kappa-1}$ is an upper bound to the logarithm of the Gaussian integral term. 
To bound $(\mathbf{II})$ using the above inequality, we introduce the following events:
		\begin{align*}
	\m{A}_5 & := \left\{\sup_{y \in \R} \h_t(y) \ge \frac{s}{2^{17}}\right\}, \quad \m{A}_6  := \left\{\h_t(0) \le -\frac{s}{2^{17}}\right\}, \quad 	\m{A}_7  :=\left\{\sup_{z\in\R}\left(\widetilde{\sh}_{\beta t}(z)+\frac{z^2}{4}\right) \ge \frac{s}{8}\right\}.
	\end{align*}
Note that on $\neg\m{A}_5\cap\neg \m{A}_6 \cap \neg\m{A}_7$, we get
	\begin{align*}
	\mbox{r.h.s~of }\eqref{uddd}  \le -\frac{1}{2^{15}}s\beta^{4\kappa-1}+ \frac{s}{2^{16}}+\beta^{1/4}t^{-1/12}\frac{s}{8} \le \frac14\beta^{1/4}s-\frac{s\beta^{4\kappa-1}}{2^{16}}.
	\end{align*}
for any $\beta<1$. Owing to this and the union bound, we have 
$$(\mathbf{II})\leq \Pr(\m{A}_5)+ \Pr(\m{A}_6)+ \Pr(\m{A}_7). $$ 	
From Proposition \ref{onepointlowtail} and Proposition \ref{ppn:UpperTail} with $\nu=1$, we get $\Pr(\m{A}_5), \Pr(\m{A}_6) \le \exp(-cs^{3/2})$. Lemma \ref{sup-proc} shows $\Pr(\m{A}_7)\le \exp(-cs^{3/2})$. Combining these bounds with the above inequality proves $(\mathbf{II})\leq \exp(-cs^{3/2})$ for all $s$ large and $\beta$ small.  This completes the proof of \eqref{sd-up}. 
	


\subsection{Proof of Proposition~\ref{longdiff}}\label{sec:Prop5.2}
	Recall the composition law
	\begin{align} \label{e:comp1}
	\h_t(1+\beta,0)=\frac{1}{t^{1/3}}\log\int_{\R} \exp\left(t^{1/3}\left(\h_t(1,t^{-2/3}y)+\beta^{1/3}\widehat{\h}_{(t+\beta t)\downarrow t}(-(\beta t)^{-2/3}y)\right)\right)dy
	\end{align}
	where  $\widehat{\h}_{\beta t}(x):=\beta^{-1/3}\h_{(t+\beta t)\downarrow t}(\beta^{2/3}x)$. We prove \eqref{ld-up} and \eqref{ld-low} in \emph{Stage-1} and \emph{Stage-2} respectively.
	
	\medskip
	
	\noindent \emph{Stage-1: Proof of \eqref{ld-up}}: We use the following notation $\h^{\nabla}_t(y):= \h_t(y)-\h_t(0)$ throughout this proof. Subtract $\h_t(1,0)$ from both sides of \eqref{e:comp1}. Furthermore, subtracting and adding the parabola $\frac{y^2}{4\beta t}$ inside the exponential of \eqref{e:comp1} shows  
	\begin{align}
	 \h_t(1+\beta,0)-\h_t(1,0) & = \frac{1}{t^{\frac13}}\log\int_\R \exp\Big(-\frac{y^2}{4\beta t}+t^{\frac13}\Big(\h^{\nabla}_t(t^{-\frac23}y)+\beta^{\frac13}\widehat{\h}_{\beta t}(-\beta^{-2/3}t^{-2/3}y)+\frac{y^2}{4\beta t^{4/3}}\Big)\Big)dy \nonumber\\ & \label{arhs} \le \beta^{1/3}\sup_{y\in\R} \big(\widehat{\h}_{\beta t}(y)+\frac{y^2}{4}\big)+ \frac{1}{t^{1/3}}\log\int_\R \exp\Big(-\frac{y^2}{4\beta t}+t^{1/3}\h^{\nabla}_t(t^{-2/3}y)\Big)dy,
	\end{align}
Let us define $\widehat{\mathfrak{Int}}_s(\beta,t):= \frac{1}{32} t^{2/3}\sqrt{\beta s}$ and consider the following events.
\begin{align*}
\m{A}_1 & := \left\{\sup_{x\in \sqrt{\beta s}/32} \h^{\nabla}_t(x) \ge \frac14\beta^{1/4}s\right\},\quad
\m{A}_2  := \left\{\sup_{x\in\R} \left(\widehat{\h}_{\beta t}(y)+\frac{y^2}{4} \right)\ge \frac{s}{4}\right\}	 \\ \m{A}_3	& := \left\{\sup_{|x|\in \R} \h_t(x) \ge \frac{s}{2^{14}}\right\}, \quad
\m{A}_4  := \left\{\h_t(0) \le -\frac{s}{2^{14}}\right\}, 
\end{align*}
To complete the proof of \eqref{ld-up}, we need the following lemma.
\begin{lemma}\label{lem:ContLemma}
$\{\h_t(1+\beta,0)- \h_t(1,0)\geq \beta^{1/4}s\}\subset (\m{A}_1\cup\m{A}_2\cup \m{A}_3\cup\m{A}_4)$. 
\end{lemma}
Before proceeding to prove \ref{lem:ContLemma}, we show how this will imply \eqref{ld-up}. From the above lemma and the union bound, we get 
\begin{align*}
\Pr(\h_t(1+\beta,0)-\h_t(1,0)\ge \beta^{1/4}s) \le \sum_{i=1}^4 \Pr(\m{A}_i)
\end{align*}
By Lemma \ref{longtime:diffbd} with $\kappa=\frac14$ we get that $\Pr(\m{A}_1)\le \exp(-cs^{3/2})$. By Proposition \ref{ppn:UpperTail}, with $\nu=\frac12$ and $\nu=0$ we get $\Pr(\m{A}_2)\le \exp(-cs^{3/2})$ and $\Pr(\m{A}_3)\le \exp(-cs^{3/2})$ respectively. The one point tail estimate in Proposition \ref{onepointlowtail} yields $\Pr(\m{A}_4)\le \exp(-cs^{5/2})\le \exp(-cs^{3/2})$. Combining all these bounds and substituting those into the above inequality completes the proof of \eqref{ld-up}. Now it boils down to proving Lemma~\ref{lem:ContLemma} which we do as follows.

\noindent \emph{Proof of Lemma~\ref{lem:ContLemma}}: Observe the following two inequalities 
\begin{align}
\int_{\widehat{\mathfrak{Int}}_s(\beta,t)} \exp\big(-\frac{y^2}{4\beta t}+t^{1/3}\h^{\nabla}_t(t^{-2/3}y)\big)dy & \le \sup_{|x|\le \sqrt{\beta s}/32} \h^{\nabla}_t(x)+t^{-1/3}\log\sqrt{4\pi\beta t}.\label{brhs} \\
\int_{\mathbb{R}\backslash \widehat{\mathfrak{Int}}_s(\beta,t)} \exp\big(-\frac{y^2}{4\beta t}+t^{1/3}\h^{\nabla}_t(t^{-2/3}y)\big)dy & \le \sup_{x\in \R} \h^{\nabla}_t(x)+t^{-1/3}\log\int_{\mathbb{R}\backslash \widehat{\mathfrak{Int}}_s(\beta,t)} e^{-\frac{y^2}{4\beta t}}dy \nonumber\\ & \label{crhs} \le \sup_{x\in \R} \h^{\nabla}_t(x)+t^{-1/3}\log\sqrt{4\pi\beta t} -\frac{s}{2^{13}}
\end{align}
where the last inequality follows from the bounds on the Gaussian tail integral. 
On $\neg \m{A}_1$ and $(\neg\m{A}_3 \cap \neg\m{A}_4)$, we have $$\mbox{r.h.s.~of }\eqref{brhs} \le \frac14\beta^{\frac14}s+t^{-\frac13}\log\sqrt{4\pi \beta t}, \quad \mbox{r.h.s.~of }\eqref{crhs} \le t^{-\frac13}\log\sqrt{4\pi\beta t} +\frac14\beta^{1/4}s$$
respectively.  Thus on $\neg(\m{A}_1\cup\m{A}_2\cup\m{A}_3\cup\m{A}_4)$ we get
\begin{align*}
\mbox{r.h.s.~of }\eqref{arhs} & \le \frac14\beta^{1/3}s+ t^{-1/3}\log2+t^{-1/3}\log\sqrt{4\pi \beta t}+\frac14\beta^{1/4}s \\ & \le \frac12\beta^{1/4}s+ (2\pi\beta)^{1/3}(16\pi \beta t)^{-1/3}\log(16\pi\beta t) <\beta^{1/4}s.
\end{align*}
The last inequality is true for all large enough $s$ since $\sup_{r>0} r^{-1/3}\log r$ is bounded. This shows $\neg(\m{A}_1\cup\m{A}_2\cup\m{A}_3\cup\m{A}_4)$ is contained in $\{\h_t(1+\beta,0)\leq \h_t(1,0)+\beta^{1/4}s\}$ and hence, completes the proof of the lemma.
\medskip

\noindent \emph{Proof of \eqref{sd-low}}: Restricting the integral in \eqref{e:comp1} over the region $\{|y|\le t^{2/3}\beta^{1/2}\}$ yields 
	\begin{align}
\h_t(1+\beta,0) & \ge \frac{1}{t^{1/3}}\log\int_{|y|\le t^{2/3}\beta^{1/2}} \exp\Big(t^{1/3}\big(\h_t(t^{-2/3}y)+\beta^{1/3}\widehat{\h}_{\beta t}(-\beta^{-2/3}t^{-2/3}y)\big)\Big)dy \nonumber\\ & \ge \beta^{1/3}\inf_{y\in\R} \big(\widehat{\h}_{\beta t}(y)+\frac{y^2}{4}\big)+\inf_{|y|\le \beta^{1/2}} \h_t(y)+ \frac{1}{t^{1/3}}\log\int_{|y|\le t^{2/3}\beta^{1/2}} \exp\Big(-\frac{y^2}{4\beta t}\Big)dy.\label{ano}
\end{align}
From the Gaussian tail bound, we have
\begin{align}\label{eq:gaustail}
\frac1{t^{1/3}}\log\int_{|y|\le t^{2/3}\beta^{1/2}} \exp\Big(-\frac{y^2}{4\beta t}\Big)dy  & \ge \frac1{t^{1/3}}\log\sqrt{4\pi\beta t}-\frac{2}{t^{1/3}}\exp\big(-\frac{t^{1/3}}4\big). 
\end{align}
We now claim and prove that there exists $s_0=s_0(t_0)>0$ such that $\{\h_t(1+\beta, 0)- \h_t(1,0)\leq -\beta^{1/4}s\}\subset \m{A}_5\cup \m{A}_6$ for all $s\geq s_0$ and $\beta>0$ satisfying $\beta t\geq t_0$ where  
\begin{align*}
\m{A}_5 & := \Big\{\inf_{|y|\le\beta^{1/2}} \h_t(y) \le \h_t(0)-\beta^{\frac14}s \Big\}, \quad \m{A}_6  := \Big\{\inf_{y\in \mathbb{R}} \big(\widehat{\h}_{\beta t}(y)+\frac{y^2}{4}\big) \le -\frac{s}{4} \Big\}
\end{align*}
To see this, using \eqref{ano} and \eqref{eq:gaustail}, we have 
\begin{align*}
\mbox{r.h.s.~of }\eqref{ano} & \ge -\frac{\beta^{\frac13}s}4+\h_t(0)-\beta^{\frac14}s+\frac{\log\sqrt{4\pi \beta t}}{t^{1/3}}-\frac2{t^{1/3}}\exp\big(-\frac{t^{1/3}}{4}\big) 
\end{align*}
on $\neg (\m{A}_5 \cup \m{A}_6)$. Note that $\log\sqrt{4\pi \beta t}/t^{1/3}$ is bounded below by $\log(4\pi t_0)/2t^{1/3}_0$ for all $t,\beta>0$ satisfying $t\geq t_0$ and $\beta t\geq t_0$. Furthermore, $\exp(-t^{1/3})/t^{1/3}$ converge to $0$ as $t$ increases to $\infty$. This shows there exists $s_0=s_0(t_0)>0$ such that for all $t\geq t_0$, $s\geq s_0$ and $\beta$ satisfying $\beta t\geq t_0$, the right hand side of the above display is greater than $\h_t(1,0)- \beta^{1/4}s$. This shows $\neg (\m{A}_5 \cup \m{A}_6)\subset \{\h_t(1+\beta,0)>\h_t(1,0)-\beta^{1/4}s\}$ and hence, the claim. 

From the above claim, we have 
$$\Pr(\h_t(1+\beta, 0)- \h_t(1,0)\leq -\beta^{1/4}s)\leq \Pr(\m{A}_5)+ \Pr(\m{A}_6). $$
Using Lemma \ref{longtime:smdiffbd}, we see that $\Pr(\m{A}_5)\le e^{-cs^2}$ and Proposition \ref{ppn:LowerTail} implies $\Pr(\m{A}_6)\le e^{-cs^{5/2}}$. 
Thus, $\Pr(\h_t(1+\beta,0)-\h_t(0)\le -\beta^{1/4}s) \le \Pr(\m{A}_5)+\Pr(\m{A}_6)\le e^{-cs^2}.$ This completes the proof.

\subsection{Proof of Proposition~\ref{temp-modulus}}\label{sec:Prop5.4}

Fix $\e\in (0,\frac14)$. From the scaling of $\h_t$, it follows that $\h_t(\alpha,0)=\alpha^{1/3}\h_{\alpha t}(1,0)$ for any $\alpha,t>0$. Hence, it suffices to prove the result for $a=1$. In the following, we first set up few notations and recall relevant result that we use in this proof. Consider the following events
\begin{align*}
\m{B}_1 & := \left\{\sup_{\tau\in [0,1]} \frac{\h_t(1+\tau,0)-\h_t(1,0)}{\tau^{\frac14-\e}\log^{2/3}\frac{1}{\tau}} \ge s\right\},  \quad
\m{B}_2  := \left\{\inf_{\tau\in [0,1]} \frac{\h_t(1+\tau,0)-\h_t(1,0)}{\tau^{\frac14-\e}\log^{1/2}\frac{1}{\tau}} \le -s\right\}.
\end{align*}
 Set $\kappa_1=\frac14-\e$ and $\kappa_2=\frac1{12}+\e$. For any $\alpha_1>\alpha_2\ge 1$, define
\begin{align*}
\mathfrak{h}^{\nabla}_{t,\alpha_1, \alpha_2} := \h_t(\alpha_1, 0) - \h_t(\alpha_2, 0)=\alpha_2^{1/3}(\h_{t\alpha_2}(\tfrac{\alpha_1}{\alpha_2},0)-\h_{t\alpha_2}(1,0)).
\end{align*}
and set $\beta = \frac{\alpha_1}{\alpha_2}-1$. Combining Proposition \ref{shortdiff} and Proposition~\ref{longdiff}, we get $t_0=t_0(\e)>0$, $s_0=s_0(\e)>0$ and $c=c(\e)>0$ such that for all $s\ge s_0$ and $2\alpha_2\ge \alpha_1> \alpha_2\ge 1$, 
\begin{align}\label{eq:SortTimeTail} \Pr\Big(\frac{\mathfrak{h}^{\nabla}_{t,\alpha_1,\alpha_2}}{(\alpha_1- \alpha_2)^{\kappa_1}}\ge \alpha_2^{\kappa_2}s\Big)  \leq
\exp(-cs^{3/2}),\quad  \Pr\Big(\frac{\mathfrak{h}^{\nabla}_{t,\alpha_1,\alpha_2}}{(\alpha_1- \alpha_2)^{\kappa_1}}\le -\alpha_2^{\kappa_2}s\Big)  \leq
\exp(-cs^{2})
\end{align} 

Now we proceed to complete the proof. Like as in the proof of Theorem~\ref{thm:ModCont}, we first construct a dyadic mesh of points of the interval $[1,2]$ and prove the tail bounds of the modulus of continuity over that mesh. Finally, the tail bounds of  the modulus of continuity will be extended for all points of $[1,2]$. To begin with, we consider the dyadic partitions $\{\bigcup_{k=1}^{2^n} \mathcal{J}^{(n)}_{k}\}_{n\in \N}$ of the interval $[1,2]$
\begin{align*}
\mathcal{J}^{(n)}_{k}:= \Big[\alpha^{(n)}_{k-1}, \alpha^{(n)}_{k}\Big], \quad \alpha^{(n)}_{k}:= 1+ \frac{k}{2^{n}}, \text{ for }k=0,1, \ldots , 2^{n}.
\end{align*}
We now define
\begin{align*}
\m{A}_{\mathrm{up}}(s) & := \bigcup_{n=1}^{\infty} \bigcup_{k=1}^{2^n} \Big\{ \mathfrak{h}^{\nabla}_{t,\alpha^{(n)}_{k},\alpha^{(n)}_{k-1}}\geq (\alpha_{k-1}^{(n)})^{\kappa_2}(\alpha_{k}^{(n)}-\alpha_{k-1}^{(n)})^{\kappa_1}(n \log 2 )^{\frac{2}{3}}s\Big\}, \\ 
\m{A}_{\mathrm{low}}(s) & := \bigcup_{n=1}^{\infty} \bigcup_{k=1}^{2^n} \Big\{ \mathfrak{h}^{\nabla}_{t,\alpha^{(n)}_{k},\alpha^{(n)}_{k-1}}\leq - (\alpha_{k-1}^{(n)})^{\kappa_2}(\alpha_{k}^{(n)}-\alpha_{k-1}^{(n)})^{\kappa_1}(n \log 2 )^{\frac{1}{2}}s\Big\},
\end{align*}
By the union bound, we write
\begin{align}\label{eq:TieSBd}
\mathbb{P}(\m{A}_{\mathrm{up}}(s)) \leq \sum_{n=1}^{\infty} \sum_{k=1}^{2^n} \mathbb{P}\Big(\mathfrak{h}^{\nabla}_{t,\alpha^{(n)}_{k},\alpha^{(n)}_{k-1}}\geq (\alpha_{k-1}^{(n)})^{\kappa_2}(\alpha_{k}^{(n)}-\alpha_{k-1}^{(n)})^{\kappa_1}(n \log 2 )^{\frac{2}{3}}s\Big).
\end{align}
Applying \eqref{eq:SortTimeTail} in the  right hand side of \eqref{eq:TieSBd}, we get
\begin{align*}
\mathbb{P}(\m{A}_{\mathrm{up}}(s))\leq \sum_{n=1}^{\infty}\sum_{k=1}^{2^n} \exp\big(- cn\log2s^{\frac{3}{2}}\big)\leq \sum_{n=1}^{\infty}\exp\big(-n\log2(cs^{\frac{3}{2}}-1)\big)\leq \exp\big(-\frac{c}{2}s^{\frac{3}{2}}\big)
\end{align*} 
Fix $\tau\in [\frac{1}{2^{k+1}},\frac1{2^{k}})$. By continuity of the process $\h_t(\cdot,0)$, we have the following on $\neg \m{A}_{\mathrm{up}}(s)$
\begin{align*}
\h_t(1+\tau,0)-\h_t(1,0) & =\sum_{n=1}^{\infty}\left[\h_t\left(\tfrac1{2^n}\lfloor 2^n(1+\tau)\rfloor,0\right)-\h_t\left(\tfrac1{2^{n-1}}\lfloor 2^{n-1}(1+\tau)\rfloor,0\right)\right] \\ & \le  \sum_{n=1}^{\infty}\left(\frac{\lfloor 2^{n-1}(1+\tau)\rfloor}{2^{n-1}}\right)^{\kappa_2}\left(\frac{\lfloor 2^{n}(1+\tau)\rfloor-2\lfloor 2^{n-1}(1+\tau)\rfloor}{2^n}\right)^{\kappa_1}(n \log 2 )^{\frac{2}{3}}s \\ & \le  2^{\kappa_2}s\sum_{n=k+1}^{\infty}\left(\frac{\lfloor 2^{n}(1+\tau)\rfloor-2\lfloor 2^{n-1}(1+\tau)\rfloor}{2^n}\right)^{\kappa_1}(n \log 2 )^{\frac{2}{3}}s \\ & \le c'\frac{(k+1)^{2/3}s}{2^{\kappa_1(k+1)}} \le c''s\tau^{\kappa_1}\log^{2/3}\frac1{\tau}
\end{align*}
Thus $\m{B}_1 \subset \m{A}_{\mathrm{up}}(s/c'')$ which proves \eqref{temp-sup}. Similarly we get $\m{B}_2 \subset \m{A}_{\mathrm{low}}(s/\tilde{c})$ for some constant $\tilde{c}>0$ and using similar summation trick as in \eqref{eq:TieSBd}, we have  $\Pr(\m{A}_{\mathrm{low}}(s)) \le e^{-cs^{2}}$. This proves \eqref{temp-inf} and hence, completes the proof of the desired results.


\section{Law Of Iterated Logarithms} \label{sec:lil}

The main goal of this section is to prove Theorem~\ref{lil}. We will prove the liminf result and the limsup result in Section~\ref{sec:LimInf} and \ref{sec:LimSup} respectively. One of the key ideas of our proof is to approximate multi-point distributions of the KPZ temporal process $\h_t$ with a set of independent random variables using the multipoint composition law of Proposition~\ref{ppn:MulpointComposition}. The following proposition encapsulates this idea for its use in Section~\ref{sec:LimInf} and~\ref{sec:LimSup}.  
\begin{proposition} \label{prop:IndProx}
For any $0=t_0<t_1<t_2<t_3<\ldots<t_m$, with $\mathfrak{s}:=\min_{i} |\exp(t_{i+1}-t_{i})- 1|$, there exist independent random variables $Y_1,Y_2,\ldots,Y_m$ and some constants $s_0= s_0(t_0)>0, c=c(t_0)>0$ such that for all $x\mathfrak{s}^{1/3}\geq s_0$ and $1\leq i\leq m$,  
\begin{align}\label{ind}
Y_i \stackrel{d}{=} (1-e^{-(t_i-t_{i-1})})^{1/3}\h_{e^{t_i}-e^{t_{i-1}}} \quad \text{and}, \quad \Pr\left(\left|\h_{e^{t_i}}-Y_i\right| \ge x\right)\le \exp(-cx^{3/2}\mathfrak{s}^{1/2})
\end{align}	
whenever $\mathfrak{s}\geq s_0$. 
\end{proposition}
\begin{proof}
Denote $\tilde{t}_i := e^{t_i}$ and $\tilde{\beta}_i := (\tilde{t}_i - \tilde{t}_{i-1})/\tilde{t}_{i-1}$. For any $1\leq i\leq m$, define $Y_i := (1+\tilde{\beta}_i)^{-1/3}\h_{\tilde{t}_{i}\downarrow \tilde{t}_{i-1}}$. Recall from Proposition~\ref{ppn:MulpointComposition} that $\{\h_{\tilde{t}_{i}\downarrow \tilde{t}_{i-1}}\}^{m}_{i=1}$ are set of independent random variables and $\h_{\tilde{t}_{i}\downarrow \tilde{t}_{i-1}}$ is same in distribution with $(1-(\tilde{t}_{i-1}/\tilde{t}_i))^{-1/3}\h_{\tilde{t}_i-\tilde{t}_{i-1}}$. From this, it follows that $Y_1, \ldots , Y_m$ are independent and $$Y_i\stackrel{d}{=}(1-\exp(-(t_i-t_{i-1})))^{1/3}\h_{e^{t_i}-e^{t_{i-1}}}.$$ Furthermore, applying Proposition~\ref{lngdiff}  with setting $t:= \tilde{t}_{i-1}$, $\beta:=\tilde{\beta}_i$  and $s:=x\tilde{\beta}^{1/3}_i$, there exists $s_0>0$ such that for all $x$ satisfying $x\mathfrak{s}^{1/3}\geq s_0$,
	\begin{align*}
	\Pr\Big(\big|\h_{\tilde{t}_{i-1}}(1+\tilde{\beta}_{i},0)-\h_{\tilde{t}_{i}\downarrow \tilde{t}_{i-1}}(1,0)\big| \ge x\tilde{\beta}^{1/3}_{i}\Big)\le \exp(-cx^{3/2}\tilde{\beta}^{1/2}_i)
	\end{align*}
	for some absolute constant $c>0$ which does not depend on $t_1,\ldots, t_m$. Note that $\h_{\tilde{t}_{i-1}}(1+\tilde{\beta}_{i},0)$ is equal to $(1+\tilde{\beta}_i)^{1/3}\h_{\tilde{t}_{i}}$. We do this substitution along with replacing $\h_{\tilde{t}_{i}\downarrow \tilde{t}_{i-1}}(1,0)$ by $(1+\tilde{\beta}_i)^{1/3}Y_i$ into the left side of the above inequality. Furthermore, we set $x$ to $1$. As a result, we obtain
	\begin{align*}
	\Pr\big(|\h_{\tilde{t}_i}- Y_i|\geq x\tilde{\beta}^{1/3}_i(1+\tilde{\beta}_i)^{-1/3}\big)\leq \exp( - x^{3/2}\tilde{\beta}^{1/2}_i)\leq \exp(- x^{3/2}\mathfrak{s}^{1/2})
\end{align*}	    
where the last inequality follows since $\tilde{\beta}_i\geq \min_{i}(e^{t_i-t_{i-1}}-1)= \mathfrak{s}$. Now \eqref{ind} follows from the above inequality. 
\end{proof}

\subsection{Proof of Liminf}\label{sec:LimInf}
 In this section, we will prove that the liminf of $\h_t/(\log\log t)^{1/3}$ is almost surely equal to $-6^{1/3}$. For any given $\epsilon>0$, we show that the following hold 
$$\underbrace{- \big(6(1+\epsilon)\big)^{1/3}\leq \liminf_{t\to\infty}\frac{\h_t}{(\log\log t)^{1/3}}}_{\mathfrak{LimInf}_{l}}, \qquad  \underbrace{\liminf_{t\to\infty}\frac{\h_t}{(\log\log t)^{1/3}}\le -\big(6(1-\epsilon)\big)^{1/3}}_{\mathfrak{LimInf}_{u}}$$
with probability $1$ in Section~\ref{sec:LimInfl} and~\ref{sec:LimInfu} respectively. By letting $\epsilon\to 0$ in the above two inequalities, it follows that $\liminf \h_t/(\log\log t)^{1/3}$ is equal to $-6^{1/3}$.

\subsubsection{Proof of $\mathfrak{LimInf}_{u}$}\label{sec:LimInfu}
 For any $n\in \mathbb{N}$, define $\mathcal{I}_n:=[\exp(e^{n}),\exp(e^{n+1})] ]$. We show that for any $\epsilon\in (0,1)$, 
\begin{align}\label{eq:BorCant}
\sum_{n=1}^{\infty}\Pr\Big(\inf_{t\in \mathcal{I}_n}\frac{\h_t}{(\log\log t)^{1/3}} \ge - \big(6(1-\epsilon)\big)^{1/3}\Big) <\infty.
\end{align}
Hence, by the Borel-Cantelli lemma, we have 
\begin{align}\label{eq:LimInfBd}
\liminf_{t\to \infty} \frac{\h_t}{(\log\log t)^{1/3}} \le - \big(6(1-\epsilon)\big)^{1/3} \quad \text{w.p. } 1
\end{align}
yielding the inequality $\mathfrak{LimInf}_{u}$.

Throughout the rest of the proof, we show \eqref{eq:BorCant}. Fix any $\epsilon\in (0,1)$ and set $\gamma := (6(1-\epsilon))^{1/3}$. Choose $\eta>0$ small such that $(\tfrac{1}{6}+\eta)(\gamma+2\eta)^3<1$. We define $\zeta:=(\tfrac{1}{6}+\eta)(\gamma+2\eta)^3$.
	Fix $\theta\in (\zeta,1)$ and choose $\delta\in (0,\theta-\zeta)$. For any $n\geq 1$, we consider the following sub-intervals of $\mathcal{I}_n$,
	\begin{equation}\label{eq:Mtheta}
	\mathcal{I}^{(j)}_{n}: = [\exp(e^n+(j-1)e^{n\theta}),\exp(e^n+je^{\theta n})], \quad 1\leq j\leq \mathcal{M}_{\theta}:=\lfloor e^{n-\theta n+1}-e^{n - n\theta}\rfloor.
	\end{equation} 

	 By the union bound, we have
\begin{align}
\Pr\big(\inf_{t\in \mathcal{I}_{n}}\frac{\h_t}{(\log\log t)^{1/3}} \ge -\gamma\big)  & \le \sum_{j=1}^{\mathcal{M}_{\theta}} \Pr\big(\inf_{t\in \mathcal{I}^{(j)}_n}\frac{\h_t}{(\log\log t)^{1/3}} \ge -\gamma\big) \leq \sum_{j=1}^{\mathcal{M}_{\theta}} \Pr\big(\inf_{t\in \mathcal{I}^{(j)}_n}\h_{t} \ge -(n+1)^{1/3}\gamma\big) \label{unino}
\end{align}
where the last inequality follows since $\max_{t\in \mathcal{I}^{(j)}_n}\log \log t\leq (n+1)$. Now we bound each term of the above sum. For convenience, we use the shorthand $\m{A}^{(j)}_n$ to denote $\big\{\inf_{t\in \mathcal{I}^{(j)}_n}\h_{t} \ge -(n+1)^{1/3}\gamma\big\}$. 

We now claim that there exists constants $c_1,c_2>0$ such that 
\begin{align}\label{eq:AjBound}
\Pr(\m{A}^{(j)}_n)\leq \exp(-ce^{n(\theta-\delta)}e^{-n\zeta}) + \exp\big(n(\theta-\delta)-c_2(\exp(e^{n\delta})-1)^{1/2}\big)
\end{align} 
 for all $1\leq j\leq n$ and all large $n$. We first assume \eqref{eq:AjBound} and complete the proof of \eqref{eq:BorCant}. Using \eqref{eq:AjBound}, we may estimate the right hand side of \eqref{unino} as 
 \begin{align}\label{eq:AjBoundCons}
 \text{r.h.s. of \eqref{unino}} \leq e^{n-n\theta +1}\Big(\exp(-ce^{n(\theta-\delta)}e^{-n\zeta}) + \exp\big(n(\theta-\delta)-c_2(\exp(e^{n\delta})-1)^{1/2}\big)\Big) 
\end{align}  
Here, the factor $e^{n-n\theta +1}$ is an upper bound to the number of summands in \eqref{unino}. Recalling that $\theta>\zeta+\delta$, we observe that the right hand side of the above display can be bounded by $\exp(- c_1e^{n\omega})$ for some constant $c_1>0$ and $\omega\in (0,1)$ for all large $n$. This shows the sum in \eqref{eq:BorCant} is finite and hence, completes the proof of \eqref{eq:LimInfBd} modulo \eqref{eq:AjBound} which we show as follows.

 Fix $j\in \{1, \ldots , \mathcal{M}_{\theta}\}$ and some constant $a>1$. We choose a sequence $t_1<t_2<\cdots<t_{L_n} \in [e^{n}+(j-1)e^{n\theta},e^{n}+je^{n\theta}]$ such that $$\min |t_{i+1}-t_i| \ge e^{n\delta }\quad\text{and},  \frac1a(e^{n(\theta-\delta)})\le L_n\le a(e^{n(\theta-\delta)}). $$ Applying Proposition~\ref{prop:IndProx}, we get $Y_1, Y_2, \ldots , Y_{L_n}$ such that \eqref{ind} (with $\mathfrak{s}\geq e^{n\delta}$) will be satisfied for the above choice of $t_1, t_2, \ldots , t_{L_n}$. As a result, we get 
\begin{align}
\Pr(\m{A}^{(j)}_n) &  \le \Pr(\min_{1\le i\le L_n}\h_{e^{t_i}} \ge -(n+1)^{1/3}\gamma) \nonumber\\ &  \le \Pr\big(\min_{1\le i\le L_n}Y_i \ge -(n+1)^{1/3}\gamma-1\big)+\sum_{i=1}^{L_n}\Pr(\h_{e^{t_i}}-Y_i\ge 1) \nonumber\\ &  \le \prod_{i=1}^{L_n}\Pr\big(Y_i \ge -(n+1)^{1/3}\gamma-1\big)+a\exp\big(n(\theta-\delta)-c(\exp(e^{n\delta })-1)^{1/2}\big) \label{eq:AjBdComp}
\end{align}
where in the last line we use the independence of $Y_i$ to write $\Pr\big(\min_{1\le i\le L_n}Y_i \ge -(n+1)^{1/3}\gamma-1\big)$ as a product over  $\Pr\big(Y_i \ge -(n+1)^{1/3}\gamma-1\big)$ and use the inequality in \eqref{ind} to bound $\Pr(\h_{e^{t_i}}-Y_i\ge 1)$. Using the distributional identity of \eqref{ind} , we get  
\begin{align*}
\prod_{i=1}^{L_n}\Pr\left(Y_i \ge -(n+1)^{1/3}\gamma-1\right)  & = \prod_{i=1}^{L_n}\Pr\left((1-e^{-(t_i-t_{i-1})})^{1/3}\h_{e^{t_i}-e^{t_{i-1}}} \ge -(n+1)^{1/3}\gamma-1\right)\\ &  \le \prod_{i=1}^{L_n}\Pr\big(\h_{e^{t_i}-e^{t_{i-1}}} \ge -n^{1/3}(\gamma+\eta)\big) \\ & \le \big[1-\exp(-(\tfrac1{6}+\eta)n(\gamma+2\eta)^3)\big]^{L_n} \\ & \le \exp\big(-L_n\exp(-(\tfrac1{6}+\eta)n(\gamma+2\eta)^3)\big)\leq \exp(-e^{n(\theta-\delta)}
e^{-n\zeta}/a)
\end{align*}
where the first inequality follows by noting that $(1-e^{-(t_{i}-t_{i-1})})^{-1/3}((n+1)^{1/3}\gamma+1)\leq n^{1/3}(\gamma+\eta)$ for all large $n$ and the second inequality follows due to \eqref{eq:lowtailTlarge} of Proposition~\ref{ppn:LargeTime}. The last inequality follows since $L_n\geq e^{n(\theta-\delta)}/a$ and $\zeta= (\tfrac1{6}+\eta)(\gamma+2\eta)^3$. 

Substituting the inequality of the above display into the right hand side of \eqref{eq:AjBdComp} yields the inequality \eqref{eq:AjBound}. This completes the proof.

\subsubsection{Proof of $\mathfrak{LimInf}_{l}$}\label{sec:LimInfl}
 Fix $t_0>0$. Define $\psi:\mathbb{R}_{>1}\to \mathbb{R}_{>0} $ as $\psi(\alpha)=\alpha^{1/3}(\log\log \alpha)^{1/3}$. Let $\alpha_n:=2^n$ and $k_n:=\lfloor(\log\log\alpha_n)^4\rfloor$ for any $n\in \mathbb{N}$. Let us denote $\mathcal{I}_n:=[\alpha_n,\alpha_{n+1}]$ and its $k_n$ many equal length sub-intervals as $\mathcal{I}^{(j)}_n:=[(1+\frac{j-1}{k_n})\alpha_{n},(1+\frac{j}{k_n})\alpha_{n}]$ for $1\leq j\leq k_n $. We will show that 
 \begin{align}\label{eq:SunBd}
 \sum_{n=1}^{\infty}\Pr\Big(\inf_{\alpha\in \mathcal{I}_n}\frac{\h_{t_0}(\alpha,0)}{\psi(\alpha)}\leq -\big(6(1+\epsilon)\big)^{1/3}\Big)<\infty
 \end{align}
 Applying \eqref{eq:SunBd} and Borel-Cantelli lemma, we can conclude that 
 \[\liminf_{\alpha\to \infty}\frac{h_{t_0}(\alpha,0)}{\psi(\alpha)}= \liminf_{t\to \infty}\frac{\h_t}{(\log \log t)^{1/3}}\geq - \big(6(1+\epsilon)\big)^{1/3}\] 
 with probability $1$ where the equality is obtained by substituting $t= \alpha t_0$. Letting $\epsilon\to 0$ on the right hand side of the above inequality yields $\mathfrak{LimInf}_{l}$. It boils down to showing \eqref{eq:SunBd} which we do as follows. 
 
 We claim that there exist constant $c_1>0$ and $c_2>1$ such that 
 \begin{align}\label{eq:EachBd}
 \Pr\Big(\inf_{\alpha\in \mathcal{I}_n}\frac{\h_{t_0}(\alpha,0)}{\psi(\alpha)}\leq -(6(1+\epsilon))^{1/3}\Big)\leq (\log \log \alpha_n)^4\Big(e^{-c_1(\log \log \alpha_n)^{7/6}} + e^{-c_2\log \log \alpha_n}\Big)
 \end{align}
 for all large $n$. Recall that $\alpha_n = 2^{n}$. Substituting this into the right hand side of the above inequality, we see that \eqref{eq:SunBd} holds modulo the last inequality. We now proceed to prove this last inequality.
 
 Let $N$ be the smallest positive integer such that $\alpha_N \ge e^e$. For any $n\ge N$, using the union bound we have
\begin{align} \label{2-sum}
  \Pr\Big(\inf_{\alpha\in \mathcal{I}_n}\frac{\h_{t_0}(\alpha,0)}{\psi(\alpha)}\le -\big(6(1+\epsilon)\big)^{1/3}\Big) & \le \sum_{j=1}^{k_n}\Pr\Big(\inf_{\alpha\in \mathcal{I}^{(j)}_n}\frac{\h_{t_0}(\alpha,0)}{\psi(\alpha)}\le -\big(6(1+\epsilon)\big)^{1/3}\Big)
\end{align}
In what follows, we will show 
\begin{align}\label{eq:IndVBound}
\Pr\Big(\inf_{\alpha\in \mathcal{I}^{(j)}_n}\frac{\h_{t_0}(\alpha,0)}{\psi(\alpha)}\le -\big(6(1+\epsilon)\big)^{1/3}\Big) \leq e^{-c_1(\log \log \alpha_n)^{7/6}} + e^{-c_2\log \log \alpha_n}
\end{align}
for all $1\leq j\leq n$, $n$ large and some constant $c_1>0$ and $c_2>1$. Substituting the above inequality into right side of \eqref{2-sum} and recalling that $k_n\leq (\log \log \alpha_n)^4$ show \eqref{eq:EachBd}.

Throughout the rest of the proof, we focus on proving \eqref{eq:IndVBound}. Fix any $j \in \{1,\ldots , k_n\}$. Denote the left and right end point of $\mathcal{I}^{(j)}_n$ by $a_n$ and $b_n$. For convenience, we will denote $(6(1+\epsilon))^{1/3}$ by $s$. We choose $\eta\in (0,1)$ such that $(1-\eta)^4(1+\epsilon)>1$. Using the fact that $\psi(\alpha)$ is an increasing function of $\alpha$, we get
\begin{align}
 \Pr\Big(\inf_{\alpha\in \mathcal{I}^{(j)}_n}\frac{\h_{t_0}(\alpha,0)}{\psi(\alpha)}\le -s\Big)  & \leq \Pr\Big(\inf_{\alpha\in \mathcal{I}^{(j)}_n}\h_{t_0}(\alpha,0)\le -s\psi(a_n)\Big) \nonumber \\
 &\leq \Pr\Big(\inf_{\alpha\in \mathcal{I}^{(j)}_n}\h_{t_0}(\alpha,0)-\h_{t_0}(a_n,0)\leq -\eta s\psi(a_n)\Big)  \nonumber \\&+ \Pr\Big(\h_{t_0}(a_n,0)\leq -(1-\eta) s\psi(a_n)\Big) \label{trm1}
\end{align}
where the last inequality follows by the union bound. We now apply \eqref{temp-inf} of Proposition~\ref{temp-modulus} and \eqref{eq:lowtailTlarge} of Proposition~\ref{ppn:LargeTime} to bound the first and the second term of the right side of the last inequality.

To apply \eqref{temp-inf}, we set $\e=\frac{1}{8}$. We may now write  
\begin{align}
 \Pr &\Big(\inf_{\alpha\in \mathcal{I}^{(j)}_n}\h_{t_0}(\alpha,0)-\h_{t_0}(a_n,0)\leq -\eta s\psi(a_n)\Big)\nonumber \\& \leq \Pr\Big(\inf_{\tau\in [0,k^{-1}_n]}\frac{\h_{t_0}(a_n+\tau,0)-\h_{t_0}(a_n,0)}{(\tau/a_n)^{1/8}(\log|\tau/a_n|)^{1/2}}\leq -\frac{\eta s\psi(a_n)}{k^{-1/8}_n(\log|k_n|)^{1/2}}\Big)\nonumber\\
 &\leq \exp\big(-c(\eta s)^2k^{1/4}_n(\log|k_n|)^{-1}(\log \log \alpha_n)^{2/3}\big)\label{eq:FRstBd}
\end{align}
where the second inequality follows by applying \eqref{temp-inf}. Since $k_n = \lfloor (\log \log \alpha_n)^4\rfloor$, we get the following bound  
 $$k^{1/4}_n(\log|k_n|)^{-1}(\log \log \alpha_n)^{2/3} \geq (\log \log \alpha_n)^{\frac{1}{2}+\frac{2}{3}} = (\log\log\alpha_n)^{7/6}$$
 for all large $n$. By substituting inequality into the right hand side of \eqref{eq:FRstBd}, we may bound the first term in the right hand side of \eqref{trm1} by $\exp(-c_1(\log \log \alpha_n)^{7/6})$ for all large integer $n$ where $c_1$ is a positive which does not depend on $n$. 
On the other hand, \eqref{eq:lowtailTlarge} of Proposition~\ref{ppn:LargeTime} implies
\begin{align}
\Pr\Big(\h_{t_0}(a_n,0)\leq -(1-\eta) s\psi(a_n)\Big)\leq e^{-(1-\eta)^4(1+\epsilon)(\log \log a_n)}\leq e^{-c_2\log \log \alpha_n}\label{eq:SEcBd}
\end{align} 
for sufficiently large $n$ where $c_2$ is a constant greater than $1$. The second inequality of the above display follows since $a_n\geq \alpha_n$ and $(1-\eta)^4(1+\epsilon)>1$.  

Combining the bounds in \eqref{eq:FRstBd} and \eqref{eq:SEcBd} and substituting those into \eqref{trm1} shows \eqref{eq:IndVBound}. This completes the proof of $\mathfrak{Liminf}_{l}$. 

\subsection{Proof of Limsup}\label{sec:LimSup}
The main goal of this section is to prove the limsup result of the law of iterated logarithm for which we need to show that for any $\epsilon\in (0,1)$,
\begin{align*}
\underbrace{\limsup_{t\to \infty} \frac{\h_{t}}{(\log \log t)^{2/3}} \geq  \Big(\frac{3(1-\epsilon)}{4\sqrt2}\Big)^{2/3}}_{\mathfrak{LimSup}_{l}}, \quad \underbrace{\limsup_{t\to \infty} \frac{\h_{t}}{(\log \log t)^{2/3}} \leq  \Big(\frac{3(1+\epsilon)}{4\sqrt2}\Big)^{2/3}}_{\mathfrak{LimSup}_{u}}. 
\end{align*} 
with probability $1$. 

In what follows, we first show $\mathfrak{LimSup}_{u}$. 
As we discuss in the next section, $\mathfrak{LimSup}_{u}$ will imply that the macroscopic Hausdorff dimension of the level sets $\{t\geq e^{e}:\h_t/(\log \log t)^{2/3}\geq (3(1+\epsilon)/4\sqrt2)^{2/3}\}$ are equal to $0$ with probability $1$ for any $\epsilon>0$ proving partially \eqref{eq:Monofractal}.  

\subsubsection{Proof of $\mathfrak{LimSup}_{u}$}\label{sec:imSupu}

Fix $\z,\theta \in (0,1)$ and $t_0>0$. 
Define $\phi:\mathbb{R}\to \mathbb{R}$ by $\phi(x)= x^{1/3}(\log \log x)^{2/3}$. We note that $\phi(x)$ is increasing in $x$. We will show that
\begin{align} \label{reduced}
\limsup_{\alpha\to\infty} \frac{\h_{t_0}(\alpha,0)}{\phi(\alpha)} \le \Big(\frac{3(1+\epsilon)}{4\sqrt{2}}\Big)^{2/3}
\end{align}
holds with probability $1$ for all large $t$ and $\epsilon>0$. To see how $\mathfrak{LimSup}_{l}$ follows from \eqref{reduced}, note 
\begin{align*}
\limsup_{t\to\infty} \frac{\h_t}{(\log\log t)^{2/3}} & = \limsup_{\alpha\to \infty} \frac{\h_{\alpha t_0}}{(\log\log \alpha t_0)^{2/3}} \\ & = \limsup_{\alpha\to \infty} \Big[\frac{\h_{ t_0}(\alpha,0)}{\alpha^{1/3}(\log\log \alpha)^{2/3}}\cdot \big(\frac{\log\log \alpha t_0}{\log\log \alpha}\big)^{2/3}\Big]=\limsup_{\alpha\to \infty} \frac{\h_{t_0}(\alpha,0)}{\phi(\alpha)}
\end{align*}
Fix $\delta\in (0,1)$. We will make the choice $\delta$ precise in due course of the proof. For any $n\in \mathbb{N}$, we define $\alpha_n:=(1+\delta)^i$ and denote $\mathcal{I}_n:=[\alpha_n,\alpha_{n+1}]$. We claim and prove that 
\begin{align}\label{eq:LimsupSum}
\sum_{n\geq N}\Pr\Big(\sup_{\alpha\in \mathcal{I}_n}\frac{\h_t(\alpha,0)}{\phi(\alpha)}\ge \big(\frac{3(1+\epsilon)}{4\sqrt{2}}\big)^{2/3}\Big) <\infty
\end{align}
for all $\epsilon>0$. By Borel-Cantelli lemma, we get \eqref{reduced} from \eqref{eq:LimsupSum} which we show by proving the following: there exists and $c>1$ such that 
\begin{align}\label{eq:EachTermUpBd}
\Pr\Big(\sup_{\alpha\in \mathcal{I}_n}\frac{\h_t(\alpha,0)}{\phi(\alpha)}\ge \big(\frac{3(1+\epsilon)}{4\sqrt{2}}\big)^{2/3}\Big) \leq \exp(-c\log \log \alpha_n)
\end{align}
for all large $n$ and $t$.

 Let $N$ be the smallest positive integer such that $\alpha_{N} \ge e^e$. Define $s:= (3(1+\epsilon)/4\sqrt{2})^{2/3}$. For $n\ge N$ and $\eta\in (0,1)$, we have
\begin{align} \nonumber
\Pr\Big(\sup_{\alpha\in \mathcal{I}_n}\frac{\h_{t_0}(\alpha,0)}{\phi(\alpha)}\ge s\Big) 
& \le \Pr\Big(\sup_{\alpha\in \mathcal{I}_n}\h_{t_0}(\alpha, 0)\ge s\phi(\alpha_i)\Big) \\&\le \Pr\Big(\sup_{\alpha\in \mathcal{I}_{n}} \h_{t_0}(x,0)-\h_{t_0}(\alpha_{i},0)\ge \eta s\phi(\alpha_i)\Big) + \Pr\Big(\frac{\h_{t_0}(\alpha_{i},0)}{\phi(\alpha_i)}\ge (1-\eta)s\Big) \label{term2} 
\end{align}
where the first inequality follows since $\phi$ is an increasing function of $\alpha$ and the second inequality follows by the union bound. 
We proceed to bound the two terms in the right hand side of the last display. For the first term, we seek to apply \eqref{temp-sup} of Proposition~\ref{temp-modulus}. We set $\e=\frac18$ in Proposition \ref{temp-modulus}, and define  $r:=\sup_{\tau\in (0,\delta]} \tau^{\frac{1}{8}}\big(\log(1/\tau)\big)^{\frac{2}{3}}.$ It straightforward to see that $r$ decreases to $0$ as $\delta$ goes to $0$. We may now write 
\begin{align}
\Pr\Big(\sup_{\alpha\in \mathcal{I}_{n}} \h_{t_0}(\alpha,0)-\h_{t_0}(\alpha_{n},0)\ge \eta s\phi(\alpha_i)\Big) &\le \Pr\Big(\sup_{\tau\in (0,\delta]} \frac{\h_{t_0}((1+\tau)\alpha_n,0)-\h_{t_0}(\alpha_{n},0)}{\tau^{1/8}\log^{2/3}(1/\tau)}\ge \eta \frac{s}{r}\phi(\alpha_n)\Big)\nonumber\\
  &\leq \exp\big(-c\big(s\eta r^{-1}(\log \log \alpha_n)^{2/3}\big)^{3/2}\big)\label{eq:FirstTermBd}
\end{align}
where the last inequality follows from Proposition~\ref{temp-modulus}. For any fixed $\eta$, we choose $\delta>0$ small such that 
$c(\eta r^{-1})^{3/2} > 1.$ For such choice of $\delta$, the right hand side of the last line of the above display will be bounded above by $\exp(-c_1\log \log \alpha_n)$ for some constant $c_1>1$. This bounds the first term in the right hand side of \eqref{term2}. We now proceed to bound the second term.


 Using the second inequality of \eqref{eq:uptailTlarge} in Proposition~\ref{ppn:LargeTime}, for all large $t$ and $n$
\begin{align*}
\Pr\Big(\frac{\h_{t_0}(\alpha_{n},0)}{\phi(\alpha_{n})}\ge (1-\eta)s\Big) & = \Pr\big(\h_{\alpha_{n}t_0}\ge (1-\eta)s(\log \log \alpha_n)^{2/3}\big) \le \exp\Big(-\frac{4\sqrt{2}}{3}(1-\gamma)^{5/2}s^{3/2}\log\log \alpha_n\Big)
\end{align*} 
were the equality holds since $\h_{t_0}(\alpha_{n},0)/\alpha^{1/3}_n$ is same as $\h_{\alpha_n t_0}(1,0)$ and $\h_{\alpha_n t_0}$ is the shorthand for $\h_{\alpha_n t_0}(1,0)$.
Recall that $\eta $ is chosen in a way such that $(1-\eta)^{5/2}(1+\epsilon)>0$. Since $\frac{4\sqrt{2}}{3}(1-\gamma)^{5/2}s^{3/2} = (1-\eta)^{5/2}(1+\epsilon)$, the right hand side of the above display is bounded by $\exp(-c_2\log \log \alpha_n)$ for some constant $c_2>0$. Combining this upper bound with the bounds in \eqref{eq:FirstTermBd} an substituting those into the right hand side of \eqref{term2} yields \eqref{eq:EachTermUpBd}. This completes the proof.



\subsubsection{Proof of $\mathfrak{LimSup}_{l}$}\label{sec:LimSupl}
 
We prove $\mathfrak{LimSup}_{l}$ using similar argument as in the proof of $\mathfrak{LimInf}_{u}$ (see Section~\ref{sec:LimInfu}). Recall the definitions of the interval $\mathcal{I}_n$ from Section~\ref{sec:LimInfu}. Due to Borel-Cantelli lemma, it suffices to show 
\begin{align}\label{eq:BClimSupl}
\sum_{n=1}^{\infty}\Pr\Big(\sup_{t\in \mathcal{I}_n} \frac{\h_t}{(\log \log t)^{1/3}}\leq \Big(\frac{3(1-\epsilon)}{4\sqrt{2}}\Big)^{\frac{3}{2}}\Big)<\infty
\end{align} 
Set $\gamma := (3(1-\epsilon)/4\sqrt{2})^{2/3}$. Choose $\eta>0$ such that $\zeta:=(\tfrac{4\sqrt{2}}{3}+\eta)(\gamma+2\eta)^{3/2}<1$. Fix $\theta \in (\zeta,1)$ and $\delta \in (0, \theta - \zeta)$. For such choice of $\theta$, recall the definition of the subintervals $\{\mathcal{I}^{(j)}_{n}\}^{\mathcal{M}_{\theta}}_{j=1}$ from \eqref{eq:Mtheta}. Denoting $\tilde{\m{A}}^{(j)}_{n}:= \{\sup_{t\in \mathcal{I}^{(j)}_n} \h_t\leq (n+1)^{2/3}\gamma \}$, we have   
$$\Pr\Big(\sup_{t\in \mathcal{I}_n} \frac{h_t}{(\log \log t)^{1/3}}\leq \big(\frac{3(1-\epsilon)}{4\sqrt{2}}\big)^{\frac{3}{2}}\Big)\leq \sum_{j=1}^{\mathcal{M}_{\theta}}\Pr(\tilde{\m{A}}^{(j)}_{n}).
$$
by the union bound. In a similar way as in \eqref{eq:AjBound}, we claim that there exists $c_1,c_2>0$ such that  
\begin{align}\label{eq:AjBoundAlt}
\Pr(\tilde{\m{A}}^{(j)}_{n})\leq \exp(-c_1e^{n(\theta-\delta)}e^{-n\zeta}) + \exp\big(n(\theta-\delta)-c_2(\exp(e^{n\delta})-1)^{1/2}\big)
\end{align}
for all $1\leq j\leq n$ and all large $n$. Using this upper bound on $\Pr(\tilde{\m{A}}^{(j)}_{n})$, we may bound each term in the sum \eqref{eq:BClimSupl} exactly in the same way as in \eqref{eq:AjBoundCons}. Since $\theta>\zeta +\delta$ by our choice, we may also bound each term of the sum in \eqref{eq:BClimSupl} by $\exp(-e^{n\omega})$ for some $\omega\in (0,1)$. This shows the finiteness of the sum in \eqref{eq:BClimSupl}. To complete the proof, it boils down to showing \eqref{eq:AjBoundAlt} which we do as follows.   

Fix $j\in \{1, \ldots , \mathcal{M}_{\theta}\}$ and some constant $a>1$. Like as in Section~\ref{sec:LimInfu}, we choose a sequence $t_1<t_2<\cdots<t_{L_n} \in [e^{n}+(j-1)e^{n\theta},e^{n}+je^{n\theta}]$ such that $$\min |t_{i+1}-t_i| \ge e^{n\delta }\quad\text{and},  \frac1a(e^{n(\theta-\delta)})\le L_n\le a(e^{n(\theta-\delta)}). $$ Proposition~\ref{prop:IndProx} shows the existence of independent r.v. $Y_1, Y_2, \ldots , Y_{L_n}$ satisfying \eqref{ind} (with $\mathfrak{s}\geq e^{n\delta}$) for the above choice of $t_1, t_2, \ldots , t_{L_n}$. Using similar ideas as in \eqref{eq:AjBdComp}, we may now write 
\begin{align*}
\Pr(\tilde{\m{A}}^{(j)}_n) &  \le \Pr(\sup_{1\le i\le L_n}\h_{e^{t_i}} \le (n+1)^{2/3}\gamma) \nonumber\\ &  \le \Pr\big(\sup_{1\le i\le L_n}Y_i \le (n+1)^{2/3}\gamma+1\big)+\sum_{i=1}^{L_n}\Pr(Y_i-\h_{e^{t_i}}\ge 1) \nonumber\\ &  \le \prod_{i=1}^{L_n}\Pr\big(Y_i \le (n+1)^{2/3}\gamma+1\big)+a\exp\big(n(\theta-\delta)-c(\exp(e^{n\delta })-1)^{1/2}\big) 
\end{align*}
Now we apply the distributional identity of \eqref{ind} to write
\begin{align*}
\prod_{i=1}^{L_n}\Pr\left(Y_i \le (n+1)^{2/3}\gamma+1\right)  & = \prod_{i=1}^{L_n}\Pr\left((1-e^{-(t_i-t_{i-1})})^{1/3}\h_{e^{t_i}-e^{t_{i-1}}} \le (n+1)^{2/3}\gamma+1\right)\\ &  \le \prod_{i=1}^{L_n}\Pr\big(\h_{e^{t_i}-e^{t_{i-1}}} \le n^{2/3}(\gamma+\eta)\big) \\ & \le \big[1-\exp(-(\eta+\tfrac{4\sqrt{2}}{3})n(\gamma+2\eta)^{3/2})\big]^{L_n} \\ & \le \exp\big(-L_n\exp(-(\eta+\tfrac{4\sqrt{2}}{3})n(\gamma+2\eta)^{3/2})\big)\leq \exp(-e^{n(\theta-\delta)}
e^{-n\zeta}/a)
\end{align*}
where the first inequality follows by noting that $(1-e^{-(t_{i}-t_{i-1})})^{-1/3}((n+1)^{1/3}\gamma+1)\leq n^{1/3}(\gamma+\eta)$ for all large $n$ and the second inequality follows due to \eqref{eq:uptailTlarge} of Proposition~\ref{ppn:LargeTime}. The last inequality follows since $L_n\geq e^{n(\theta-\delta)}/a$ and $\zeta= (\eta+\tfrac{4\sqrt{2}}{3})(\gamma+2\eta)^{3/2}$. Substituting the inequality in the above display into the right hand side of \eqref{eq:AjBoundAlt} completes the proof of \eqref{eq:AjBoundAlt}.

\section{Mono- and Multifractality of the KPZ equation} \label{sec:MuonoMult}

The aim of this section is to prove Theorem~\ref{thm:FracDim}. The monofractality result of the KPZ equation which is stated in \eqref{eq:Monofractal} will be proved in Section~\ref{sec:Monofractal} where the multifractality result of \eqref{eq:Mult} will be proved in Section~\ref{sec:Multifractality}.  

\subsection{Monofractality: Proof of \ref{eq:Monofractal}}\label{sec:Monofractal}

  By the inequality $\mathfrak{LimSup}_{u}$, we know that the limsup of $\h_t/(\log \log t)^{2/3}$ as $t$ goes to $\infty$ is strictly less than $\gamma$ with probability $1$ for any $\gamma>(3/4\sqrt{2})^{2/3}$. This shows $\{t\geq e^{e}: \h_t/(\log \log t)^{2/3}\geq \gamma\}$ is almost surely bounded. From Proposition~\ref{ppn:EssHaus}, it follows that the Barlow-Taylor Hausdorff dimension of a bounded set is zero. This shows $\mathrm{Dim}_{\mathbb{H}}(\{t\geq e^{e}: \h_t/(\log \log t)^{2/3}\geq \gamma\}) \stackrel{a.s.}{=} 0$ for any $\gamma>(3/4\sqrt{2})^{2/3}$. We now proceed to prove $\mathrm{Dim}_{\mathbb{H}}(\{t\geq e^{e}: \h_t/(\log \log t)^{2/3}\geq \gamma\}) = 1$ with probability $1$ for any $\gamma\leq (3/4\sqrt{2})^{2/3}$. For this, it suffices to show that
\begin{align}\label{eq:HausdorffAtCrit}
\operatorname{Dim}_{\mathbb{H}}(\mathcal{P}_{\mathfrak{h}})\stackrel{a.s.}{=}1, \quad\text{where } \mathcal{P}_{\mathfrak{h}}:=\Big\{t\geq e^{e}: \frac{\h_t}{(\log \log t)^{2/3}}\geq \frac{3}{4\sqrt{2}}\Big\}
\end{align}

Throughout the rest of this section, we show \eqref{eq:HausdorffAtCrit}. Denote $\gamma_0:=(3/4\sqrt{2})^{2/3}$. We use the following shorthand notation 
\begin{align}\label{eq:decorr}
\m{A}_s:=\Big\{\frac{\h_{s}}{(\log\log s)^{2/3}}\ge \gamma_0\Big\}, \quad \text{for any }s>0.
\end{align}
For showing \eqref{eq:HausdorffAtCrit}, we need the following two propositions. These two propositions will shed light on the nature of dependence between $\m{A}_t$ and $\m{A}_s$ when $t$ and $s$ are far from each other and $1$-dimensional Hausdorff content (see Definition~\ref{bfHausdorffCont}) of the the set $\mathcal{P}_{\mathfrak{h}}$. We first complete the proof of \eqref{eq:HausdorffAtCrit} using these two propositions and then, those will be proved in two ensuing subsections.

We are now ready to state Proposition~\ref{asymp_ind} which will demonstrate that $\m{A}_t$ and $\m{A}_s$ are approximately independent when $t$ and $s$ are sufficiently far apart. 

\begin{proposition}\label{asymp_ind}
	There exist $T_0>0$, such that for all $t>s\ge T_0$ with \begin{align}\label{cond}
	t \ge s(\log\log t)^{3}(\log\log s +\log\log t)^2,
	\end{align} we have
	\begin{align*} 
	\Pr(\m{A}_s\cap \m{A}_t)= (1+o(1))\Pr(\m{A}_s)\Pr(\m{A}_t).
	\end{align*}
	where $o(1)$ goes to zero as $s,t\to \infty$.
\end{proposition}

The next proposition will investigate $1$-dimensional Hausdorff contents of the set $\mathcal{P}_{\mathfrak{h}}$. 

\begin{proposition}\label{haus}
Denote $\mathcal{V}_n:=[-e^n,e^n]$ and $\mathcal{S}_0:=\mathcal{V}_0, \mathcal{S}_{n+1}:=\mathcal{V}_{n+1}\setminus \mathcal{V}_n$ for $n\in \mathbb{N}$. For any Borel set $G$, define $
\mu(G):=\operatorname{Leb}\left(\mathcal{P}_{\mathfrak{h}}\cap G\right).$ We have
\begin{align}\label{eq:InfinityHaus}
\sum_{n=4}^{\infty}e^{-n}\mu(\mathcal{S}_n)\stackrel{a.s.}{=}\infty.
\end{align}

\end{proposition}	

Assuming Proposition~\ref{asymp_ind} and~\ref{haus}, we proceed to complete the proof of \eqref{eq:HausdorffAtCrit}.
\begin{proof}[Proof of \eqref{eq:HausdorffAtCrit}] 
%
Recall the definition of $\rho$-dimensional Hausdorff content $\nu_{n,\rho}$ from Definition~\ref{bfHausdorffCont}. By Proposition~\ref{ppn:Frostman}, there exists some constant $K_{1,n}>0$ (defined in \eqref{eq:Kdef}) such that $$\nu_{n,1}(\mathcal{P}_{\mathfrak{h}}) \ge K^{-1}_{1,n}e^{-n}\mu(\mathcal{P}_{\mathfrak{h}}).$$ Since $\mu(Q)\leq \mathrm{Leb}(Q)$ for any finite Borel set $Q$, $K_{1,n}$ is less than $1$. This implies $\nu_{n,1}(\mathcal{P}_{\mathfrak{h}}) \ge e^{-n}\mu(\mathcal{P}_{\mathfrak{h}}).$ Combining this inequality with \eqref{eq:InfinityHaus} of Proposition~\ref{haus} yields $$\sum_{n=4}^{\infty} \nu_{n,1}(\mathcal{P}_{\mathfrak{h}}) \ge \sum_{n=4}^{\infty} e^{-n}\mu(\mathcal{S}_n) \stackrel{a.s.}{=}\infty.$$ 
From the definition of the Barlow Taylor Hausdorff dimension (see Definition~\ref{bfHausdorffCont}), it now follows that $\operatorname{Dim}_{\mathbb{H}}(\mathcal{P}_{\mathfrak{h}})= 1$ occurs with probability $1$. This completes the proof.
\end{proof}


\begin{proof}[Proof of Proposition \ref{asymp_ind}]
 By Proposition \ref{ppn:FKG}, we know $\Pr(\m{A}_t\cap \m{A}_s)\geq \Pr(\m{A}_t)\Pr(\m{A}_s)$. It suffices to show that $\Pr(\m{A}_t\cap \m{A}_s)\leq (1+o(1))\Pr(\m{A}_t)\Pr(\m{A}_s)$ as $s,t\to\infty$.   
 
  For showing this, we use Proposition~\ref{prop:IndProx}. Fix $t>s>T_0$ such that $t,s$ satisfy the inequality \eqref{cond}. Note that $(\log \log t)^{-1/2}(\tfrac{t}{s}-1)^{1/3}\to \infty$ as $s,t\to \infty$. By Proposition~\ref{prop:IndProx}, there exists a r.v. $Y$ independent of $\h_s$ and constant $c>0$ such that 
\begin{align}
Y \stackrel{d}{=} \big(1-\frac{s}{t}\big)^{\frac{1}{3}}\h_{t-s}, \quad \Pr\big(|\h_t - Y|\geq (\log \log t)^{-1/2}\big) \leq \exp\big(-c(\tfrac{t}{s}-1)^{1/2}(\log \log t)^{-3/4}\big). \label{eq:YBd}
\end{align}  
  Using the above display and the union bound of the probability, we write 
\begin{align*}
		\Pr(\m{A}_s\cap \m{A}_t)  \le  &\Pr\big(\{\m{A}_s\cap \m{A}_t\}\cap\{ |\h_t-Y|\le (\log\log t)^{-1/2}\}\big)+\Pr\big(|\h_t-Y|\ge (\log\log t)^{-1/2}\big) \\  \le &\Pr\left(\{\h_s  \ge  \gamma_0(\log\log s)^{2/3}\}\cap \{ Y \ge \gamma_0(\log\log t)^{2/3}-(\log\log t)^{-1/2}\} \right) \\ & +\exp\big(-c(t/s-1)^{1/2}(\log \log t)^{-3/4}\big)\\  \le &\Pr\left(\h_s  \ge  \gamma_0(\log\log s)^{2/3})\Pr( Y \ge \gamma_0(\log\log t)^{2/3}-(\log\log t)^{-1/2} \right) \\ & + \exp\big(-c(\log \log t)^{3/4}\log\big(\log s \log t\big)\big)  \\  \le &\Pr\left(\h_s  \ge  \gamma_0(\log\log s)^{2/3}\right)\Pr\left(\h_{t-s} \ge \gamma_0(\log\log t)^{2/3}-(\log\log t)^{-1/2}\right)\\&+ \frac{o(1)}{\log \log t\log t\log \log s\log s}\\
		\leq &\frac{(16\pi \gamma^{3/2}_0)^{-2}(1+o(1))}{\log \log s\log s\log \log t\log t} + \frac{o(1)}{\log \log t\log t\log \log s\log s} = (1+o(1))\Pr(\m{A}_t)\Pr(\m{A}_s)  
		\end{align*}  
  where the inequality in the second line follows by observing that $$\m{A}_t\cap \{|\h_t-Y|\le (\log\log t)^{-1/2}\}\subset \big\{Y \ge \gamma_0(\log\log t)^{2/3}-(\log\log t)^{-1/2}\big\}$$ and using the probability bound in \eqref{eq:YBd}. The next inequality follows by the independence of $\h_s$ and $Y$ and through the following observation 
$$\exp\big(-c(t/s-1)^{1/2}(\log \log t)^{-3/4}\big)\leq \exp\big(-c(\log \log t)^{3/4}\log\big(\log s \log t\big)\big)$$  
  which is a consequence of the fact that $t,s$ satisfy the condition \eqref{cond}. The inequality in the sixth line follows by noting $Y \stackrel{d}{=} \big(1-s/t\big)^{1/3}\h_{t-s}$ and observing 
  $$\exp\big(-c(\log \log t)^{3/4}\log\big(\log s \log t\big)\big)= \frac{o(1)}{\log \log t\log t\log \log s\log s}.$$
   The last inequality follows by using Lemma~\ref{lem:RefUpTail}. This completes the proof of Proposition~\ref{asymp_ind}. 
	\end{proof}

\begin{proof}[Proof of Proposition~\ref{haus}]
Fix $\e>0$. Let $N_0=N_0(\e)>T_0$ be such that for any $t,s \geq e^{N_0}$ satisfying \eqref{cond}, we have 
\begin{align*}
\frac{(1-\e)}{\log s\log \log s}\leq \frac{\Pr(\m{A}_s)}{(16\pi\gamma^{3/2}_0)^{-1}}\leq \frac{(1+\e)}{\log s\log \log s}, \quad &  \quad \quad  \frac{(1-\e)}{\log t\log\log t}\leq \frac{\Pr(\m{A}_t)}{(16\pi\gamma^{3/2}_0)^{-1}}\leq \frac{(1+\e)}{\log t\log \log t} ,\\  (1-\e)\leq &\frac{\Pr(\m{A}_t\cap \m{A}_s)}{\Pr(\m{A}_t)\Pr(\m{A}_s)}\leq (1+\e).
\end{align*}
Note that the first two inequality holds due to Proposition~\ref{lem:RefUpTail} of Section~\ref{sec:App} and the last inequality holds due to Proposition~\ref{asymp_ind} which we just proved. Next we define a subsequence $\{N_k\}$ recursively as follows:
\begin{align}
N_1:=\max\{N_0, e^{e^{10}}\}, \quad N_{k+1} = N_{k}+ 10 \log\log N_k, \quad k\in \mathbb{N}. 
\end{align} 
Consider the following random variables
\begin{align*}
S_{M}:=\sum_{k=1}^{M}e^{-N_k}\mu(\mathcal{S}_{N_k}), \quad M \in \mathbb{N}, \qquad S_{\infty} := \sum_{k=1}^{\infty} e^{-N_k}\mu(\mathcal{S}_{N_k}).
\end{align*}
Define $\kappa:=(1-e^{-1})$. For $M\in \mathbb{N}$, we show that
\begin{align}\label{eq:SnmBd}
\frac{\Ex[S_{M}]}{1-\e}\geq (1+o(1))\frac{\kappa\log\log\log  N_{M}}{16\pi\gamma^{3/2}_0}, \qquad  \frac{\Ex [S_{N,M}^2]- (1+\e)(\Ex[S_{N,M}])^2}{(1+\e)(1+o(1))}\le \frac{\kappa^2\log\log\log  N_{M}}{80\pi\gamma^{3/2}_0}
\end{align}
where the term $o(1)$ goes to $0$ as $M$ goes to $\infty$.

By assuming  the above inequality, we first complete the proof of Proposition~\ref{haus}. We seek to show $\Pr(S_{\infty}= \infty)\geq 1$. Note that $\mathbb{P}(S_{\infty}=\infty)=\mathbb{P}(\lim_{M\to \infty} S_{M}=\infty)$. We may now write 
\begin{align}
\Pr\big(\lim_{M\to \infty }S_{M}=\infty\big) & \ge 
\liminf_{M\to\infty}\Pr\big(S_{M}\ge \tfrac{1}{\sqrt{\log\log \log N_M}}\Ex S_{N,M}\big) \nonumber\\ & \ge \liminf_{M\to\infty}\frac{(1-(\log\log \log N_M)^{-1/2})^2(\Ex[S_{N,M}])^2}{\Ex[S^2_{N,M}]} \nonumber\\ & \ge \liminf_{M\to\infty}\frac{(1-(\log \log \log N_M)^{-1/2})^2\cdot\tfrac{(1-\e)^2}{1+\e}}{(1+o(1))4\pi \gamma^{3/2}_0(\log \log \log N_M)^{-1}+1}=\frac{(1-\e)^2}{1+\e}\label{eq:SnmPrBd}
\end{align}
The first inequality in the above display follows since $(\log\log \log N_M)^{-1/2}\Ex [S_{M}]$ converges to $\infty$ thanks to the first inequality of \eqref{eq:SnmBd}. We obtained the second inequality by applying Paley-Zygmund inequality (see Proposition~\ref{ppn:Paley-Zygmund}) for the random variable $S_{M}$ with setting $\delta :=(\log\log\log N_M)^{-1/2}$. The third inequality follows by noticing from \eqref{eq:SnmBd} that $$ \frac{\Ex [S_{M}^2]}{(\Ex[S_{M}])^2}\leq \frac{1+\e}{(1-\e)^2}\Big((1+o(1))4\pi \gamma^{3/2}_0(\log \log \log N_M)^{-1}+1\Big).$$ From \eqref{eq:SnmPrBd}, Proposition~\ref{haus} follows by letting $\e$ to $0$ and observing that $S_{\infty} \le \sum_{n=4}^{\infty}e^{-n}\mu(\mathcal{S}_n)$.

Throughout the rest, we prove \eqref{eq:SnmBd}. Note that 
\begin{align*}
\Ex[S_{M}]  =\sum_{k=1}^M e^{-N_k}\int_{e^{N_{k}-1}}^{e^{N_k}}\Pr(\m{A}_s)\d s  \ge \sum_{k=1}^M e^{-N_k}\int_{e^{N_k-1}}^{e^{N_k}}\frac{(1-\e)(16 \pi\gamma^{3/2}_0)^{-1}}{\log \log s\log s}\d s \ge \sum_{k=1}^M\frac{\kappa(16 \pi\gamma^{3/2}_0)^{-1}}{N_k\log N_k}.
\end{align*}
where the first inequality follows since $\Pr(\m{A}_s)\geq (1-\e)(16\pi \gamma^{3/2}_0\log s\log \log s)^{-1}$ for all $s\geq e^{N_0}$ and the second inequality follows since $\log s\leq N_k$ for all $s\in [e^{N_k-1}, e^{N_k}]$. To lower bound the right hand side of the above display, we note   
\begin{align*}
\sum_{n=N_{k-1}}^{N_k-1}\frac{1}{n\log n \log \log n} \leq \sum_{n=N_{k-1}}^{N_k-1} \frac{(\log \log N_{k-1})^{-1}}{N_{k-1}\log N_{k-1}}  \leq \frac{10\log \log N_{k-1}}{N_{k-1}\log N_{k-1}\log \log N_{k-1}} = \frac{10}{N_{k-1}\log N_{k-1}}
\end{align*} 
where the first inequality is straightforward and the second  inequality follows since $|N_{k}-N_{k-1}|= 10(1+o(1))\log \log N_{k-1}$. Using the above display, we may write 
\begin{align}
\sum_{k=1}^{M}\frac{1}{N_k\log N_k} \geq \frac{1}{10} \sum_{n=N_0}^{N_k} \frac{1}{n\log n\log \log n} = \frac{1}{10}(1+o(1))\log \log \log N_M
\end{align}
where $o(1)$ goes to $0$ as $M$ goes to $\infty$. This implies the first inequality of \eqref{eq:SnmBd}.


Now we proceed to prove the second inequality of \eqref{eq:SnmBd}. Define $\mathfrak{Int}(n,m) := e^{-n-m}\int_{e^{n-1}}^{e^n}\int_{e^{m-1}}^{e^m} \Pr(\m{A}_t\cap \m{A}_s)\d t \d s$. Observe that  

\begin{align*}
\Ex[S_{M}^2]= \sum_{k=1}^M\sum_{\ell=1}^M  \mathfrak{Int}(N_k,N_{\ell}) = \underbrace{\sum_{k=1}^M \mathfrak{Int}(N_{k},N_k)}_{(\mathbf{I})}  + \underbrace{\sum_{k\neq \ell}^M \mathfrak{Int}(N_k,N_\ell)}_{(\mathbf{II})}.
\end{align*}
 We first bound $(\mathbf{I})$. Using the inequality $\Pr(\m{A}_s\cap \m{A}_t)\le \Pr(\m{A}_t)\le (1+\e)(16\pi \gamma^{3/2}_0\log\log t\log t)^{-1}$ for any $s, t\in [e^{N_k-1}, e^{N_k}]$, we see
\begin{align*}
\int_{e^{N_k-1}}^{e^{N_k}}\int_{e^{N_{k}-1}}^{e^{N_k}} \Pr(\m{A}_s\cap \m{A}_t)\d s \d t \le \frac{(1+\e)(e^{N_k}-e^{N_k-1})^2}{16\pi \gamma^{3/2}_0(N_k-1)\log(N_k-1)}. 
\end{align*}

Multiplying both sides by $e^{-2N_{k}}$ and summing over $k$ as $k$ varies in $[1,M]\cap \mathbb{Z}_{\geq 1}$ yields 
\begin{align}
(\mathbf{I}) \le \sum_{k=1}^M \frac{\kappa^2(1+\e)}{16\pi \gamma^{3/2}_0(N_{k}-1)\log(N_{k}-1)}&\leq \frac{\kappa^2(1+\e)}{80\pi \gamma^{3/2}_0}\sum_{n=N_0}^{N_{M}}\frac{C}{(n-1)\log(n-1)\log \log (n-1)}\nonumber\\&= \kappa^2(1+o(1))\frac{(1+\e)\log \log \log N_M}{80\pi \gamma^{3/2}_0}.\label{eq:IandII}
\end{align}
The equality in the last line follows since  $\sum_{n=N_0}^{N_M}((n-1)\log(n-1)\log \log (n-1))^{-1} = (1+o(1))\log \log \log (n)$. It remains to explain the second inequality of the above display. To see this, notice that for any $k\in \mathbb{N}$, 
\begin{align*}
\frac{1}{(N_k-1)\log (N_k-1)}  \leq \frac{2\log \log(N_{k-1})}{(N_{k}-1)\log(N_k-1)\log \log (N_{k}-1)}\leq 2\sum_{n=N_{k-1}}^{N_k} \frac{(\log \log (n-1))^{-1}}{10(n-1)\log(n-1)}
\end{align*}
The first inequality follows since $2\log \log N_{k-1}\geq \log \log (N_{k}-1)$ whereas the second inequality is obtained by noting that $|N_{k} - N_{k-1}|\leq 10 \log \log N_{k-1}$.

Now we bound $(\mathbf{II})$. Fix any $t\in [e^{N_k-1}, e^{N_k}]$ and $s\in [e^{N_{\ell}-1}, e^{N_{\ell}}]$ for $k> \ell \in \mathbb{N}$. Using this information, we write
$$t/s\geq e^{N_{k}-N_{\ell}-1}\geq e^{10 \log \log (N_{k-1})-1}\geq e^{5\log \log (N_k)+\log 4}= 4(\log N_k)^5\geq 4(\log \log t)^5$$
where the second inequality follows since $N_k-N_{\ell}\geq N_k-N_{k-1}\geq 10 \log \log N_{k-1}$ and the third inequality is obtained by noting that $10\log \log (N_{k-1})\geq 5 \log \log N_k+1+\log 4$. From the above display, it follows that $t$ and $s$ satisfy \eqref{cond}.
 Due to Proposition~\ref{asymp_ind}, we have $\Pr(\m{A}_t\cap \m{A}_s)\leq(1+\e)\Pr(\m{A}_t)\Pr(\m{A}_s)$. This implies  
\begin{align}
(\mathbf{II}) & = 2\sum_{\ell=1}^M\sum_{k=\ell+1}^{M}e^{-N_k-N_\ell}\int_{e^{N_k-1}}^{e^{N_k}}\int_{e^{N_{\ell}-1}}^{e^{N_{\ell}}} \Pr(\m{A}_s\cap \m{A}_t)\d s \d t  \nonumber\\ & \le 2(1+\e)\sum_{\ell=1}^M\sum_{k=\ell+1}^{M}e^{-N_k-N_\ell}\int_{e^{N_k-1}}^{e^{N_k}}\int_{e^{N_{\ell}-1}}^{e^{N_{\ell}}} \Pr(\m{A}_s)\Pr(\m{A}_t)\d s \d t \le (1+\e)\big(\Ex[S_{M}]\big)^2.\label{eq:III}
\end{align}
Combining \eqref{eq:IandII} and \eqref{eq:III} yields \eqref{eq:SnmBd}. This completes the proof.
\end{proof}

\subsection{Multifractality: Proof of \ref{eq:Mult}}\label{sec:Multifractality}

 Recall the definition of the exponential time changed process $\mur(t)$. We use the following shorthand notation throughout this section-
\begin{align}\label{eq:Lambda}
\Lambda_{\gamma}:= \left\{t\ge e \mid \mur(t)\ge \gamma\left(\frac{3}{4\sqrt{2}}\log t\right)^{2/3}\right\}, \quad \gamma\in \mathbb{R}.
\end{align}
Due to Theorem~\ref{thm:MainTheorem}, we know 
$$\limsup_{t\to \infty} \frac{\mur(t)}{(3\log t/4\sqrt{2})^{2/3}} \stackrel{a.s.}{=} 1$$
which shows that $\Lambda_{\gamma}$ is almost surely bounded for $\gamma>1$. This proves $\mathrm{Dim}_{\mathbb{H}}(\Lambda_{\gamma}) =0$ with probability $1$ when $\gamma>1$. In the rest of the section, we focus on showing \eqref{eq:Mult} for $\gamma\in (0,1]$. We divide the proof into two stages. The first stage will show the upper bound $\mathrm{Dim}_{\mathbb{H}}(\Lambda_{\gamma})\leq 1 - \gamma^{3/2}$ and the lower bound $\mathrm{Dim}_{\mathbb{H}}(\Lambda_{\gamma})\geq 1 - \gamma^{3/2}$ will be shown in the second stage.  

\medskip


\noindent \textbf{Stage 1: Proof of $\mathrm{Dim}_{\mathbb{H}}(\Lambda_{\gamma})\leq 1-\gamma^{3/2}$.} 
Recall the definition of $\rho$-dimensional Hausdorff content $\nu_{n,\rho}$ from Definition~\ref{bfHausdorffCont}. The main step of the proof of the upper bound is to show that 
\begin{align}\label{eq:ExpSer}
\sum_{n=1}^{\infty}\mathbb{E}\big[\nu_{n,\rho}\big(\Lambda_{\gamma}\big)\big]<\infty, \quad \forall \rho> 1-(1-\epsilon)\gamma^{3/2}, \epsilon\in (0,1)
\end{align}
This immediately implies that $\sum_{n=1}^{\infty}\nu_{n,1-(1-\epsilon)\gamma^{3/2}}(\Lambda_\gamma)<\infty$ almost surely for all $\epsilon\in (0,1)$ and hence, $\mathrm{Dim}_{\mathbb{H}}(\Lambda_{\gamma}) \leq 1-(1-\epsilon)\gamma^{3/2}$. From this upper bound, the result will follow by taking $\epsilon$ to $0$. Below, we state a lemma showing a technical estimate which will be required to bound $\mathbb{E}\big[\nu_{n,1-(1-\epsilon)\gamma^{3/2}}\big(\Lambda_{\gamma}\big)\big]$ for any $n\in \mathbb{N}$. After that, we will proceed to complete the proof of the upper bound which will be followed by the proof of the lemma.


\begin{lemma}\label{tech} Fix $\z\in (0,1)$. We have
\begin{align}\label{eq:PrLambda}
\Pr\left(\Lambda_\gamma\cap [m,m+1]\neq \varnothing \right) \le 2m^{-(1-\z)^{3/2}\gamma^{3/2}+o(1)}\log m
\end{align}
where $o(1)$ term goes to zero as $m\to \infty$. 
\end{lemma}

\begin{proof}[Final steps of the upper bound proof] Fix $\z>0$ and take any $\rho> 1-(1-\e)^{3/2}\gamma^{3/2}$. For any $n\geq 1$, define $\Xi_n:= [-e^{n+1},-e^{n})\cup (e^{n},e^{n+1}]$. From the definition of $\nu_{n,\rho}$, it follows that 
$$\nu_{n,\rho}(\Lambda_{\gamma}) \le \sum_{m\in \Z_{>0}} \frac{1}{e^{n\rho}}\ind_{\Lambda_{\gamma}\cap [m,m+1] \neq \varnothing}\cdot \ind_{[m,m+1]\subset \Xi_n}.$$
 Taking expectation on both sides, we get
	\begin{align}
	\Ex\left[\nu_{n,\rho}(\Lambda_{\gamma})\right] & \le e^{-n\rho} \sum_{m\in \Z_{>0} }\ind_{[m,m+1]\in \Xi_n}\cdot\Pr\left(\Lambda_{\gamma}\cap [m,m+1] \neq \varnothing\right)\nonumber \\ & \le e^{-n\rho}\sum_{m\in \Z_{>0}} 2m^{-(1-\z)^{3/2}\gamma^{3/2}+o(1)}\log m\cdot \ind_{[m,m+1]\in \Xi_n}  \nonumber\\ & \le e^{-n\rho}\cdot 2e^{n+1} \cdot 2n e^{-(1-\z)^{3/2}\gamma^{3/2}n}= 4ne^{n(1-\rho-(1-\z)^{3/2}\gamma^{3/2})+1}.\label{eq:UpBd}
	\end{align}
 The second inequality follows from Lemma \ref{tech}. We get the third inequality by observing that the number of non-zero terms in the sum is bounded by $2e^{n+1}$ and each non-zero term is bounded above by $2n e^{-(1-\z)^{3/2}\gamma^{3/2}n}$. The upper bound of $\Ex[\nu_{n,\rho}(\Lambda_{\gamma})]$ in \eqref{eq:UpBd} is summable over $n$ whenever $\rho>1-(1-\z)\gamma^{3/2}$. This shows \eqref{eq:ExpSer}. Alluding to the discussion after \eqref{eq:ExpSer}, we get the proof of $\mathrm{Dim}_{\mathbb{H}}(\Lambda_{\gamma})\leq 1-\gamma^{3/2}$.

\end{proof}

\begin{proof}[Proof of Lemma \ref{tech}] Define $B_m:=\lceil\log m \rceil $.  We divide the interval $[e^m,e^{m+1}]$ into $B_m$ many intervals $\{\mathcal{I}^{m}_{j}\}^{m}_{j=1}$ where $$\mathcal{I}^{m}_{j}:= [x^{(m)}_{j-1}, x^{(m)}_{j}] \quad \text{and, } \quad x^{(m)}_{j}:=e^m(1+\tfrac{(e-1)j}{B_m}), \quad j=1,\ldots, B_m.$$
We may now write 
	\begin{align}
	\Pr(\Lambda_\gamma\cap [m,m+1]\neq \varnothing ) &\le \Pr\Big(\sup_{t\in [m,m+1]} \mur(t)\geq \gamma\big(\frac{3}{4\sqrt{2}}\log m\big)^{\frac{2}{3}}\Big)\nonumber\\ & = \Pr \Big( \sup_{t\in [e^{m},e^{m+1}]} \h_t \ge \gamma\big(\frac{3}{4\sqrt{2}}\log m\big)^{\frac{2}{3}} \Big) \nonumber\\ & \le \sum_{j=1}^{B_m} \Pr \Big( \sup_{t\in \mathcal{I}^{m}_{j}} \h_t \ge \gamma\big(\frac{3}{4\sqrt{2}}\log m\big)^{2/3} \Big). \label{eq:trm3}
	\end{align}
	where the last inequality follows by the union bound. In what follows, we show that 
 \begin{align}\label{eq:EachTermBd}
 \Pr \Big( \sup_{t\in \mathcal{I}^{m}_{j}} \h_t \ge \gamma\big(\frac{3}{4\sqrt{2}}\log m\big)^{2/3} \Big)\leq 2m^{-(1-\z)^{3/2}\gamma^{3/2}+o(1)}
\end{align} 	
	where $o(1)$ term converges to $0$ as $m$ goes to $\infty$ uniformly for all $j=1,\ldots , B_m$. From \eqref{eq:EachTermBd}, \eqref{eq:PrLambda} of Lemma~\ref{tech} will follow by noting that there are at most $\log m$ terms in the sum \eqref{eq:trm3}.

	Fix $j\in \{1,\ldots , B_m\}$. For convenience, we use shorthand $x_j$ and $x_{j-1}$ to denote $x^{(m)}_{j}$ and $x^{(m)}_{j-1}$ respectively. Consider the events
	\begin{align*}
	\m{A}_{j,m}  := \Big\{\sup_{t\in \mathcal{I}^{m}_{j}} \big(\h_t-\big(\tfrac{x_{j-1}}{t}\big)^{\frac{1}{3}}\h_{x_{j-1}}\big) \ge \z\gamma\big(\frac{3\log m}{4\sqrt{2}}\big)^{\frac{2}{3}} \Big\}, \quad \m{B}_{j,m}  := \Big\{\h_{x_{j-1}} \ge (1-\z)\gamma\big(\frac{3\log m}{4\sqrt{2}}\big)^{\frac{2}{3}}\Big\}. 
	\end{align*}	
	 Note that 
	\begin{align*}
	\sup_{t\in \mathcal{I}^{m}_j} \h_t & \le \sup_{t\in \mathcal{I}^{m}_j} \Big(\h_t-\big(\tfrac{x_{j-1}}{t}\big)^{1/3}\h_{x_{j-1}}\Big) +\sup_{t\in \mathcal{I}^{m}_j} \left(\tfrac{x_{j-1}}{t}\right)^{1/3}\h_{x_{j-1}} \\ & = \sup_{t\in \mathcal{I}^{m}_j} \Big(\h_t-\big(\tfrac{x_{j-1}}{t}\big)^{1/3}\h_{x_{j-1,m}}\Big) +\max \Big\{\big(\tfrac{x_{j-1}}{x_{j}}\big)^{1/3}\h_{x_{j-1}}, \h_{x_{j-1}} \Big\}.
	\end{align*}
Due to the above inequality, we have 
	\begin{align*}
	\Big\{\sup_{t\in \mathcal{I}^{m}_j} \h_t \ge \gamma\big(\tfrac{3}{4\sqrt{2}}\log m\big)^{2/3} \Big\} \subset \m{A}_{j,m} \cup \m{B}_{j,m}
	\end{align*}
By the union bound, we get 
\begin{align}\label{eq:EachBdSplit}
\Pr \Big( \sup_{t\in \mathcal{I}^{m}_{j}} \h_t \ge \gamma\big(\frac{3}{4\sqrt{2}}\log m\big)^{2/3} \Big) \leq \Pr(\m{A}_{j,m})+ \Pr(\m{B}_{j,m}).
\end{align}
In what follows, we claim and prove that 
\begin{align}\label{eq:AjBjBd}
m^{\gamma^{3/2}}\Pr(\m{A}_{j,m}) = o(1), \quad \text{and }\quad\Pr(\m{B}_{j,m}) = m^{-(1-\z)^{3/2}\gamma^{3/2}+o(1)}
\end{align}
where $o(1)$ terms converge to $0$ as $m\to \infty$ uniformly for all $j \in \{1,\ldots , B_m\}$. Substituting the bounds of \eqref{eq:AjBjBd} into the right hand side of \eqref{eq:EachBdSplit} shows \eqref{eq:EachBdSplit}. To complete the proof of this lemma, it suffices to to show \eqref{eq:AjBjBd}.   	
	
	By noting that $\log x_{j-1,m} \in [m,m+1]$, we use Proposition~\ref{lem:RefUpTail} to get 
	\begin{align*}
	\Pr(\m{B}_{j,m}) & \le \exp\left(-(1+o(1))\gamma^{3/2}(1-\z)^{3/2}\log m\right)  = m^{-(1-\z)^{3/2}\gamma^{3/2}+o(1)}
	\end{align*}
	where the $o(1)$ term goes to zero as $m\to \infty$ uniformly for all $j$. This proves the bound on $\Pr(\m{B}_{j,m})$ in \eqref{eq:AjBjBd}.

	Now we proceed to prove the bound on $\Pr(\m{A}_{j,m})$.
 To this end,  recall that $h_{t}(\alpha,0)={\alpha}^{1/3}h_{\alpha t}$ for any $\alpha, t>0$. Using this, we may write  
		\begin{align} \nonumber
		\Pr(\m{A}_{j,m}) & = \Pr\Big(\sup_{t\in \mathcal{I}^{j}_m} \big(\tfrac{x_{j-1}}{t}\big)^{1/3} \big(\h_{x_{j-1}}(\tfrac{t}{x_{j-1}},0)-\h_{x_{j-1}}(1,0)\big) \ge \z\gamma\big(\tfrac{3}{4\sqrt{2}}\log m\big)^{2/3} \Big) \\ & \le \Pr\Big(\sup_{\tau\in [0,\frac{e-1}{B_m}]} \big(\h_{x_{j-1}}(1+\tau,0)-\h_{x_{j-1}}(1,0)\big) \ge \z\gamma(\tfrac{x_{j}}{x_{j-1}})^{1/3}\big(\tfrac{3}{4\sqrt{2}}\log m\big)^{2/3} \Big) \label{la}
		\end{align}
		where the second inequality follows since $(t^{-1}x_{j-1})^{1/3}$ is bounded below by $(x^{-1}_jx_{j-1})^{1/3}$ for any $t\in \mathcal{I}^{m}_j$. Setting  $r:=\sup_{\tau\in (0,(e-1)/B_m]} \tau^{1/8}\log^{2/3} (1/\tau)<\infty$, we get
		\begin{align}
		\mbox{r.h.s.~of \eqref{la}} & \le \Pr\Big(\sup_{\tau\in [0,\frac{e-1}{B_m}]} \frac{\h_{x_{j-1}}(1+\tau,0)-\h_{x_{j-1}}(1,0)}{\tau^{1/8}\log^{2/3}(1/\tau)} \ge \frac{\z\gamma}{r}(\tfrac{x_{j}}{x_{j-1}})^{1/3}\big(\tfrac{3}{4\sqrt{2}}\log m\big)^{2/3} \Big) \label{la2}
		\end{align}
		Applying Proposition \ref{temp-modulus} with $\e=\frac{1}{8}$, $\delta=\frac{e-1}{B_m}$, $a=1$, we get
		\begin{align*}
		\mbox{r.h.s.~of \eqref{la2}} & \le \exp\Big(-c(\frac{\z \gamma}{r})^{\frac{3}{2}}(\tfrac{x_{j}}{x_{j-1}})^{\frac{1}{2}}\big(\tfrac{3}{4\sqrt{2}}\log m\big)\Big) \le \exp(-C(\log m)^{1+\frac{3}{32}})=o(m^{-\gamma^{3/2}})
		\end{align*}
 for all large $m$. Here, $C$ is a constant which will only depend $\epsilon$. The second inequality follows since $r^{-\frac{3}{2}}\geq c_1(\log m)^{\frac{3}{32}}$ for some $c_1>0$ and $(x_j/x_{j-1})\geq 1$. This proves the first bound in \eqref{eq:AjBjBd} and hence, completes the proof of the lemma.
%
\end{proof}

\medskip

\noindent \textbf{Stage 2: Proof of $\mathrm{Dim}_{\mathbb{H}}(\Lambda_{\gamma})\geq 1-\gamma^{3/2}$.} To prove the lower bound, we use similar techniques used as in \cite[(4.14) of Theorem 4.7]{KKX17}. Recall the definition of `thickness' of a set from Definition~\ref{bd:Thickness}. We seek to use to use Proposition~\ref{thick}. 

 Let us fix $\theta \in (\gamma^{3/2},1)$. Recall $\Lambda_{\gamma}$ from \eqref{eq:Lambda}. We will show that $\Lambda_{\gamma}$ is $\theta$-thick  with probability $1$. This will prove the lower bound $\mathrm{Dim}_{\mathbb{H}}(\Lambda_{\gamma})\stackrel{a.s.}{\geq} 1-\gamma^{3/2}$ via Proposition~\ref{thick}. Let us define
	$$\m{D}_n:=\left\{ \Lambda_{\gamma} \cap [x,x+e^{\theta n}]=\varnothing, \mbox{ for some }x\in \Pi_n(\theta) \right\}.$$
The $\theta$-thickness of $\Lambda_{\gamma}$ will follow through the Borel-Cantelli lemma if the following holds 
\begin{align}\label{eq:Dn}
\sum_{n=1}^{\infty} \Pr(\m{D}_n) < \infty.
\end{align}  
Showing the above display will be the main focus of the rest of the proof.

Recall the definition of the interval $\mathcal{I}_n$ and its $\mathcal{M}_{\theta}$ many sub-intervals $\{\mathcal{I}^{(j)}_n\}^{\mathcal{M}_{\theta}}_{j=1}$ from Section~\ref{sec:LimInfu}. Let us denote the end points of the sub-intervals $\{\mathcal{I}^{(j)}_n\}^{\mathcal{M}_{\theta}}_{j=1}$  as $x^{(1)}_n,\ldots ,x^{(\mathcal{M}_{\theta})}_{n}$ such that $\mathcal{I}^{(j)}_n= [\exp(x^{(j-1)}_n), \exp(x^{(j)}_n)]$. Let us define 
\begin{align*}
\m{B}^{(j)}_n:= \Big\{\sup_{t\in [x^{(j-1)}_n, x^{(j)}_n]}\mur(t)\leq \gamma \big(\frac{3}{4\sqrt{2}}\big)^{\frac{2}{3}}(n+1)^{\frac{2}{3}}\Big\}.
\end{align*}
From the definition of $\m{B}^{(j)}_n$, it immediately follows that $\m{D}_n\subset \cup_{j=1}^{\mathcal{M}_{\theta}}\m{B}^{(j)}_n. $
By the union, we get $\Pr(\m{D}_n)\leq \sum_{j}\Pr(\m{B}^{(j)}_n)$. We will now show \eqref{eq:Dn} by proving a bound (uniform on $j$ and $n$) on $\Pr(\m{B}^{(j)}_n)$.

  Choose $\eta>0$ such that $\zeta := (\tfrac{4\sqrt{2}}{3}+\eta)(\gamma(\tfrac{3}{4\sqrt2})^{\frac23}+ 2\eta)^{3/2}<\theta$ and pick $\delta\in (0,\theta-\zeta)$. We now claim and prove that there exists $c_1,c_2>0$ such that  
\begin{align}\label{eq:BjBound}
\Pr(\m{B}^{(j)}_{n})\leq \exp(-ce^{n(\theta-\delta)}e^{-n\zeta}) + \exp\big(n(\theta-\delta)-c_2(\exp(e^{n\delta})-1)^{1/2}\big)
\end{align}
for all $1\leq j\leq n$ and all large $n$. Using the above inequality, we may bound $\Pr(\m{D}_n)$ by $\exp(n-n\theta+1)$  times the right hand side of \eqref{eq:BjBound}. Since $\theta>\zeta+\delta$, we can bound $\Pr(\m{D}_n)$ by $\exp(- e^{n\omega})$ for some $\omega\in (0,1)$ and for all large $n$. This shows \eqref{eq:Dn} and hence, completes the proof modulo \eqref{eq:BjBound} which is finally remained to be shown. By the identification $\mur(t)= h_{e^{t}}$ and $\mathcal{I}^{(j)}_n= [\exp(x^{(j-1)}_n), \exp(x^{(j)}_n)]$, it is straightforward to see that 
$$\m{B}^{(j)}_{n} = \Big\{\sup_{t\in \mathcal{I}^{(j)}_n}\h_{t}\leq \gamma \big(\frac{3}{4\sqrt{2}}\big)^{\frac{2}{3}}(n+1)^{\frac{2}{3}}\Big\}.$$  
Due to this identity, \eqref{eq:BjBound} now follows from the proof of \eqref{eq:AjBoundAlt}. This completes the proof. 
\qed

\appendix
\section{Auxiliary results}\label{sec:App}

In this section, we will show an improved bound on the upper tail probability of the KPZ equation as time goes to $\infty$. This is used in Section~\ref{sec:MuonoMult} for showing the macroscopic fractal properties of the KPZ equation. 
\begin{proposition}\label{lem:RefUpTail}  Recall $\m{A}_t$ and $\gamma_0$ from \eqref{eq:decorr}. Define $b_t:=(\log\log t)^{-7/6}$. Then, for any constant $K>0$,
	\begin{align*}
	\Pr(\m{A}_t)=\frac{(16\pi)^{-1}(1+o(1))}{\gamma^{3/2}_0\log t\log\log t}, \quad \quad
	\Pr\Big(\frac{\h_t}{(1+Kb_t)} \ge \gamma_0(\log\log t)^{2/3}\Big)= \frac{(16\pi)^{-1}(1+o(1))}{\gamma^{3/2}_0\log t\log\log t}.
	\end{align*}
	where $o(1)$ term converges to $0$ as $t$ goes to $\infty$. 
\end{proposition}

Our proof of Proposition~\ref{lem:RefUpTail} is closely in line with the proof of Proposition~4.1 of \cite{CG18b}. It will use a Laplace transform formula for $\mathcal{Z}^{\mathbf{nw}}(T,0)$ proved in \cite{BorGor16}. It connects $\mathcal{Z}^{\mathbf{nw}}(T,0)$ with the Airy point process $\mathbf{a}_1> \mathbf{a}_2>\ldots $, a well studied determinantal point process in random matrix theory (see, e.g., \cite[Section~4.2]{AGZ10}).

Throughout the rest, we use the following shorthand notations.
\begin{align*}
\mathcal{I}_s(x) := \frac{1}{1+\exp(t^{\frac{1}{3}}(x-s))} , \qquad \mathcal{J}_s(x):= \log \big(1+\exp(t^{\frac{1}{3}}(x-s))\big).
\end{align*}
\begin{proposition}[Theorem~1 of \cite{BorGor16}]\label{ppn:PropConnection}
For all $s\in\mathbb{R}$,
\begin{align}\label{eq:Connection}
\mathbb{E}_{\mathrm{KPZ}}\Big[\exp\Big(-\exp\big(t^{\frac{1}{3}}(\h_t(0)-s)\big)\Big)\Big]=\mathbb{E}_{\mathrm{Airy}}\left[\prod_{k=1}^{\infty} \mathcal{I}_s(\mathbf{a}_k)\right].
\end{align}
\end{proposition}

The following proposition proves an upper and lower bound on the r.h.s. of \eqref{eq:Connection}. We use these bounds to complete the proof of Proposition~\ref{lem:RefUpTail}. We defer the proof of Proposition~\ref{thm:MainTheorem} to Section~\ref{NWLaplaceTail}. 

\begin{proposition}\label{thm:MainTheorem}
Fix any constant $K_1,K_2,K_3>0$. Recall $b_t $ from Proposition~\ref{lem:RefUpTail}. There exists $t_0= t_0(K_1,k_2,K_3)>0$ and two sequences $\{\mathfrak{p}_t\}_{t\geq t_0}$, $\{\mathfrak{q}_t\}_{t\geq t_0}$ such that for all $t\geq t_0$, $ K_1(\log\log t)^{2/3}\leq s\leq K_2(\log\log t)^{2/3}$ and $K\in [-K_3,K_3]$,
\begin{align}
1- \mathbb{E}\Big[\prod_{k=1}^{\infty} \mathcal{I}_{(1+Kb_t)s}(\mathbf{a}_k)\Big]&\leq (1+\mathfrak{p}_t)\frac{1}{16\pi s^{3/2}}e^{- \frac{4}{3}s^{3/2}},\label{eq:UpBound}\\
1- \mathbb{E}\Big[\prod_{k=1}^{\infty} \mathcal{I}_{(1+Kb_t)s}(\mathbf{a}_k)\Big]&\geq (1+\mathfrak{q}_t)\frac{1}{16\pi s^{3/2}}e^{- \frac{4}{3} s^{3/2}}\label{eq:LowrBound}
\end{align}
and $\mathfrak{p}_t\to 0$, $\mathfrak{q}_t\to 0$ as $t\to \infty$.
\end{proposition}

\begin{proof}[Proof of Proposition~\ref{lem:RefUpTail}] 
Define $s:=\gamma_0(\log \log t)^{2/3}$, $\bar{s}:= \gamma_0(1+b_t)(\log \log t)^{2/3}$ and $\theta(s) := \exp\big(- \exp\big(t^{\frac{1}{3}}(\h_t-s)\big)\big)$. By \eqref{eq:Connection}, we know $\mathbb{E}_{\mathrm{KPZ}}[\theta(s)] = \mathbb{E}_{\mathrm{Airy}}[\prod_{k=1}^{\infty} \mathcal{I}_s(\mathbf{a}_k)]$. Note that
\begin{align*}
\theta(s)\leq \mathbf{1}(\h_t(0)\leq \bar{s})+\mathbf{1}(\h_t(0)> \bar{s}) \exp(-\exp(b_t st^{1/3}))
\end{align*}
which after rearranging, taking expectations and applying \eqref{eq:Connection} will lead to 
\begin{align*}
\mathbb{P}(\h_t(0)> \bar{s})\leq  \Big(1-\exp(-\exp(b_t st^{\frac{1}{3}}))\Big)^{-1} \Big(1-\mathbb{E}_{\mathrm{Airy}}\big[\prod_{k=1}^{\infty} \mathcal{I}_s(\mathbf{a}_k)]\Big).
\end{align*}
 We may write $1-\exp(-\exp(b_t s t^{\frac{1}{3}}))=1+ o(t)$. Combining this with \eqref{eq:UpBound} yields
\begin{align*}
\mathbb{P}(\h_t(0)\geq \bar{s}) \leq (1+o(1))\frac{1}{16\pi s^{3/2}}e^{-\frac{4}{3}s^{3/2}}
\end{align*}
   for all large $t$. 

We turn now to prove the lower bound. By Markov's inequality, we get
\begin{align*}
\mathbb{P}(\h_t\leq s)= \mathbb{P}\Big(\theta(\bar{s})\geq \exp\big(- e^{-b_t st^{1/3}}\big)\Big)\leq \exp\big(e^{-b_t st^{1/3}}\big)\cdot \mathbb{E}[\theta(\bar{s})]
\end{align*}
which after rearranging yields
$1- \exp\big(- e^{-b_t s t^{1/3}}\big)\mathbb{P}(\h_t\leq s) \geq 1- \mathbb{E}\left[\theta(\bar{s})\right]$.
Finally, applying \eqref{eq:LowrBound} to the right hand side of the above display shows the lower bound.
\end{proof}


\subsubsection{Proof of Proposition~\ref{thm:MainTheorem}}\label{NWLaplaceTail}
\begin{proof}[Proof of \eqref{eq:UpBound}]
  Define $\bar{s}:=(1+Kb_t)s$. Define $\mathbf{A}=\big\{\mathbf{a}_1\leq  (1-\tilde{K}b_t)\bar{s}\big\}$ for some $\tilde{K}\in [0,K_3]$ and note the following lower bound
   \begin{align}\label{eq:TrivLowerBound}
   \mathbb{E}_{\mathrm{Airy}}\big[\prod_{k=1}^{\infty} \mathcal{I}_{\bar{s}}(\mathbf{a}_k)\big] &\geq \mathbb{E}_{\mathrm{Airy}}\big[\prod_{k=1}^{\infty} \mathcal{I}_{\bar{s}}(\mathbf{a}_k)\mathbf{1}(\mathbf{A})\big].
   \end{align}
   We show a lower bound to the right hand side of the above display. We set $k_0:= \lfloor \frac{2}{3\pi}\bar{s}^{\frac{9}{4}+2b_t}\rfloor$. By the inequality $\mathcal{J}_{\bar{s}}(\mathbf{a}_k)\leq \exp(-ct^{\frac{1}{3}}\bar{s}b_t)$ which follows on the event $\mathbf{A}$, we observe that  
   \begin{align}
   \prod_{k=1}^{k_0} \mathcal{I}_s(\mathbf{a}_k)\mathbf{1}(\mathbf{A})&= \exp\Big(- \sum_{k=1}^{k_0}\mathcal{J}_{\bar{s}}(\mathbf{a}_k)\Big)\mathbf{1}(\mathbf{A})\geq \exp\Big(-\frac{2}{3\pi}\bar{s}^{\frac{9}{4}+2b_t}e^{-\tilde{K}\bar{s}b_tt^{\frac{1}{3}}}\Big). \label{eq:1stBound}
   \end{align}
   We now focus on to bound $\prod_{k>k_0} \mathcal{I}_s(\mathbf{a}_k)$ from below on the event $\mathbf{A}$. By the result of \cite[Proposition~4.5]{CG18a}, for any $\epsilon,\delta\in (0,1)$ the probability space corresponding to the Airy point process can be augmented so that there exists a random variable $C^{\mathrm{Ai}}_{\epsilon}$  satisfying
   \begin{equation*}
(1+\epsilon)\lambda_{k} - C^{\mathrm{Ai}}_{\epsilon}\leq \mathbf{a}_k\leq  (1-\epsilon)\lambda_k+ C^{\mathrm{Ai}}_{\epsilon}\quad \text{for all }k\geq 1\quad \text{ and } \quad \mathbb{P}(C^{\mathrm{Ai}}_{\epsilon}\geq s)\leq e^{- s^{1-\delta}}
\end{equation*}
 for all $s\geq s_0$ where $s_0=s_0(\epsilon,\delta)$ is a constant. Here, $\lambda_k$ is the $k$-th zero of the Airy function (see \cite[Proposition~4.6]{CG18a}) and we fix some $\delta\in(0, \epsilon)$. Define $\phi(s) := s^{\frac{3+8\epsilon/3}{2(1-\delta)^2}}$ and observe that 
\begin{align}\label{eq:ResBound}
\prod_{k>k_0} \mathcal{I}_s(\mathbf{a}_k) \geq \prod_{k>k_0} \mathcal{I}_s(\mathbf{a}_k)\mathbf{1}(C^{\mathrm{Ai}}_{\epsilon}\leq \phi(s)) &\geq \exp\Big(-\sum_{k>k_0} \mathcal{J}_s\big((1-\epsilon)\lambda_k+ \phi(s)\big)\Big).
\end{align}
 Using tail probability of $C^{\mathrm{Ai}}_{\epsilon}$ \cite[Proposition~4.6]{CG18a}, we have
$\mathbb{P}(C^{\mathrm{Ai}}_{\epsilon}\leq \phi(s))\geq 1- \exp\big(-s^{\frac{3}{2}+\frac{4}{3}\epsilon}\big)$.
we now claim and prove that 
\begin{align}\label{eq:JBound}
\mathcal{J}_{s}\Big((1-\epsilon)\lambda_k+ \phi(s)\Big)\leq  e^{t^{1/3}\big(-s -(1-\epsilon)(3\pi k/2)^{2/3}+\phi(s)\big)} \leq e^{t^{1/3}\big(-s-(1-\epsilon)(k-k_0)^{2/3}\big)}.
\end{align}
To see this note that for all $k\geq k_0$,
\begin{align*}
\lambda_k \leq -\Big(\frac{3\pi k}{2}\Big)^{\frac{3}{2}} \quad \text{and, } \quad (1-\epsilon)(\frac{3\pi k}{2}\big)^{\frac{3}{2}} - \phi(s) \geq (1-\epsilon)\Big(\frac{3\pi}{2} (k-k_0)\Big)^{\frac{1}{3}}.
\end{align*}
The first and second inequalities are consequence of \cite[Proposition~4.6]{CG18a} and \cite[Lemma~5.6]{CG18a} respectively.
Summing both sides of \eqref{eq:JBound} over $k>k_0$ in \eqref{eq:JBound}, approximating the sum by the corresponding integral, and evaluating shows 
\begin{equation}\label{eq:ResBoundSep}
\sum_{k>k_0} \mathcal{J}_{\bar{s}}((1-\epsilon)\lambda_k+ \phi(s)) \leq Ct^{-\frac{1}{3}}\exp(-\bar{s}t^{\frac{1}{3}}).
\end{equation}
for some constant $C>0$. Now, we substitute \eqref{eq:ResBoundSep} into the r.h.s. of \eqref{eq:ResBound} to write 
\begin{align*}
\prod_{k>k_0} \mathcal{I}_s(\mathbf{a}_k)\mathbf{1}(C^{\mathrm{Ai}}_{\epsilon}\leq \phi(\bar{s}))\geq \exp\left(-\frac{C}{t^{\frac{1}{3}}} \exp(- \bar{s}t^{\frac{1}{3}})\right).
\end{align*}
Applying \eqref{eq:1stBound} in combination with the above inequality shows 
\begin{equation}
\text{l.h.s. of \eqref{eq:TrivLowerBound}}\geq \exp\Big(-\frac{2}{3\pi}\bar{s}^{\frac{9}{4}+2b_t}e^{-\tilde{K}b_t \bar{s}t^{\frac{1}{3}}}-Ct^{-\frac{1}{3}}e^{-\bar{s}t^{\frac{1}{3}}}\Big) \mathbb{P}\big(C^{\mathrm{Ai}}_{\epsilon}\leq \phi(\bar{s}), \mathbf{A}\big).  \label{eq:LowB1}
\end{equation}
First we note that 
$$\exp\Big(-\frac{2}{3\pi}s^{\frac{9}{4}+2b_t}e^{-\tilde{K}b_t \bar{s}t^{\frac{1}{3}}}-Ct^{-\frac{1}{3}}e^{-\bar{s}t^{\frac{1}{3}}}\Big) = 1+o(1)$$
as $t\to \infty$. Using the tail bound of $C^{\mathrm{Ai}}_{\epsilon}\leq \phi(\bar{s})$, we may now write 
\begin{equation}\label{eq:LowB2}
\mathbb{P}\big(C^{\mathrm{Ai}}_{\epsilon}\leq \phi(\bar{s}),\mathbf{A}\big)\geq 1 - \mathbb{P}(C^{\mathrm{Ai}}_{\epsilon}\geq \phi(\bar{s})) - \mathbb{P}(\neg\mathbf{A})\geq 1- e^{-\bar{s}^{\frac{3}{2}+\frac{4}{3}\epsilon}} - (1+o(1))\frac{1}{16\pi s^{3/2}}e^{-\frac{4}{3}s^{\frac{3}{2}}}
\end{equation}
for all large $t$. The second inequality above used $$\mathbb{P}(\neg\mathbf{A})=\mathbb{P}(\mathbf{a}_1\geq (1-\tilde{K}b_t)s)\leq (1+o(1))\frac{1}{16\pi s^{3/2}}\exp(-\frac{4}{3}s^{\frac{3}{2}})$$ which holds when $t$ is sufficiently large (see \cite[Theorem~1]{BN12}). Substituting \eqref{eq:LowB2} into the right hand side of \eqref{eq:LowB1} yields \eqref{eq:UpBound}.
\end{proof}

\begin{proof}[Proof of \eqref{eq:LowrBound}]
Now we show an upper bound on $\mathbb{E}\big[\prod_{k=1}^{\infty} \mathcal{I}_{s+Kb_t}(\mathbf{a}_k)\big]$. Recall that $\bar{s}= (1+Kb_t)s$ and define $\mathbf{A}=\big\{\mathbf{a}_1 \leq (1+\tilde{K}b_t)\bar{s}\big\}$ for some $\tilde{K}\in [0,K_3]$. We split $\mathbb{E}\big[\prod_{k=1}^{\infty} \mathcal{I}_{\bar{s}}(\mathbf{a}_k)\big]$ into two different parts shown as follows  
\begin{align}\label{eq:LowTail1stStep}
\mathbb{E}\Big[\prod_{k=1}^{\infty} \mathcal{I}_{\bar{s}}(\mathbf{a}_k)\Big]\leq \mathbb{E}\Big[\prod_{k=1}^{\infty} \mathcal{I}_{\bar{s}}(\mathbf{a}_k)\mathbf{1}(\mathbf{A})\Big] + \mathbb{P}(\neg \mathbf{A})\cdot \exp(- \tilde{K}b_t \bar{s} t^{\frac{1}{3}}).
\end{align}
%
Let us denote $\chi^{\mathrm{Ai}}(s):= \#\{\mathbf{a}_i\geq s\}$. Fix $\epsilon\in (0,1)$, $c\in(0,\tfrac{2}{3\pi})$ and define
\begin{align*}
\mathbf{B}: = \Big\{ \chi^{\mathrm{Ai}}(-\epsilon \bar{s})- \mathbb{E}\big[\chi^{\mathrm{Ai}}(-\epsilon \bar{s})\big]\geq - c(\epsilon \bar{s})^{\frac{3}{2}}\Big\}
\end{align*}
We split the first term on the r.h.s. of \eqref{eq:LowTail1stStep} as follows
 \begin{align*}
 \mathbb{E}\Big[\prod_{k=1}^{\infty} \mathcal{I}_{\bar{s}}(\mathbf{a}_k)\mathbf{1}(\mathbf{A})\Big] &\leq \mathbb{E}\Big[\prod_{k=1}^{\infty}\mathcal{I}_{\bar{s}}(\mathbf{a}_k)\mathbf{1}\big(\mathbf{B}\cap \mathbf{A}\big)\Big]+  \mathbb{E}\Big[\mathbf{1}\big((\neg\mathbf{B}) \cap \mathbf{A})\Big]. 
\end{align*}
We now bound each term on the right hand side of the above display. Note that 
$$\prod_{k=1}^{\infty}\mathcal{I}_{\bar{s}}(\mathbf{a}_k) \mathbf{1}(\mathbf{B})\leq \exp\Big(-\Big(\frac{2}{3\pi}-c\Big)(\epsilon \bar{s})^{\frac{3}{2}}e^{-(1+\epsilon)\bar{s} t^{\frac{1}{3}}}\Big)$$ holds on the event $\mathbf{B}$. As a consequence, we get 
\begin{align}\label{eq:2ndSplitBd}
\mathbb{E}\Big[\prod_{k=1}^{\infty}\mathcal{I}_{\bar{s}}(\mathbf{a}_k)\mathbf{1}\big(\mathbf{B}\cap\mathbf{A}\big)\Big]\leq \exp\Big(-\Big(\frac{2}{3\pi}-c\Big)(\epsilon \bar{s})^{\frac{3}{2}}e^{-(1+\epsilon)\bar{s}t^{\frac{1}{3}}}\Big)\cdot\mathbb{P}(\mathbf{A}).
\end{align}
We may bound the r.h.s. of \eqref{eq:2ndSplitBd} by $(1-\exp(-\zeta \bar{s} t^{1/3}))\mathbb{P}(\mathbf{A})$ for some $\zeta>0$ as $t$ gets large. On the other hand, we note that there exists $t_{\delta}>0$ such that $\mathbb{P}(\neg\mathbf{B})\leq e^{- c(\epsilon \bar{s})^{3-\delta}}$ for all $t>t_{\delta}$. Substituting these bounds into the r.h.s. of \eqref{eq:LowTail1stStep} shows 
\begin{align}
1-\mathbb{E}\big[\prod_{k=1}^{\infty}\mathcal{I}_{\bar{s}}(\mathbf{a}_k)\big]&\geq \Pr(\neg\mathbf{A})+ \mathbb{P}(\mathbf{A})(e^{- \zeta \bar{s} t^{\frac{1}{3}}}- e^{-\tilde{K}b_t\bar{s}t^{1/3}})- e^{- c(\zeta \bar{s})^{3-\delta}}.\qquad  \label{eq:LowrFnStep}
\end{align}
Due to the following inequality (thanks to \cite[Theorem~1]{BN12}) $$\mathbb{P}(\neg\mathbf{A})\geq (1+o(1))\frac{1}{16\pi s^{3/2}}\exp\big(-\frac{4}{3}s^{\frac{3}{2}}\big)$$  and since $\exp(-\zeta st^{1/3}),\exp(-b_ts t^{1/3})$ and $\exp(-c(\epsilon s)^{3-\delta})$ can be represented as $o(1)\exp(-4s^{3/2}/3)$ as $t$ grows large, the r.h.s. of \eqref{eq:LowrFnStep} is lower bounded by $(1+o(1))(16\pi s^{3/2})^{-1}\exp(-4s^{3/2}/3)$. This completes the proof of \eqref{eq:LowrBound} and hence also of Proposition~\ref{thm:MainTheorem}.
\end{proof}

\bibliographystyle{alphaabbr}		
\bibliography{frt}
\end{document}